\documentclass[12pt]{amsart}
\usepackage{amssymb,mathrsfs,amsmath}

\newtheorem{thm}{Theorem}[section]
\newtheorem{lemma}[thm]{Lemma}
\newtheorem{cor}[thm]{Corollary}
\theoremstyle{remark}
\newtheorem{rk}[thm]{Remark}

\newcommand {\SN} {{\mathbb N}}
\newcommand {\SR} {{\mathbb R}}

\newcommand {\SZ} {{\mathbb Z}}

\newcommand{\bsp}{\begin{split}}
\newcommand{\esp}{\end{split}}
\newcommand{\be}{\begin{equation}}
\newcommand{\ee}{\end{equation}}
\newcommand{\bes}{\begin{equation*}}
\newcommand{\ees}{\end{equation*}}
\newcommand{\bv}\boldsymbol{}
\newcommand{\hxy}{H^{(k+1)}(x,\bv y,2\bv y)}
\newcommand{\prob}{{\bf Prob}}
\DeclareMathOperator{\vol}{Vol}

\numberwithin{equation}{section}

\begin{document}

\title[Generalized multiplication tables]{On the number of integers in a generalized multiplication table}

\author{Dimitris Koukoulopoulos}
\address{Centre de recherches math\'ematiques\\
Universit\'e de Montr\'eal\\
CP 6128 succ. Centre-Ville\\
Montr\'eal, QC H3C 3J7\\
Canada}
\email{{\tt koukoulo@CRM.UMontreal.CA}}
\thanks{The author wishes to acknowledge support in the form of a Research Assistantship from the National Science Foundation grants DMS 05-55367, DMS 08-38434 ``EMSW21-MCTP:
Research Experience for Graduate Students'' and DMS 09-01339.}
\date{\today}

\maketitle

\begin{abstract} Motivated by the Erd\H os multiplication table problem we study the following question: Given numbers $N_1,\dots,N_{k+1}$, how many distinct products of the form $n_1\cdots n_{k+1}$ with $1\le n_i\le N_i$ for $i\in\{1,\dots,k+1\}$ are there? Call $A_{k+1}(N_1,\dots,N_{k+1})$ the quantity in question. Ford established the order of magnitude of $A_2(N_1,N_2)$ and the author the one of $A_{k+1}(N,\dots,N)$ for all $k\ge2$. In the present paper we generalize these results by establishing the order of magnitude of $A_{k+1}(N_1,\dots,N_{k+1})$ for arbitrary choices of $N_1,\dots,N_{k+1}$ when $2\le k\le 5$. Moreover, we obtain a partial answer to our question when $k\ge6$. Lastly, we develop a heuristic argument which explains why the limitation of our method is $k=5$ in general and we suggest ways of improving the results of this paper.
\end{abstract}

\tableofcontents

\section{Introduction}\label{intro}
\subsection{The Erd\H os multiplication table problem and its generalizations}
In 1955 Erd\H os posed the so-called {\it multiplication table problem} \cite{erd1}: Given a large number $N$,
how many integers can be written as a product $ab$ with $a\le N$ and
$b\le N$? Erd\H os gave the first estimates of this quantity
\cite{erd1,erd2}, which were subsequently sharpened by Tenenbaum
\cite{ten}. The problem of establishing the order of magnitude of
the size of the $N\times N$ multiplication table was completely
solved by Ford in \cite{kf1,kf2}, where he showed that
$$A_2(N):=|\{ab:a\le N\;{\rm and}\;b\le N\}|\asymp\frac{N^2}{(\log N)^{Q(1/\log2)}(\log\log
N)^{3/2}}\quad(N\ge3),$$ where $$\label{Q}Q(u):=\int_1^u\log t\,dt=u\log
u-u+1\quad(u>0).$$ More generally, we may ask the same question
about higher dimensional analogues of the multiplication table
problem, that is to say, we may ask for estimates of
$$
A_{k+1}(N):=|\{n_1\cdots n_{k+1}:n_i\le N\;(1\le i\le k+1)\}|.
$$
In \cite{dk}\;the author determined the order of $A_{k+1}(N)$ for
every fixed $k\ge2$: we have that
$$A_{k+1}(N)\asymp_k\frac{N^{k+1}}{(\log
N)^{Q(k/\log(k+1))}(\log\log N)^{3/2}}\quad(N\ge 3).$$

\medskip

In the present paper we broaden our scope and study the number of
integers that appear in a $(k+1)$-dimensional multiplication table
when the side lengths of the table are different. More precisely,
given numbers $N_1,\dots,N_{k+1}$, we seek uniform bounds on
$$\label{Ak}A_{k+1}(N_1,\dots,N_{k+1}):=|\{n_1\cdots n_{k+1}:n_i\le N_i\;(1\le
i\le k+1)\}|.$$ Instead of studying $A_{k+1}(N_1,\dots,N_{k+1})$
directly, we focus on a closely related function: given $x\ge1$,
$\bv y=(y_1,\dots,y_k)\in\SR^k$ and $\bv z=(z_1,\dots,z_k)\in\SR^k$,
define 
$$
\label{Hk}H^{(k+1)}(x,\bv y,\bv z)=|\{n\le x:\exists d_1\cdots d_k|n~\text{such that}~y_i<d_i\le z_i~(1\le i\le k)\}|.
$$ 
We then have the following theorem, which establishes the expected quantitative relation
between $A_{k+1}(N_1,\dots,N_{k+1})$ and $H^{(k+1)}(x,\bv y,\bv z)$;
its proof will be given in Subsection \ref{thm1pf}.

\begin{thm}\label{thm1}Let $k\ge1$ and $3\le N_1\le N_2\le\cdots\le
N_{k+1}$. Then
$$A_{k+1}(N_1,\dots,N_{k+1})\asymp_kH^{(k+1)}\left(N_1\cdots
N_{k+1},\left(\frac{N_1}2,\dots,\frac{N_k}2\right),(N_1,\dots,N_k)\right).$$
\end{thm}

In light of the above theorem, it suffices to restrict ourselves to
the study of $\hxy$, which is slightly easier technically. What is
more, bounds on  $\hxy$ have applications beyond the multiplication
table problem (for example, see~\cite{kf2} for several such
applications when $k=1$). Before we state the results of this paper,
we summarize some already known estimates on $H^{(k+1)}(x,\bv y,2\bv
y)$ in the theorem below. Briefly, this theorem gives the order of
magnitude of $\hxy$ when the numbers $\log y_1,\cdots,\log y_k$ are
roughly of the same size. In particular, it establishes the order of
magnitude of $H^{(2)}(x,y,2y)$ for all $2\le y\le \sqrt{x}$. For a proof of it
we refer the reader to \cite{kf1,kf2}\;and \cite{dk}; the first two
papers handle the case $k=1$ and the latter the case $k\ge2$.

\begin{thm}[Ford \cite{kf1,kf2}, Koukoulopoulos \cite{dk}]\label{kfdk} Let $k\ge1$, $c\ge1$ and $\delta>0$. Consider numbers
$x\ge3$ and $3\le y_1\le\cdots\le y_k\le y_1^c$ with
$2^{k+1}y_1\cdots y_k\le x/y_1^\delta$. Then
$$\hxy\asymp_{k,c,\delta}\frac x{(\log y_1)^{Q(k/\log(k+1))}(\log\log y_1)^{3/2}}.$$
\end{thm}

In this paper we extend Theorem \ref{kfdk}\;to a broader range of
the parameters $y_1,\dots,y_k$. In particular, when $2\le k\le 5$ we establish the order of $\hxy$ for any choice of the parameters $y_1,\dots,y_k$. In order to state our results we introduce some notation. Given numbers $3=y_0\le y_1\le\cdots\le y_k$, set
$$\label{elli}\ell_i=\log\frac{3\log y_i}{\log y_{i-1}}\quad(1\le i\le k).$$
Also, let $\label{i1}i_1$ be the smallest element of $\{1,\dots,k\}$ such that
$$\ell_{i_1}=\max\{\ell_i:1\le i\le k\}$$ and define $\beta=\beta(k;\bv
y)$ by
$$\label{beta}\beta=\min\left\{1,\frac{(1+\ell_1+\cdots+\ell_{i_1-1})(1+\ell_{i_1+1}+\cdots+\ell_k)}{\ell_{i_1}}\right\}.$$
Lastly, define $\label{alpha}\alpha=\alpha(k;\bv y)$ implicitly, via the equation
$$
\sum_{i=1}^k(k-i+2)^\alpha\log(k-i+2)\ell_i=\sum_{i=1}^k(k-i+1)\ell_i.
$$ 
Note that 
$$\alpha\ge\min_{1\le i\le k}\frac1{\log(k-i+2)}\log\left(\frac{k-i+1}{\log(k-i+2)}\right)=\frac1{\log2}\log\left(\frac1{\log2}\right)=0.528766373\dots
$$ 
as well as
$$
\alpha \le \max_{1\le i\le k}\frac1{\log(k-i+2)} 
		\log\left(\frac{k-i+1}{\log(k-i+2)}\right)
	= \frac1{\log(k+1)}\log\left(\frac k{\log(k+1)}\right) < 1
$$ 
(here we used Lemma~\ref{hl1}, which will be stated and proven in Section~\ref{heur}).

\begin{thm}\label{cor2} Let $k\in\{2,3,4,5\}$, $x\ge3$ and $3\le y_1\le\cdots\le y_k$ be such
that $2^ky_1\cdots y_k\le x/y_k$. Then 
$$
\frac\hxy x
	\asymp  \frac\beta{\sqrt{\log\log y_k}}
		\prod_{i=1}^k\left(\frac{\log y_i}{\log y_{i-1}}\right)^{-Q((k-i+2)^\alpha)}.
$$
\end{thm}

As we shall see later, the hypothesis that $k\in\{2,3,4,5\}$ in the above theorem is necessary: when $k\ge6$ there are choices of the parameters $y_1,\dots,y_k$ for which $\hxy$ has genuinely smaller order than what Theorem~\ref{cor2} predicts. However, if $\log y_k$ is not much bigger than $\log y_1$, then the conlcusion of Theorem~\ref{cor2} is valid. More precisely, we have the following result, which extends Theorem~\ref{kfdk}.

\begin{thm}\label{cor3} Let $k\ge 6$, $x\ge3$ and $3\le y_1\le\cdots\le y_k$
such that $2^ky_1\cdots y_k\le x/y_k$ and $\log y_k\le(\log
y_1)^{1+\delta}$ for a sufficiently small positive $\delta=\delta(k)$. Then
$$\frac\hxy x\asymp_k\frac{\log\frac{3\log y_k}{\log y_1}}{(\log\log y_1)^{3/2}}\prod_{i=1}^k\left(\frac{\log y_i}{\log
y_{i-1}}\right)^{-Q((k-i+2)^\alpha)}.$$
\end{thm}.

\subsection{Main results} Both Theorems~\ref{cor2} and~\ref{cor3} are consequences of a more general estimate on $\hxy$, which is the main result of this paper.

\begin{thm}\label{thm2}Let $k\ge2$, $x\ge3$ and $3\le y_1\le\cdots\le y_k$ be such that $2^ky_1\cdots y_k\le x/y_k$. Then
$$\frac\hxy x\ll_k\frac\beta{\sqrt{\log\log y_k}}\prod_{i=1}^k\left(\frac{\log y_i}{\log
y_{i-1}}\right)^{-Q((k-i+2)^\alpha)}.$$ If we also assume that
\be\label{e0}\alpha\ge1+\epsilon-\frac1{\log(k+1)}\log\left(\frac{(k+1)\log(k+1)-2\log2}{k-1}\right)\ee
for some fixed $\epsilon>0$, then
$$\frac\hxy x\asymp_{k,\epsilon}\frac\beta{\sqrt{\log\log
y_k}}\prod_{i=1}^k\left(\frac{\log y_i}{\log
y_{i-1}}\right)^{-Q((k-i+2)^\alpha)}.$$
\end{thm}

Condition \eqref{e0}\;is essentially optimal in the sense that for
every fixed $\gamma$ that satisfies
\be\label{e00}\frac1{\log2}\log\left(\frac1{\log2}\right)<\gamma<1-\frac1{\log(k+1)}\log\left(\frac{(k+1)\log(k+1)-2\log2}{k-1}\right)\ee
there is a choice of $y_1\le\cdots\le y_k$ such that
$\alpha=\alpha(k;\bv y)=\gamma$ and for which the order of $\hxy$ is genuinely
smaller than the one stated above. We shall discuss this further in
the next section using some heuristic arguments. In relation to our comments following the statement of Theorem~\ref{cor2}, note that
the smallest value of $k$ for which the range~\eqref{e00} is non-empty is $k=6$.

\medskip

Despite its optimality, condition~\eqref{e0} is not very easy to work with due to the implicit definition of $\alpha$.
Below we state a weaker version of Theorem~\ref{thm2}, whose hypotheses are easier to verify.

\begin{cor}\label{cor1} Let $k\ge2$, $h\in\{1,\dots,k\}$, $x\ge3$ and $3\le y_1\le\cdots\le y_k$ such that $2^ky_1\cdots y_k\le x/y_k$,
$$
\frac1{\log(k-h+2)}\log\left(\frac{k-h+1}{\log(k-h+2)}\right)>1-\frac1{\log(k+1)}\log\left(\frac{(k+1)\log(k+1)-2\log2}{k-1}\right).
$$
and $\log y_k\le(\log y_h)^{1+\delta}$ for a sufficiently small positive $\delta=\delta(k)$. Then 
$$
\frac\hxy
x\asymp_k\frac\beta{\sqrt{\log\log y_k}}\prod_{i=1}^k\left(\frac{\log y_i}{\log y_{i-1}}\right)^{-Q((k-i+2)^\alpha)}.
$$
\end{cor}

\begin{proof} We consider for the moment $\delta$ to be a free parameter.
Since $\log y_k\le(\log y_h)^{1+\delta}$, we have that
$$\sum_{i=1}^h(k-i+2)^\alpha\log(k-i+2)\ell_i=\sum_{i=1}^h(1+O_k(\delta))(k-i+1)\ell_i.$$
Therefore \be\begin{split}\alpha&\ge\min_{1\le i\le
h}\frac1{\log(k-i+2)}\log\left(\frac{k-i+1}{\log(k-i+2)}\right)-O_k(\delta)\nonumber\\
&=\frac1{\log(k-h+2)}\log\left(\frac{k-h+1}{\log(k-h+2)}\right)-O_k(\delta),\end{split}\ee
by Lemma \ref{hl1}. So if we choose $\delta$ small enough, then
\eqref{e0}\;holds and hence the desired result follows.
\end{proof}

Applying the above corollary with $h=k\le5$ gives us Theorem~\ref{cor2} immediately. Similarly, Theorem~\ref{cor3} follows by Corollary~\ref{cor1} with $h=1$; we only need to check that
\be\label{e000}\frac1{\log(k+1)}\log\left(\frac k{\log(k+1)}\right)>1-\frac1{\log(k+1)}\log\left(\frac{(k+1)\log(k+1)-2\log2}{k-1}\right)\ee
or, equivalently, that $$(k+1)\log(k+1)>k\log 4$$
for $k\ge2$, which is indeed true.

\medskip

The main tool we shall use in order to prove Theorems
\ref{thm1}\;and \ref{thm2}\;is a result that reduces the counting in
$\hxy$, which contains information about the local distribution of
factorizations of integers, to the estimation of a certain sum which
contains information about the global distribution of factorizations
of integers. More precisely, for $\bv a=(a_1,\dots,a_k)\in\SN^k$ define
$$\label{cLk}\mathcal{L}^{(k+1)}(\bv a)=\bigcup_{\substack{d_1\cdots
d_i|a_1\cdots a_i\\1\le i\le k}}\left[\log(d_1/2),\log
d_1\right)\times\cdots\times\left[\log(d_k/2),\log d_k\right),$$ and
$$\label{Lk}L^{(k+1)}(\bv a)=\vol(\mathcal{L}^{(k+1)}(\bv a)),$$ where ``$\vol$"
denotes the $k$-dimensional Lebesgue measure. Also, for $1\le y<z$
set $$\label{Pyz}\mathcal{P}_*(y,z)=\{n\in\SN:\mu^2(n)=1,p|n\Rightarrow y<p\le z\}$$ and
for $\bv t=(t_1,\dots,t_k)$ with $t_k\ge t_{k-1}\ge\cdots\ge t_1\ge1=:t_0$ set
$$\label{Pvt}\mathcal{P}_*^k(\bv t)=\{(a_1,\dots,a_k)\in\SN^k:a_i\in\mathcal{P}_*(t_{i-1},t_i)\;(1\le i\le k)\}.$$
Then we have the following estimate.

\begin{thm}\label{thm3}Let $k\ge1$, $x\ge1$ and $3\le y_1\le\cdots\le y_k$ with $2^ky_1\cdots y_k\le x/y_k$. Then
$$\frac \hxy x\asymp_k\prod_{i=1}^k\left(\frac{\log y_i}{\log y_{i-1}}\right)^{-(k-i+2)}\sum_{\bv a\in\mathcal{P}_*^k(\bv
y)}\frac{L^{(k+1)}(\bv a)}{a_1\cdots a_k}.$$
\end{thm}

When $k=1$, the above theorem is an immediate consequence of the results and the methods in~\cite{kf1}: see Lemmas 2.1 and 3.2 there. As an immediate consequence of Theorem~\ref{thm3}, we have the following result.

\begin{cor}\label{cor4} Let $k\ge1$ be an integer and for $i\in\{1,2\}$ consider $x_i\ge1$ and $\bv y_i=(y_{i,1},\dots,y_{i,k})\in[1,+\infty)^k$.
Assume that $2^ky_{i,1}\cdots y_{i,k}\le x/y_{i,k}$ for $i\in\{1,2\}$ and that there exist constants $c$ and $C$ such that
$y_{1,j}^c\le y_{2,j}\le y_{1,j}^C$ for all $j\in\{1,\dots,k\}$. Then
$$\frac{H^{(k+1)}(x_1,\bv y_1,2\bv y_1)}{x_1}\asymp_{k,c,C}\frac{H^{(k+1)}(x_2,\bv y_2,2\bv y_2)}{x_2}.$$
\end{cor}

\begin{proof} The result follows by Theorem \ref{thm3}, Lemma~\ref{ub1l0}(a) and the standard estimate
\be\label{tauineq}
\sum_{ a\in \mathcal{P}_*(t,t^B) } \frac{\tau_m(a)} {a} 
= \prod_{ t<p\le t^B } \left( 1+ \frac{m}{p} \right) 
\asymp_{m,B} 1 \quad (t\ge 1),\ee 
where 
$$
\label{tau}\tau_m(a)=\sum_{d_1\cdots d_{m-1}|a}1=\sum_{d_1\cdots d_m=a}1\quad(m\in\SN,~a\in\SN).
$$
\end{proof}

When $k=1$, a stronger version of the above corollary is known to be
true: see Corollary 1 in \cite{kf2}.

\medskip

\subsection{Outline of the paper} The paper is organized in the following way: In Section~\ref{heur} we demonstrate a heuristic argument in support of Theorem
\ref{thm2}\;and the optimality of condition \eqref{e0}. The first three subsections of Section~\ref{loc_glob} are devoted to
establishing Theorem \ref{thm3}, whereas in the last one we prove Theorem~\ref{thm1}. In Section~\ref{poisson}
we develop some estimates related to the probability that a multidimensional Poisson random variable lies close to a hyperplane. Such estimates play a
crucial role in the proof of Theorem \ref{thm2}. Also, in combination with the heuristic arguments of Section \ref{heur}, they
explain how the parameter $\alpha$ makes its appearance in the statements of our results. In Section~\ref{ub} we give the proof of
the upper bound in Theorem \ref{thm2}. The main steps of the proof are described in Subsection~\ref{ub_outline} and proven in
Subsection \ref{ub_proof}. The proof of the lower bound in Theorem~\ref{thm2} is divided in three sections: in Section~\ref{lb_outline} we describe the main steps of our argument.
The first major such step is then carried out in Section~\ref{lb_proof1}. Finally, Section~\ref{lb_proof2}
contains the last piece of our argument and completes the proof of Theorem~\ref{thm2}.

\medskip

\subsection{Notation} We make use of some standard notation. For $n\in\SN$ we use $P^+(n)$ and $P^-(n)$ to denote the largest and smallest prime factor of $n$, respectively, with the notational conventions that $P^+(1)=1$ and $P^-(1)=+\infty$. Also, $\omega(n)$ denotes the number of distinct prime
factors of $n$. Constants implied by $\ll$, $\gg$ and $\asymp$ are absolute unless otherwise specified, e.g. by a subscript. Also, we use the letters $c$ and $C$ to denote constants, not necessarily the same ones in every place, possibly depending on certain parameters that will be specified by subscripts and other means. Also, bold letters always denote vectors whose coordinates are indexed by the same letter with subscripts, e.g. $\bv a=(a_1,\dots,a_k)$ and $\bv\xi=(\xi_1,\dots,\xi_r)$. The dimension of the vectors will not be explicitly specified if it is clear by the context. Finally, we give a table of some basic non-standard notation that we will be using with references to page numbers for its definition.

\begin{table}[htbp]
\begin{center}
\begin{tabular}{llllll}
\hline 
Symbol & Page & Symbol & Page & Symbol & Page \\
\hline
$Q(u)$ &\pageref{Q} & $\alpha$ & \pageref{alpha} & $\alpha_i$ & \pageref{alphai}  \\
$\beta$ & \pageref{beta} & $i_1$ & \pageref{i1} & $i_0$ & \pageref{i0} \\
$\ell_i$ & \pageref{elli} & $v_i$ & \pageref{vi} & $\Delta_r$ & \pageref{Dr} \\
$\rho_m$ & \pageref{rho} & $\bv e_k,\ e_{k,i}$ & \pageref{ek} & $\tau_m(a)$ & \pageref{tau} \\
$\tau_{k+1}(\bv a)$ & \pageref{tauv} & $\tau_{k+1}(a,\bv y,\bv z)$ & \pageref{tauvv} &
$\mathcal{P}_*(y,z)$ &  \pageref{Pyz}  \\
$\mathcal{P}_*^k(\bv t)$ & \pageref{Pvt} &
$H^{(k+1)}(x,\bv y,\bv z)$ & \pageref{Hk} &
$A_{k+1}(N_1,\dots,N_{k+1})$ & \pageref{Ak} \\
$\mathcal{L}^{(k+1)}(\bv a)$ & \pageref{cLk} &
$L^{(k+1)}(\bv a)$ & \pageref{Lk} &
$S^{(k+1)}(\bv t)$ & \pageref{Sk} \\
\hline
\end{tabular}
\end{center}
\end{table}

\medskip

{\bf Acknowledgement.} I would like to thank Kevin Ford for many valuable suggestions as well as for discussions that led to an earlier version of Lemma~\ref{conv_ineq}.


\section{Heuristic arguments}\label{heur}

In this section we develop a heuristic argument in support of Theorem \ref{thm2}. Its main part is given in Subsection 2.1 and it is a generalization of an argument developed by Ford in~\cite{kf1} for the case $k=1$ and subsequently by the author in~\cite{dk} for the case $k\ge2$. In Subsection~\ref{conde0} we introduce some new ideas in order to explain the appearance of condition \eqref{e0} in the statement of Theorem \ref{thm2}.

\medskip

Before we delve into the details of this argument, we introduce some additional notation and state two elementary but basic results we will be
using throughout the entire paper. For $\bv a=(a_1,\dots,a_k)\in\SN^k$ and $\bv y,\bv z\in\SR^k$ let $$\label{tauv}\tau_{k+1}(\bv a)=|\{(d_1,\dots,d_k)\in\SN^k:d_1\cdots
d_i|a_1\cdots a_i\;(1\le i\le k)\}|$$ and $$\label{tauvv}\tau_{k+1}(\bv a,\bv y,\bv z)=|\{(d_1,\dots,d_k)\in\SN^k:d_1\cdots
d_i|a_1\cdots a_i,\;y_i<d_i\le z_i~(1\le i\le k)\}|.$$ Finally, set $$\label{alphai}\alpha_i=\frac1{\log(i+1)}\log\left(\frac i{\log(i+1)}\right)\quad(i\in\SN)$$ and let $\label{i0}i_0$ be the minimum element of $\{1,\dots,k\}$ such that $$|\alpha-\alpha_{k-i_0+1}|=\min\{|\alpha-\alpha_{k-i+1}|:1\le i\le k\}.$$

\begin{lemma}\label{ub1l0} The following assertions hold:
\renewcommand{\labelenumi}{(\alph{enumi})}
\begin{enumerate}
\item For $\bv a\in\SN^k$ we have $$L^{(k+1)}(\bv a)\le\min\left\{\tau_{k+1}(\bv a)(\log2)^k,\prod_{i=1}^k(\log a_1+\cdots+\log a_i+\log2)\right\}.$$
\item If $(a_1\cdots a_k,b_1\cdots b_k)=1$, then $$L^{(k+1)}(a_1b_1,\dots,a_kb_k)\le\tau_{k+1}(\bv a)L^{(k+1)}(\bv b).$$
\end{enumerate}
\end{lemma}

\begin{proof} The proof is similar to the proof of Lemma
3.1 in \cite{kf1}.
\end{proof}

\begin{lemma}\label{hl1}The sequence $\{\alpha_i\}_{i\in\SN}$ is strictly increasing.
\end{lemma}

\begin{proof} The function $$\frac1{\log(x+1)}\log\left(\frac
x{\log(x+1)}\right)$$ is easily seen to be strictly increasing for $x\ge15$. Finally, we check numerically that
$\alpha_1<\alpha_2<\cdots<\alpha_{15}$.
\end{proof}


\subsection{Basic set-up and development of the main argument}\label{main} Our goal is to understand when an integer
$n\le x$ is counted by $\hxy$. We write $n=a_1\cdots a_kb$, where 
$$
a_i=\prod_{\substack{p^\nu \| n \\ 2y_{i-1}<p\le2y_i }} p^\nu \quad(1\le i\le k).
$$
For simplicity, we assume that the numbers $a_1,\dots,a_k$ are square-free and satisfy $\log a_i\asymp\log y_i$ for all
$i\in\{1,\dots,k\}$. Observe that if $\bv d=(d_1,\dots,d_k)\in\SN^k\cap\prod_{i=1}^k(y_i,2y_i]$, then all prime factors of $d_i$ are at most $2y_i$ for all $i\in\{1,\dots,k\}$. Hence $\bv d$ satisfies the relation $d_1\cdots d_k|n$ if, and only if, $d_1\cdots d_i|a_1\cdots
a_i$ for all $i\in\{1,\dots,k\}$. Therefore the integer $n$ is counted by $\hxy$ precisely when $\tau_{k+1}(\bv a,\bv y,2\bv
y)\ge1$. Consider now the set 
$$
D_{k+1}(\bv a)=\{(\log d_1,\dots,\log d_k):d_1\cdots d_i|a_1\cdots a_i\;(1\le i\le k)\}.
$$ 
Assume for the
moment that $D_{k+1}(\bv a)$ is well-distributed in $\prod_{i=1}^k[0,\log(a_1\cdots a_i)]$. Then we should have that
\be\label{he0}\bsp
\tau_{k+1}(\bv a,\bv y,2\bv y)
		=  \left\lvert D_{k+1}(\bv a)\cap\prod_{i=1}^k  (\log y_i,\log y_i+\log2]\right\rvert
	&\approx\tau_{k+1}(\bv a)\frac{(\log 2)^k}{\prod_{i=1}^k\log(a_1\cdots a_i)} \\
	&\asymp_k \frac{\prod_{i=1}^k(k-i+2)^{\omega(a_i)}}{\prod_{i=1}^k\log y_i}.
\end{split}\ee 
The right hand side of~\eqref{he0} is at least 1 when
$$
\sum_{i=1}^k\log(k-i+2)\omega(a_i)\ge\sum_{i=1}^k\log\log y_i+O_k(1)=\sum_{i=1}^k(k-i+1)\ell_i+O_k(1).
$$
Since we expect that
$$
\lvert\{n\le x:\omega(a_i)=r_i\;(1\le i\le k)\}\rvert\approx\frac x{\log
y_k}\frac{\ell_1^{r_1-1}\cdots\ell_k^{r_k-1}}{(r_1-1)!\cdots(r_k-1)!}
$$
(for example, see \cite[Theorem 4, p. 205]{tenbook}), summing the above relation over all vectors $\bv r\in(\SN\cup\{0\})^k$ that satisfy
\be\label{he1}\sum_{i=1}^kr_i\log(k-i+2)\ge\sum_{i=1}^k\ell_i(k-i+1)+O_k(1)\ee
leads to the estimate
\be\label{he2}\hxy\approx\frac x{\log y_k}\sum_{\substack{\bv r\in(\SN\cup\{0\})^k\\\eqref{he1}}}\frac{\ell_1^{r_1-1}\cdots\ell_k^{r_k-1}}{(r_1-1)!\cdots(r_k-1)!}.\ee
Using Stirling's formula and Lagrange multipliers, we find that the maximum of
$\prod_{i=1}^k\ell_i^{r_i-1} / (r_i-1)!$ under condition \eqref{he1}\;occurs
when $r_i\sim(k-i+2)^\alpha\ell_i$ (see Section~\ref{poisson} and, in particular, Remark~\ref{sprk1} and the proof of Lemma~\ref{spl2}(a)). In fact, Lemma~\ref{spl2}(a) implies that
$$\sum_{\substack{\bv r\in(\SN\cup\{0\})^k\\\eqref{he1}}}\frac{\ell_1^{r_1-1}\cdots\ell_k^{r_k-1}}{(r_1-1)!\cdots(r_k-1)!}\asymp_k\frac{\log y_k}{\sqrt{\log\log
y_k}}\prod_{i=1}^k\left(\frac{\log y_i}{\log y_{i-1}}\right)^{-Q((k-i+2)^\alpha)},$$ so that~\eqref{he2} becomes
\be\label{he3}\hxy\approx\frac x{\sqrt{\log\log y_k}}\prod_{i=1}^k\left(\frac{\log y_i}{\log y_{i-1}}\right)^{-Q((k-i+2)^\alpha)}.\ee

\bigskip

If now $\beta\gg1$, then \eqref{he3}\;agrees with the conclusion of Theorem \ref{thm2}. However, if $\beta=o_{y_1\to\infty}(1)$, then relation \eqref{he3}\;overestimates $\hxy$ slightly. The problem lies in our assumption that $D_{k+1}(\bv a)$ is well-distributed. Actually, if $\beta=o_{y_1\to\infty}(1)$, then with high probability the elements of $D_{k+1}(\bv a)$ form large clumps. In order to measure the amount of clustering in $D_{k+1}(\bv a)$, we use the function $L^{(k+1)}(\bv a)$, which we introduced in Section \ref{intro}. We will show that, unless the prime factors of $a_1,\dots,a_k$ satisfy certain constraints, the measure of $L^{(k+1)}(\bv a)$ is small and, as a consequence, the set $D_{k+1}(\bv a)$ cannot be well-distributed.

Fix a vector $\bv r\in\SN^k$ such that
\be\label{he1b}
0\le\sum_{i=1}^kr_i\log(k-i+2)-\sum_{i=1}^k\ell_i(k-i+1)\le\log(k+1)
\ee 
and $r_i\sim(k-i+2)^\alpha\ell_i$ as $\ell_i\to\infty$, for all
$i\in\{1,\dots,k\}$, since most of the contribution to the sum in the right hand side of~\eqref{he2} comes from such vectors. Consider $n$ with $\omega(a_i)=r_i$ for all $i\in\{1,\dots,k\}$ and write $a_i=p_{i,1}\cdots p_{i,r_i}$ with $2y_{i-1}<p_{i,1}<\cdots<p_{i,r_i}\le 2y_i$. Set $$U_i=2\log(k+1)+\sum_{m=1}^{i-1}\ell_m(k-m+1)-\sum_{m=1}^{i-1}r_m\log(k-m+2)\quad(1\le i\le k).$$ Note that
\bes\begin{split}&\ell_m(k-m+1)-r_m\log(k-m+2)\\
&=(k-m+1-(k-m+2)^\alpha\log(k-m+2)+o(1))\ell_m\\
&=\log(k-m+2)((k-m+2)^{\alpha_{k-m+1}}-(k-m+2)^\alpha+o(1))\ell_m.\end{split}\ees
So Lemma~\ref{hl1} and the definition of $i_0$ imply that
\be\label{he101}U_i\asymp_k\begin{cases}1+\ell_1+\cdots+\ell_{i-1}&\text{if}~1\le i\le i_0,\cr
1+\ell_i+\cdots+\ell_k&\text{if}~i_0+1\le i\le k+1,\end{cases}\ee
where in the latter case we used~\eqref{he1b}. Assume that
there are integers $i\in\{1,\dots,k\}$ and $j\in\{1,\dots,r_i\}$ and a large number $C$ such that
$$0\le\log\log p_{i,j}-\log\log y_{i-1}\le\frac{\log(k-i+2)j-U_i-C}{k-i+1}.$$ We claim that this
causes clustering among the elements of $D_{k+1}(\bv a)$. Indeed, if we set $b_m=a_m$ for $1\le m<i$, $b_i=p_{i,1}\cdots p_{i,j}$ and
$b_m=1$ for $i<m\le k$, then a double application of Lemma~\ref{ub1l0} implies that
\be\label{he4}\begin{split}
L^{(k+1)}(\bv a)
	&\le\tau_{k+1}(a_1/b_1,\dots,a_k/b_k)L^{(k+1)}(\bv b)\\
	&\le\left((k-i+2)^{r_i-j}\prod_{m=i+1}^k(k-m+2)^{r_m}\right)\left(\prod_{m=1}^k\log(2b_1\cdots b_m)\right)\\
	&\ll_k(k-i+2)^{-j}\left(\prod_{m=i}^k(k-m+2)^{r_m}\right)
		\left(\prod_{m=1}^{i-1}\log y_m\right) \\ 
	&\qquad \times (\log y_{i-1}+\log(p_{i,1}\cdots p_{i,j}))^{k-i+1} \\
	&\lesssim\left(\prod_{m=i}^k(k-m+2)^{r_m}\right)\left(\prod_{m=1}^{i-1}\log y_m\right) 
		(\log y_{i-1})^{k-i+1}e^{-U_i-C} \\ 
	& \asymp_ke^{-C}\prod_{i=1}^k(k-i+2)^{r_i}.\end{split}\ee
The right hand side of \eqref{he4}\;is much less than $\tau_{k+1}(\bv a)=\prod_{m=1}^k(k-m+2)^{r_m}$ if $C\to\infty$, in which case there
must be many elements of $D_{k+1}(\bv a)$ that are close together. The above argument suggests that we should focus on numbers $n$ for
which \be\label{he5}\log\log p_{i,j}-\log\log y_{i-1}\ge\frac{\log(k-i+2)j-U_i-O(1)}{k-i+1}\quad(1\le i\le
k,\;R_{i-1}<j\le R_i).\ee The number of integers $n$ that satisfy conditions similar to \eqref{he5}\;was studied by Ford in
\cite{kf4}. Using similar considerations, we find that the probability that an integer $n$ satisfies \eqref{he5}\;is about
$$\prod_{i=1}^k\min\left\{1,\frac{U_iU_{i+1}}{r_i}\right\}\asymp_k\min\left\{1,\frac{(1+\ell_1+\cdots+\ell_{i_0-1})(1+\ell_{i_0+1}+\cdots+\ell_k)}{\ell_{i_0}}\right\},$$
by \eqref{he101}. Thus we are led to the refined estimate \be\label{he6}\frac{\hxy}x\approx\frac{\displaystyle\min\left\{1,\frac{(1+\ell_1+\cdots+\ell_{i_0-1})(1+\ell_{i_0+1}+\cdots+\ell_k)}{\ell_{i_0}}\right\}}
{\displaystyle\sqrt{\log\log y_k}\prod_{i=1}^k\left(\frac{\log y_i}{\log y_{i-1}}\right)^{Q((k-i+2)^\alpha)}}.\ee
Finally, we claim that
\be\label{altbeta}\min\left\{1,\frac{(1+\ell_1+\cdots+\ell_{i_0-1})(1+\ell_{i_0+1}+\cdots+\ell_k)}{\ell_{i_0}}\right\}\asymp_k\beta.\ee
To see this, fix a small parameter $\delta=\delta(k)$ and observe that if $$\sum_{i\neq i_1}\ell_i\le\delta\ell_{i_1}=\delta\max_{1\le i\le k}\ell_i,$$ then $$(k-i_1+2)^\alpha\log(k-i_1+2)\ell_{i_1}=(1+O_k(\delta))(k-i_1+1)\ell_{i_1},$$ by the definition of $\alpha$. This implies that $|\alpha-\alpha_{k-i_1+1}|\ll_k\delta$. So if $\delta=\delta(k)$ is small enough, then $i_1=i_0$ and~\eqref{altbeta} follows immediately. Consider now the case when $\sum_{i\neq i_1}\ell_i\ge\delta\ell_{i_1}$. We may also assume that $i_1\neq i_0$; else,~\eqref{altbeta} holds trivially. Under these assumptions we have that $$\beta\ge\min\left\{1,\frac{\sum_{i\neq i_1}\ell_i}{\ell_{i_1}}\right\}\ge\delta$$ and
$$\min\left\{1,\frac{(1+\ell_1+\cdots+\ell_{i_0-1})(1+\ell_{i_0+1}+\cdots+\ell_k)}{\ell_{i_0}}\right\}
\ge\min\left\{1,\frac{\ell_{i_1}}{\ell_{i_0}}\right\}=1,$$
which together prove~\eqref{altbeta} in this last case too. By~\eqref{altbeta}, we see that~\eqref{he6} agrees with the conclusion of Theorem~\ref{thm2}.


\subsection{Further analysis and optimality of condition \eqref{e0}}\label{conde0} Even though the argument given in
Subsection \ref{main}\;gives us Theorem \ref{thm2} heuristically, it does not explain the presence of condition
\eqref{e0}\;in the statement of the theorem. This deficiency stems from the fact that the only piece of information we used about
$\mathcal{L}^{(k+1)}(\bv a)$ is Lemma \ref{ub1l0}. In order to understand condition~\eqref{e0}, we need to pay closer attention to
the structure of $\mathcal{L}^{(k+1)}(\bv a)$. It turns out that when $k$ is large, the rich multiplicative structure and the high
dimension of the set $\mathcal{L}^{(k+1)}(\bv a)$ lead to many more bounds on its volume $L^{(k+1)}(\bv a)$ than just those included in
the statement of Lemma \ref{ub1l0}.

\begin{lemma}\label{conv_ineq} Consider integers $0=z_0\le z_1\le\cdots\le z_k\le z_{k+1}=k$ with
$z_i\ge i-1$ for all $i\in\{1,\dots,k\}$. Let $\bv
a=(a_1,\dots,a_k)\in\SN^k$ such that $\mu^2(a_1\cdots a_k)=1$. Then we have that
\bes\begin{split}L^{(k+1)}(\bv a)&\le\sum_{\substack{d_j|a_j\\1\le j\le k}}\left(\prod_{j=1}^k(z_j-j+1)^{\omega(d_j)}\right)\\
&\times\min\left\{\prod_{j=0}^k\log^{z_{j+1}-z_j}(2a_1\cdots a_j),(\log 2)^k\prod_{j=1}^k(k-z_j+1)^{\omega(a_j/d_j)}\right\},\end{split}\ees with the convention that $0^0=1$.
\end{lemma}

\begin{proof} Given a $k$-tuple $(d_1,\dots,d_k)\in\SN^k$ with $d_1\cdots d_i|a_1\cdots a_i$ for $1\le i\le k$, we may uniquely write $d_i=d_{i,1}d_{i,2}\cdots d_{i,i}$, $1\le i\le k$, with $d_{j,j}d_{j+1,j}\cdots d_{k,j}|a_j$ for $1\le j\le k$. Thus $$\mathcal{L}^{(k+1)}(\bv a)=\bigcup_{\substack{d_{j,j}d_{j+1,j}\cdots d_{k,j}|a_j\\1\le j\le k}}\prod_{i=1}^k[\log(d_{i,1}d_{i,2}\cdots
d_{i,i}/2),\log(d_{i,1}d_{i,2}\cdots d_{i,i})).$$ For $i\in\{1,\dots,k\}$ define $m_i$ as the unique element of $\{0,1,\dots,k\}$ such that $z_{m_i}<i\le z_{m_i+1}$. Note that $i>z_{m_i}\ge m_i-1$ and thus $m_i\le i$. Set $$\mathcal{I}=\{(i,j):1\le j\le k,j\le i\le z_j\}=\{(i,j):1\le i\le k, m_i<j\le i\}.$$ Given numbers $d_{i,j}$, $(i,j)\in\mathcal{I}$, with $d_{j,j}\cdots d_{z_j,j}|a_j$, $1\le j\le k$, we define the set \bes\bsp\mathcal{L}(\{d_{i,j}:(i,j)\in\mathcal{I}\})&=\bigcup_{\substack{d_{i,j},\;(i,j)\notin\mathcal{I}\\1\le j\le i\le k\\d_{z_j+1,j}\cdots d_{k,j}|\frac{a_j}{d_{j,j}\cdots
d_{z_j,j}}\ \forall j}}\prod_{i=1}^k[\log(d_{i,1}\cdots
d_{i,i}/2),\log(d_{i,1}\cdots d_{i,i}))\\
&=\bigcup_{\substack{d_{z_j+1,j}\cdots d_{k,j}|\frac{a_j}{d_{j,j}\cdots
d_{z_j,j}}\\1\le j\le k}}\prod_{i=1}^k[\log(d_{i,1}\cdots
d_{i,m_i}/2),\log(d_{i,1}\cdots d_{i,m_i}))\\
&\quad+\left(\log(d_{1,m_1+1}\cdots d_{1,1}),\log(d_{2,m_2+1}\cdots
d_{2,2}),\dots,\log(d_{k,m_k+1}\cdots d_{k,k})\right).\end{split}\ees The above identity implies that $$\vol\left(\mathcal{L}(\{d_{i,j}:(i,j)\in\mathcal{I}\})\right)\le\min\left\{\prod_{i=1}^k\log(2a_1\cdots a_{m_i}),(\log2)^k\prod_{i=1}^k(k-z_i+1)^{\omega\left(\frac{a_i}{d_{i,i}\cdots d_{z_i,i}}\right)}\right\}.$$ Since $$\mathcal{L}^{(k+1)}(\bv a)=\bigcup_{\substack{d_{i,j},\;(i,j)\in\mathcal{I}\\d_{j,j}\cdots
d_{z_j,j}|a_j\ \forall j}}\mathcal{L}(\{d_{i,j}:(i,j)\in\mathcal{I}\}),$$ we find that
\bes\begin{split}L^{(k+1)}(\bv
a)&\le\sum_{\substack{d_{i,j},\;(i,j)\in\mathcal{I}\\D_j=d_{j,j}\cdots
d_{z_j,j}|a_j\ \forall j}}\min\left\{\prod_{i=1}^k\log(2a_1\cdots a_{m_i}),(\log2)^k\prod_{i=1}^k(k-z_i+1)^{\omega(a_i/D_i)}\right\}\\
&=\sum_{\substack{D_j|a_j\\1\le j\le
k}}\left(\prod_{i=1}^k(z_i-i+1)^{\omega(D_i)}\right)\\
&\qquad\times\min\left\{\prod_{i=1}^k\log(2a_1\cdots
a_{m_i}),(\log2)^k\prod_{i=1}^k(k-z_i+1)^{\omega(a_i/D_i)}\right\}.\end{split}\ees
To complete the proof of the lemma note that
$$\prod_{i=1}^k\log(2a_1\cdots
a_{m_i})=\prod_{j=0}^k\log^{z_{j+1}-z_j}(2a_1\cdots a_j).$$
\end{proof}

Using the above lemma, we show that condition~\eqref{e0} is optimal, that is to say, for every fixed $\gamma$ such that
\be\label{gamma}\frac1{\log2}\log\frac1{\log2}<\gamma<1-\frac1{\log(k+1)}\log\left(\frac{(k+1)\log(k+1)-2\log2}{k-1}\right),\ee
there are choices of $y_1\le y_2\le \cdots y_k$ such that $\alpha=\alpha(k;\bv y)=\gamma$ and \be\label{thmfails}\hxy=o\left(\frac{\beta}{\sqrt{\log\log y_k}}\prod_{i=1}^k\left(\frac{\log
y_i}{\log y_{i-1}}\right)^{-Q((k-i+2)^\alpha)}\right)\quad(y_1\to\infty).\ee The argument we give is heuristic but, if combined with the results of Sections~\ref{poisson} and~\ref{ub}, it can be made rigorous.

The right inequality in~\eqref{gamma} is equivalent to \be\label{gamma1}\frac{(k+1)\log(k+1)-2\log2}{(k+1)^{1-\gamma}}<k-1.\ee Also, inequalities~\eqref{e000} and~\eqref{gamma} imply that $$k-(k+1)^\gamma\log(k+1)>0\quad{\rm and}\quad2^\gamma\log2-1>0.$$ So if we select $y_1=y_2=\cdots=y_{k-1}$ large enough, then there is a unique $y_k\ge y_{k-1}$ such that $$\ell_k=\frac1{2^\gamma\log2-1}\sum_{i=1}^{k-1}(k-i+1-(k-i+2)^\gamma\log(k-i+2))\ell_i,$$ that is to say, there is a unique $y_k$ so that the $k$-tuple $\bv y=(y_1,\dots,y_k)$ satisfies the relation $\alpha(k,\bv y)=\gamma$. We claim that~\eqref{thmfails} holds and we support this claim with a heuristic argument:

Similarly to Subsection \ref{main}, we consider $\bv a=(a_1,\dots,a_k)$ such that $a_i\in\mathcal{P}_*(2y_{i-1},2y_i)$. Note that we necessarily have that $a_2=\cdots a_{k-1}=1$, since $y_1=\cdots = y_{k-1}$. Set $r_i=\omega(a_i)$ for all $i\in\{1,\dots,k\}$ and assume further that $\log a_i\asymp\log y_i$ for $i\in\{1,k\}$, that
\be\label{he300}
r_i\sim(k-i+2)^\alpha\ell_i=(k-i+2)^\gamma\ell_i\quad(i\in\{1,k\},\ y_1\to\infty)
\ee 
and that \eqref{he1b}~holds. We will show that 
\be\label{he11}
L^{(k+1)}(\bv a)
	= o\left(\prod_{i=1}^k\log y_i\right)\quad(y_1\to\infty).
\ee 
Indeed, Lemma~\ref{conv_ineq}, applied with $z_1=\cdots=z_k=k-1$, implies that
\bes\bsp 
L^{(k+1)}(\bv a)
	&\ll_k\sum_{\substack{d_j|a_j \\ 1\le j\le k}}\left(\prod_{j=1}^{k}(k-j)^{\omega(d_j)}\right)
		\min\left\{\log y_k,\prod_{j=1}^{k} 2^{\omega(a_j/d_j)}\right\}  \\
	& = \sum_{d_1|a_1}(k-1)^{\omega(d_1)}\min\left\{\log y_k,2^{r_k+\omega(a_1/d_1)}\right\}
\end{split}\ees
(note that all summands with $d_k>1$ vanish and $d_j=a_j=1$ for $j\in\{2,\dots,k-1\}$). The main contribution to the sum $$\sum_{d_1|a_1}(k-1)^{\omega(d_1)}2^{r_k+\omega(a_1/d_1)}=(k+1)^{r_1}2^{r_k}=\prod_{j=1}^k(k-j+2)^{r_j}
\asymp_k\prod_{i=1}^k\log y_i$$ comes from integers $d_1$ such that
\be\label{he7}\omega(d_1)\sim\frac{k-1}{k+1}\cdot r_1\quad(y_1\to\infty).\ee
If $d_1$ satisfies~\eqref{he7}, then relations~\eqref{he1b},~\eqref{gamma1}~and \eqref{he300}~ and the fact that $r_i=0$ and $\ell_i=O(1)$ for $i\in\{2,\dots,k-1\}$ imply
that \bes\bsp(r_k+\omega(a_1/d_1))\log 2-\log\log y_k&=\frac{2\log2}{k+1}r_1+(\log 2)r_k-\ell_1-\ell_k+o_k(\ell_1)\\
&=(k-1)\ell_1-\frac{(k+1)\log(k+1)-2\log2}{k+1}r_1+o_k(\ell_1)\to+\infty\end{split}\ees as $y_1\to\infty$.
Consequently, for integers $d_1$ that satisfy~\eqref{he7} we have that
$$\min\left\{\log y_k,2^{r_k+\omega(a_1/d_1)}\right\}=\log y_k=o\left(2^{r_k+\omega(a_1/d_1)}\right)\quad(y_1\to\infty),$$ which in turn implies that
relation~\eqref{he11} is indeed true. This yields that, in contrast to the prediction of the arguments in
Subsection \ref{main}, $D_{k+1}(\bv a)$ is not well-distributed for such $\bv a$. Hence, in general, relation~\eqref{he6} overestimates the size of
$\hxy$.

\begin{rk} The information about $L^{(k+1)}(\bv a)$ that is contained in Lemma~\ref{conv_ineq} makes its appearance implicitly in the statement of Lemma \ref{lb1l2}. An approach that could potentially extend Theorem~\ref{thm2} to the case when condition~\eqref{e0} fails is to insert Lemma~\ref{conv_ineq} into the proof of the upper bound in Theorem \ref{thm2} (Section \ref{ub}) and then adjust the lower bound argument accordingly (Sections~\ref{lb_outline},~\ref{lb_proof1} and~\ref{lb_proof2}).
\end{rk}

\bigskip


\section{Local-to-global estimates}\label{loc_glob}

In this section we reduce the counting in $\hxy$ to the estimation of
$$\label{Sk}S^{(k+1)}(\bv t):=\sum_{\bv a\in\mathcal{P}_*^k(\bv t)}\frac{L^{(k+1)}(\bv a)}{a_1\cdots a_k}$$ and prove Theorem \ref{thm3}. This reduction has also been carried in the author's thesis~\cite{dkthesis}, but we give it here for completeness. The basic ideas behind it can be found in \cite{kf1}\;and \cite{dk}. However, the details are more complicated, especially in the proof of the upper bound implicit in Theorem \ref{thm3}, because of the presence of more parameters. Finally, we employ Theorem~\ref{thm3} to deduce Theorem~\ref{thm1} in Subsection~\ref{thm1pf}.

\begin{rk}\label{rk_induction} In order to show Theorem~\ref{thm3} for some $k\ge1$, we may assume
without loss of generality that $y_1>C_k'$, where $C_1',C_2',\dots,C_k',\dots$ is an increasing sequence of
large constants. Indeed, suppose for the moment that
Theorem~\ref{thm3} holds for all $k\ge1$ when $y_1>C_k'$ and consider
the case when $y_1\le C_k'$. Then either $y_k\le C_k'$, in which case
Theorem~\ref{thm3} follows immediately, or there exists
$l\in\{1,\dots,k-1\}$ such that $y_l\le C_k'<y_{l+1}$. In the latter
case let $\bv y^\prime=(y_{l+1},\dots,y_k)$ and $d=\lfloor
2y_1\rfloor\cdots\lfloor 2y_l\rfloor\le2^ly_1\cdots y_l\le(2C_k')^k$
and note that
$$
H^{(k-l+1)}\left(\frac xd,\bv y^\prime,2\bv y^\prime\right)\le\hxy
\le H^{(k-l+1)}(x,\bv y^\prime,2\bv y^\prime),
$$ 
Moreover,
$$
\frac{x/d}{y_{l+1}\cdots y_k}\ge\frac x{2^ly_1\cdots
y_k}\ge2^{k-l}y_k.
$$ 
As a result, the desired bound on $\hxy$ follows by Theorem~\ref{thm3} applied to $H^{(k-l+1)}(x,\bv y^\prime,2\bv
y^\prime)$ and $H^{(k-l+1)}(x/d,\bv y^\prime,2\bv y^\prime)$, which holds since $y_{l+1}>C_k'\ge C_{k-l}'$.
\end{rk}


\subsection{Auxiliary results}

\medskip

Before we launch into the proof of Theorem \ref{thm3}, we list a few
results from number theory and analysis that we shall need. First, we state a standard sieve estimate for easy reference (see for example~\cite[Theorem 06]{Hall_Ten}).

\begin{lemma}\label{prell1} For $4\le2z\le x$ we have
$$\lvert\{n\le x:P^-(n)>z\}\rvert\asymp\frac x{\log z}.$$
\end{lemma}

Next, we have the following result, which follows by Lemma 2.3(b) in \cite{dk}.

\begin{lemma}\label{prell4} Let $f:\SN\to[0,+\infty)$ be an arithmetic function that satisfies the inequality
$f(ap)\le C_f f(a)$, for all integers $a$ and all primes $p$ with $(a,p)=1$, where $C_f$ is a positive
constant depending only on $f$. Also, let $h\ge0$ and $3/2\le y\le x\le z^C$ for some $C>0$. Then
$$\sum_{\substack{a\in\mathcal{P}_*(y,x)\\a>z}}\frac{f(a)}{a\log^h(P^+(a))}
\ll_{C_f,h,C}\exp\left\{-\frac{\log z}{2\log x}\right\}\frac1{\log^hx}\sum_{a\in\mathcal{P}_*(y,x)}\frac{f(a)}a.$$
\end{lemma}

Finally, we need a covering lemma which is a slightly different
version of Lemma 3.15 in \cite{fol}. If $r$ is a positive real
number and $I$ is a $k$-dimensional rectangle, then $r I$ will
denote the rectangle which has the same center with $I$ and $r$
times its diameter. More formally, if $\bv{x_0}$ is the center of
$I$, then $r I:=\{r(\bv x-\bv{x_0})+\bv{x_0}:\bv x\in I\}.$ The
lemma is then formulated as follows:

\begin{lemma}\label{prell3} Let $I_1,...,I_N$ be $k$-dimensional cubes of the form
$[a_1,b_1)\times\cdots\times[a_k,b_k)$ $(b_1-a_1=\cdots=b_k-a_k>0)$. Then there exists a sub-collection $I_{i_1},\dots,I_{i_M}$ of mutually disjoint cubes such that 
$$
\bigcup_{n=1}^NI_n\subset\bigcup_{m=1}^M3I_{i_m}.
$$
\end{lemma}

\subsection{The lower bound in Theorem~\ref{thm3}} We start with the proof of the lower bound implicit in
Theorem \ref{thm3}, which is simpler. First, we prove a weaker result; then we use Lemma \ref{prell4}\;to
complete the proof. Note that the lemma below is similar to Lemma
2.1 in \cite{kf1}, Lemma 4.1 in \cite{kf2}\;and Lemma 3.2 in
\cite{dk}.

\begin{lemma}\label{lgl1} Let $k\ge1$, $x\ge 1$ and $3=y_0\le y_1\le y_2\le\cdots\le y_k$ such that $2^ky_1\cdots y_k\le x/y_k$ and $y_1>C_k^\prime$. Then
\[
\frac{\hxy}x
\gg_k\prod_{i=1}^k\left(\frac{\log y_i}{\log y_{i-1}}\right)^{-(k-i+2)}
\sum_{\substack{\bv a\in\mathcal{P}_*^k(2\bv y)\\a_i\le y_i^{1/(8k)}\;(1\le i\le k)}}\frac{L^{(k+1)}(\bv a)}{a_1\cdots a_k}.
\]
\end{lemma}

\begin{proof} Set $$x^\prime=\frac x{2^ky_1\cdots y_k}\ge y_k.$$ Consider integers $n=a_1\cdots a_kp_1\cdots
p_kb\le x$ such that the following hold:
\renewcommand{\labelenumi}{(\arabic{enumi})}
\begin{enumerate}
\item $\bv a\in\mathcal{P}_*^k(2\bv y)$ and $a_i\le y_i^{1/(8k)}$ for $i=1,\dots,k$;
\item $p_1,\dots,p_k$ are prime numbers with $(\log(y_1/p_1),\dots,\log(y_k/p_k))\in\mathcal{L}^{(k+1)}(\bv a)$;
\item If $x^\prime\le y_k^2$, then let $b$ be a prime number $>y_k^{1/8}$; if $x^{\prime}>y_k^2$, then let $b$ be an integer with $P^-(b)>2y_k$.
\end{enumerate}
Note that for every $i\in\{1,\dots,k\}$ all prime factors of $a_i$ lie in $(2y_{i-1},y_i^{1/(8k)}]$. Also, condition (2) in the
definition of $n$ is equivalent to the existence of integers $d_1,\dots,d_k$ such that $d_1\cdots d_i|a_1\cdots a_i$ and
$y_i/p_i<d_i\le2 y_i/p_i$ for all $i\in\{1,\dots,k\}$. In particular, $\tau_{k+1}(n,\bv y,2\bv y)\ge 1$. Furthermore, we have that
$$y_i^{7/8}\le\frac {y_i}{a_1\cdots a_i}\le\frac{y_i}{d_i}<p_i\le2\frac{y_i}{d_i}\le2y_i.$$ So $(a_1\cdots a_k,p_1\cdots p_kb)=1$ and hence this representation of $n$, if it exists, is unique up to a possible permutation of $p_1,\dots,p_k$ and the prime factors of $b$ that lie in $(y_1^{7/8},2y_k]$. Since $b$ has at most one such prime factor, $n$
has a bounded number of such representations. Fix $a_1,\dots,a_k$ and $p_1,\dots,p_k$ and note that
\be\label{lge2}X:=\frac x{a_1\cdots a_kp_1\cdots p_k}\ge\frac{x^\prime}{y_k^{1/8}}\ge(x^\prime)^{7/8}>2y_k^{1/8}.\ee
We start by counting the number of possibilities for $b$. We consider two cases. First, if $x^\prime>y_k^2$, then $X>4y_k$, by \eqref{lge2}, provided that $C_k'$
is large enough. So Lemma~\ref{prell1} implies that $$\sum_{b\;{\rm admissible}}1=\sum_{b\le X,P^-(b)>2y_k}1\gg_k\frac X{\log y_k},$$ by Lemma \ref{prell1}. On the other hand, if $x^\prime\le y_k^2$, then $$X=\frac x{a_1\cdots a_kp_1\cdots p_k}\le\frac{xd_1\cdots d_k}{a_1\cdots a_ky_1\cdots y_k}\le 2^kx^{\prime}\le2^ky_k^2.$$
The above inequality and~\eqref{lge2} imply that $$\sum_{b\;{\rm admissible}}1=\sum_{\substack{y_k^{1/8}<b\le X\\b\;\text{prime}}}1\ge\sum_{\substack{X/2<b\le X\\b~{\rm prime}}}1\gg\frac X{\log X}\gg_k\frac{X}{\log y_k}.$$ In any case, we have that $$\sum_{b\;{\rm admissible}}1\gg_k\frac
X{\log y_k}$$ and, consequently, \be\label{lge3}\hxy\gg_k\frac x{\log y_k}\sum_{\substack{\bv a\in\mathcal{P}_*(2\bv y)\\a_i\le
y_i^{1/8k}\;(1\le i\le k)}}\frac1{a_1\cdots a_k}\sum_{(\log\frac{y_1}{p_1},...,\log\frac{y_k}{p_k})\in\mathcal{L}^{(k+1)}(\bv
a)} \frac1{p_1\cdots p_k}.\ee Fix $\bv a\in\mathcal{P}_*^k(2\bv y)$ with $a_i\le y_i^{1/(8k)}$ for $i=1,\dots,k$. Let $\{I_r\}_{r=1}^R$ be
the collection of cubes $[\log(d_1/2),\log d_1)\times\cdots\times[\log(d_k/2),\log d_k)$ with $d_1\cdots
d_i|a_1\cdots a_i$, $1\le i\le k$. Then for $I=[\log(d_1/2),\log d_1)\times\cdots\times[\log(d_k/2),\log d_k)$ in this collection we
have that
$$\sum_{(\log\frac{y_1}{p_1},...,\log\frac{y_k}{p_k})\in I}\frac1{p_1\cdots p_k}=\prod_{i=1}^k\sum_{y_i/d_i<p_i\le2y_i/d_i}\frac1{p_i}\gg_k\frac1{\log y_1\cdots\log y_k},$$
because $d_i\le a_1\cdots a_i\le y_i^{1/8}$ for $1\le i\le k$. By Lemma~\ref{prell3}, there exists a sub-collection $\{I_{r_s}\}_{s=1}^S$ of
mutually disjoint cubes so that $$S(3\log2)^k\ge\vol\left(\bigcup_{s=1}^S3I_{r_s}\right)\ge\vol\left(\bigcup_{r=1}^RI_r\right)=L^{(k+1)}(\bv a).$$ Hence
\bes\bsp\sum_{(\log\frac{y_1}{p_1},...,\log\frac{y_k}{p_k})\in\mathcal{L}^{(k+1)}(\bv
a)}\frac1{p_1\cdots p_k}\ge\sum_{s=1}^S\sum_{(\log\frac{y_1}{p_1},...,\log\frac{y_k}{p_k})\in
I_{r_s}}\frac1{p_1\cdots p_k}&\gg_k\frac S{\log y_1\cdots\log y_k}\\
&\gg_k\frac{L^{(k+1)}(\bv a)}{\log y_1\cdots\log y_k}.\end{split}\ees Combining the above estimate with
\eqref{lge3}\;completes the proof of the lemma.
\end{proof}

Having proven the above lemma, it is not so hard to finish the proof of the lower bound of Theorem \ref{thm3}. We give the argument below.

\begin{proof}[Proof of Theorem~\ref{thm3} (lower bound)] For every fixed $i\in\{1,\dots,k\}$ and integers $a_1,\dots,a_{i-1}$ and $a_{i+1},\dots,a_k$, the function $a_i\to
L^{(k+1)}(\bv a)$ satisfies the hypothesis of Lemma \ref{prell4}\;with $C_f=k-i+2\le k+1$, by Lemma \ref{ub1l0}(b). So
if we set 
$$
\mathcal{P}=\left\{\bv a\in\SN^k:a_i\in\mathcal{P}\left(2y_{i-1},y_i^{1/M}\right)~(1\le i\le k)\right\}
$$ 
for some sufficiently large $M=M(k)$, then
\bes\begin{split}\sum_{\substack{\bv a\in\mathcal{P}_*^k(2\bv y)\\a_i\le y_i^{1/(8k)}\,(1\le i\le k)}}\frac{L^{(k+1)}(\bv a)}{a_1\cdots a_k}\ge\sum_{\substack{\bv
a\in\mathcal{P}\\a_i\le y_i^{1/(8k)}\,(1\le i\le k)}}\frac{L^{(k+1)}(\bv a)}{a_1\cdots a_k}&=\sum_{\bv a\in\mathcal{P}}\frac{L^{(k+1)}(\bv a)}{a_1\cdots a_k}\left(1+O_k\left(e^{-\frac M{16k}}\right)\right)\\
&\ge\frac12\sum_{\bv a\in\mathcal{P}}\frac{L^{(k+1)}(\bv a)}{a_1\cdots a_k}.\end{split}\ees
By the above inequality and Lemma \ref{ub1l0}(b), we deduce that
\bes\begin{split}S^{(k+1)}(\bv y)&\le\sum_{\bv a\in\mathcal{P}}\frac{L^{(k+1)}(\bv a)}{a_1\cdots a_k}\prod_{i=1}^k\sum_{\substack{b_i\in\mathcal{P}_*(y_{i-1},2y_{i-1})\\\text{or}\;b_i\in\mathcal{P}_*(y_i^{1/M},y_i)}}
\frac{\tau_{k-i+2}(b_i)}{b_i}\ll_k\sum_{\substack{\bv a\in\mathcal{P}_*^k(2\bv y)\\a_i\le y_i^{1/(8k)}\,(1\le i\le k)}}\frac{L^{(k+1)}(\bv a)}{a_1\cdots a_k}.\end{split}\ees
Combining the above estimate with Lemma \ref{lgl1}\;completes the
proof of the lower bound in Theorem \ref{thm3}.
\end{proof}


\subsection{The upper bound in Theorem~\ref{thm3}} In this subsection we complete the proof of Theorem \ref{thm3}. Before we proceed to the proof,
we need to define some auxiliary notation. For $\bv y,\bv z\in\SR^k$ and $x\ge1$ set 
$$
H^{(k+1)}_*(x,\bv y,\bv z)
	=\lvert\{n\le x:\mu^2(n)=1,\ \exists d_1\cdots d_k|n\ {\rm such~that}\ y_i<d_i\le z_i\ (1\le i\le k)\}\rvert.
$$ 
Also, for $\bv t\in[1,+\infty)^k$, $\bv h\in[0,+\infty)^k$ and $\epsilon>0$
define 
\bes\bsp\mathcal{P}_*^k(\bv t;\epsilon)
	= \left\{\bv a\in\SN^k  :   a_i\in\mathcal{P}_*\left( \max\left\{P^+(a_1\cdots a_{i-1}),
			\frac{t_{i-1}^\epsilon}{a_1\cdots a_{i-1}}\right\},t_i\right)  \ (1\le i\le k)\right\},
\end{split}\ees where $t_0=1$, and
$$
S^{(k+1)}(\bv t;\bv h,\epsilon)=\sum_{\bv a\in\mathcal{P}_*^k(\bv
t;\epsilon)}\frac{L^{(k+1)}(\bv a)}{a_1\cdots
a_k}\prod_{i=1}^k\log^{-h_i}\left(P^+(a_1\cdots
a_i)+\frac{t_i^\epsilon}{a_1\cdots a_i}\right).
$$ 
Lastly, let $$\label{ek}\bv e_k=(e_{k,1},\dots,e_{k,k})=(\underbrace{1,\dots,1}_{k-1~\text{times}},2)\in\SR^k.$$
Then we have the following estimate.

\begin{lemma}\label{lgl2}
Let $\sqrt{C_k'}\le y_1\le\cdots\le y_k\le x$ with $2^{k+1}y_1\cdots y_k\le
x/(2y_k)^{7/8}$. Then
$$H^{(k+1)}_*(x,\bv y,2\bv y)-H^{(k+1)}_*(x/2,\bv y,2\bv
y)\ll_kxS^{(k+1)}(2\bv y;\bv e_k,7/8).$$
\end{lemma}

\begin{proof} Let $n\in(x/2,x]$ be a square-free integer for which there exist integers $d_i\in(y_i,2y_i]$, $1\le i\le k$, with $d_1\cdots d_k|n$. If we set $d_{k+1}=n/(d_1\cdots d_k)$ and $y_{k+1}=x/(2^{k+1}y_1\cdots y_k)$, then we have that $n=d_1\cdots d_{k+1}$ with $y_i<d_i\le 2^{k+1}y_i$ for $1\le i\le k+1$. Let $z_1,\dots,z_{k+1}$ be the sequence
$y_1,\dots,y_{k+1}$ ordered increasingly. Also, let $\sigma$ be the unique permutation in $S_{k+1}$ for which $P^+(d_{\sigma(1)})<\cdots<P^+(d_{\sigma(k+1)})$ and set $p_j=P^+(d_{\sigma(j)})$ for $1\le j\le k+1$ and $p_0=1$. We can write $n=a_1\cdots a_k p_1\cdots p_kb$ with $P^-(b)>p_k$ and $a_i\in\mathcal{P}_*(p_{i-1},p_i)$ for all $1\le i\le k$. We claim that 
\be\label{lge7} 
p_i>Q_i  :=\max\left\{P^+(a_1\cdots a_i),\frac{(2y_i)^{7/8}}{a_1\cdots a_i}\right\}
	\quad(1\le i\le k).
\ee
Indeed, for every $j\in\{1,\dots,k\}$ we have that $y_{\sigma(j)}<d_{\sigma(j)}=p_jd$ for some $d|a_1\cdots a_j$ and therefore $y_{\sigma(j)}<p_ja_1\cdots a_j$. Consequently, 
$$
p_i=\max_{1\le j\le i}p_j  
	>  \max_{1\le j\le i}\frac{y_{\sigma(j)}}{a_1\cdots a_j}
	\ge  \frac{\max_{1\le j\le i}y_{\sigma(j)}}{a_1\cdots a_i}
	\ge \frac{z_i}{a_1\cdots a_i}\ge\frac{(2y_i)^{7/8}}{a_1\cdots a_i}\quad(1\le i\le k),
$$ 
by the definition of $z_1,\dots,z_{k+1}$ and our assumption that $y_1\le\cdots\le y_k\le \frac12y_{k+1}^{8/7}$. Moreover, 
$$
p_i=\max_{1\le j\le i}p_j>\max_{1\le j\le i}P^+(a_j)
	=P^+(a_1\cdots a_j).
$$ 
So~\eqref{lge7} follows. In addition, 
$$
P^+(a_i) < p_i 
		= P^+(d_{\sigma(i)})
		\le  \max_{1\le j\le i}P^+(d_j)
		\le 2y_i\quad(1\le i\le k),
$$ 
by the choice of
$\sigma$, and 
$$
P^-(a_i) > p_{i-1} > Q_{i-1}\quad(2\le i\le k),
$$ 
by~\eqref{lge7}. In particular, $\bv
a=(a_1,\dots,a_k)\in\mathcal{P}^k_*(2\bv y;7/8)$. Furthermore, note that 
$$
(d_{\sigma(1)}/p_1)\cdots(d_{\sigma(i)}/p_i) | a_1\cdots a_i
	\quad{\rm and}\quad
\frac{y_{\sigma(i)}}{p_i}<\frac{d_{\sigma(i)}}{p_i}\le\frac{2^{k+1}y_{\sigma(i)}}{p_i}
	\quad(1\le i\le k),
$$ 
that is to say, there are numbers $w_1,\dots,w_k\in\{1,2,2^2,\dots,2^k\}$ such that 
\be\label{lge77}
\left(\log\frac{w_1y_{\sigma(1)}}{p_1},\dots,\log\frac{w_ky_{\sigma(k)}}{p_k}\right)
	\in \mathcal{L}^{(k+1)}(\bv a).
\ee 
Lastly, observe that $p_{k+1}|b$ and consequently $b\ge p_{k+1}>p_k>Q_k$, by~\eqref{lge7}. Similarly, we have $P^-(b)>p_k>Q_k$. Combining all of the above, we deduce that
\be\label{lge9}\bsp
& H^{(k+1)}_*(x,\bv y,2\bv y)-H^{(k+1)}_*(x/2,\bv y,2\bv y) \\
& \le \sum_{\sigma\in S_{k+1}}\sum_{\substack{w_i\in\{1,2,\dots,2^k\}\\1\le i\le k}}
	\sum_{\bv a\in\mathcal{P}_*^k(2\bv y;7/8)}
	\sum_{\substack{p_1,\dots,p_k \\ \eqref{lge7},\eqref{lge77} }}
	\sum_{\substack{Q_k<b\le x/(a_1\cdots a_kp_1\cdots p_k) \\ P^-(b)>Q_k}}1 \\
&\ll_k \sum_{\sigma\in S_{k+1}}\sum_{\substack{w_i\in\{1,2,\dots,2^k\}\\1\le i\le k}}
	\sum_{\bv a\in\mathcal{P}_*^k(2\bv y;7/8)}
	\sum_{\substack{p_1,\dots,p_k \\  \eqref{lge7},\eqref{lge77}  }}
		\frac x{a_1\cdots a_kp_1\cdots p_k\log Q_k},
\end{split}\ee
by Lemma~\ref{prell1}. We fix $\sigma$, $w_1,\dots,w_k$ and $a_1,\dots,a_k$ as above and estimate the sum over the primes $p_1,\dots,p_k$ in the right hand side of~\eqref{lge9}. In order to analyze condition~\eqref{lge77}, consider the collection $\{I_r\}_{r=1}^R$ of cubes of the form $[\log(m_1/2),\log
m_1)\times\cdots\times[\log(m_k/2),\log m_k)$ with $m_1\cdots m_i|a_1\cdots a_i$ for $1\le i\le k$. By Lemma ~\ref{prell3}, there is a sub-collection $\{I_{r_s}\}_{s=1}^S$ of mutually disjoint such cubes for which $\mathcal{L}^{(k+1)}(\bv a)\subset\bigcup_{s=1}^S3I_{r_s}$.
Consider $I_{r_s}=[\log(m_1/2),\log m_1)\times\cdots\times[\log(m_k/2),\log m_k)$ in this sub-collection and set 
$$
U_i=\frac{w_i y_{ \sigma(i) } }{2m_i}   \quad(1\le i\le k).
$$ 
Then we find that 
\be\label{lge75}\bsp
&\left(\log\frac{w_1y_{\sigma(1)}}{p_1},\dots,\log\frac{w_ky_{\sigma(k)}}{p_k}\right) \\ 
&\qquad\qquad \in3I_{r_s} =  \left[\log(m_1/4),\log(2m_1)\right)\times\cdots\times\left[\log(m_k/4),\log(2m_k)\right)
\end{split}\ee
if, and only if, $U_i<p_i\le8U_i$ for all $i=1,\dots,k$. So $$\sum_{\substack{p_1,\dots,p_k\\\eqref{lge7},\eqref{lge75}}}\frac1{p_1\cdots p_k}\le\prod_{i=1}^k\sum_{\substack{U_i<p_i\le8U_i\\p_i>Q_i}}\frac1{p_i}\ll_k\prod_{i=1}^k\frac1{\log(\max\{U_i,Q_i\})}\le\prod_{i=1}^k\frac1{\log
Q_i}.$$  Therefore we deduce that $$\sum_{\substack{p_1,\dots,p_k\\\eqref{lge7},\eqref{lge77}}}\frac1{p_1\cdots
p_k}\le\sum_{s=1}^S\sum_{\substack{p_1,\dots,p_k\\\eqref{lge7},\eqref{lge75}}}\frac1{p_1\cdots
p_k}\ll_k\frac S{\log Q_1\cdots\log Q_k}\le\frac{L^{(k+1)}(\bv a)}{(\log2)^k\log Q_1\cdots\log Q_k}.$$
Inserting the above estimate into~\eqref{lge9} completes the proof of the lemma.
\end{proof}

Next, we bound the sum $S^{(k+1)}(\bv t;\bv h,\epsilon)$ from above in terms of $S^{(k+1)}(\bv t)$. This is accomplished by establishing an iterative
inequality that simplifies the complicated range of summation $\mathcal{P}_*^k(\bv t;\epsilon)$ by gradually reducing it to the
much simpler set $\mathcal{P}_*^k(\bv t)$ and, at the same time, eliminates the complicated logarithms that appear in the summands of
$S^{(k+1)}(\bv t;\bv h,\epsilon)$. Lemma~\ref{prell4} plays a crucial role in the proof of this inequality

\begin{lemma}\label{lgl3} Fix $k\ge1$, $\epsilon>0$ and $\bv h=(h_1,\dots,h_k)\in[0,+\infty)^k$.
For $\bv t=(t_1,\dots,t_k)$ with $3\le t_1\le\cdots\le t_k$ we have that
$$S^{(k+1)}(\bv t;\bv h,\epsilon)\ll_{k,\bv h,\epsilon}\left(\prod_{i=1}^k\log^{-h_i}t_i\right)S^{(k+1)}(\bv t).$$
\end{lemma}

\begin{proof} Set $\delta=\epsilon/(2k)$ and $t_0=1$. For $l\in\{1,\dots,k\}$ define
$$h_{l,i}=\begin{cases}h_i&\text{if}\;i\in\{1,\dots,l-1\}\cup\{k\},\cr h_i+k-i+1&\text{if}\;l\le i\le k-1,\end{cases}$$ and
\bes\begin{split}\mathcal{P}_l(\bv t)=\left\{\bv a\in\SN^k: a_i\in\mathcal{P}_*\left(\max\left\{P^+(a_1\cdots
a_{i-1}),t_{i-1}^{\epsilon/2+l\delta}/(a_1\cdots a_{i-1})\right\},t_i\right) \;(1\le i\le l)\right.,\\
\left.a_i\in\mathcal{P}_*(t_{i-1},t_i)\;(l+1\le i\le k)\right\}.\end{split}\ees Also, let $h_{0,i}=h_{1,i}$ for
$i\in\{1,\dots,k\}$ and $\mathcal{P}_0(\bv t)=\mathcal{P}_1(\bv t)$. Lastly, for $l\in\{0,\dots,k\}$ set $\bv
h_l=(h_{l,1},\dots,h_{l,k})$ and
\be\begin{split}\widetilde{S}_l^{(k+1)}(\bv t;\bv h_l)&=\sum_{\bv a\in\mathcal{P}_l(\bv t)}\frac{L^{(k+1)}(\bv a)}{a_1\cdots
a_k}\prod_{i=1}^l\log^{-h_{l,i}}\left(P^+(a_1\cdots a_i)+\frac{t_i^{\epsilon/2+l\delta}}{a_1\cdots a_i}\right)\\
&\qquad\times\prod_{i=l+1}^k\log^{-h_{l,i}}\left(P^+(a_1\cdots a_l)+\frac{t_i^{\epsilon/2+l\delta}}{a_1\cdots a_l}\right).\nonumber\end{split}\ee Note that \be\label{lgobs1}\widetilde{S}_k^{(k+1)}(\bv t;\bv h_k)=S^{(k+1)}(\bv t;\bv h,\epsilon)\ee and
\be\label{lgobs2}\bsp\widetilde{S}_0^{(k+1)}(\bv t;\bv h_0)&\asymp_{k,\epsilon,\bv h}\left(\prod_{i=1}^k(\log t_i)^{-h_{0,i}}\right)S^{(k+1)}(\bv t)\\
&=\left(\prod_{i=1}^{k-1}(\log t_i)^{-(h_i+k-i+1)}\right)(\log t_k)^{-h_k}S^{(k+1)}(\bv t).\end{split}\ee We claim that
\be\label{lge10}\widetilde{S}_l^{(k+1)}(\bv t;\bv h_l)\ll_{k,\bv h,\epsilon}(\log2t_{l-1})^{k-l+2}\widetilde{S}_{l-1}^{(k+1)}(\bv t;\bv h_{l-1})\quad(1\le l\le k).\ee
Clearly, if we prove~\eqref{lge10}, then the lemma will follow immediately by iterating~\eqref{lge10} and combining the resulting inequality with relations~\eqref{lgobs1} and~\eqref{lgobs2}. So we fix $l\in\{1,\dots,k\}$ and proceed to the proof of~\eqref{lge10}. Consider integers $a_1,\dots,a_{l-1}$ such that
$$a_i\in\mathcal{P}_*\left(\max\left\{P^+(a_1\cdots a_{i-1}),\frac{t_{i-1}^{\epsilon/2+l\delta}}{a_1\cdots
a_{i-1}}\right\},t_i\right)\quad(1\le i\le l-1)$$ and $a_{l+1},\dots,a_k$ such that
$$a_i\in\mathcal{P}_*(t_{i-1},t_i)\quad(l+1\le i\le k)$$ and set $$t_{l-1}'=\max\left\{P^+(a_1\cdots
a_{l-1}),\frac{t_{l-1}^{\epsilon/2+l\delta}}{a_1\cdots a_{l-1}}\right\}.$$ Observe that in order to show~\eqref{lge10} it
suffices to prove that
\be\label{lge11}\begin{split}T:&=\sum_{a_l\in\mathcal{P}_*(t_{l-1}',t_l)}\frac{L^{(k+1)}(\bv
a)}{a_l}\prod_{i=l}^k\log^{-h_{l,i}}\left(P^+(a_1\cdots a_l)+\frac{t_i^{\epsilon/2+l\delta}}{a_1\cdots a_l}\right)\\
&\ll_{k,\bv h,\epsilon}\sum_{a_l\in\mathcal{P}_*(t_{l-1}',t_l)}\frac{L^{(k+1)}(\bv a)}{a_l}\prod_{i=l}^k\log^{-h_{l,i}}\left(P^+(a_1\cdots
a_{l-1})+\frac{t_i^{\epsilon/2+(l-1)\delta}}{a_1\cdots a_{l-1}}\right).\end{split}\ee Indeed, if~\eqref{lge11} holds, then
Lemma~\ref{ub1l0}(b) and the relation
$$\sum_{a\in\mathcal{P}_*(t_{l-1}',t_{l-1})}\frac{\tau_{k-l+2}(a)}a=\prod_{t_{l-1}^\prime<p\le
t_{l-1}}\left(1+\frac{k-l+2}p\right)\ll_k\left(\frac{\log2t_{l-1}}{\log2t_{l-1}^\prime}\right)^{k-l+2}$$ imply that
\bes T\ll_{k,\bv h,\epsilon}\left(\frac{\log2t_{l-1}}{\log2t_{l-1}^\prime}\right)^{k-l+2}\sum_{a_l\in\mathcal{P}_*(t_{l-1},t_l)}\frac{L^{(k+1)}(\bv a)}{a_l}\prod_{i=l}^k\log^{-h_{l,i}}\left(P^+(a_1\cdots a_{l-1})+\frac{t_i^{\epsilon/2+(l-1)\delta}}{a_1\cdots a_{l-1}}\right)\ees thus
completing the proof of~\eqref{lge10}. To prove~\eqref{lge11} we decompose $T$ into the sums
$$
T_m:=\sum_{\substack{a_l\in\mathcal{P}_*(t_{l-1}',t_l)\\a_l\in I_m}}
	\frac{L^{(k+1)}(\bv a)}{a_l}  
		\prod_{i=l}^k\log^{-h_{l,i}}\left(P^+(a_1\cdots a_l)+\frac{t_i^{\epsilon/2+l\delta}}{a_1\cdots a_l}\right)
	\quad(l\le m\le k+1),
$$ 
where $I_l=(0,t_l^\delta]$, $I_m=(t_{m-1}^\delta,t_m^\delta]$ if $m\in\{l+1,\dots,k\}$ and
$I_{k+1}=(t_k^\delta,+\infty)$. First, we estimate $T_l$. If $a_l\in I_l$, then
$$
P^+(a_1\cdots a_l)+\frac{t_i^{\epsilon/2+l\delta}}{a_1\cdots a_l}
	\ge P^+(a_1\cdots a_{l-1})+\frac{t_i^{\epsilon/2+(l-1)\delta}}{a_1\cdots a_{l-1}}\quad(l\le i\le k)
$$ 
and thus we immediately deduce that
\be\label{lge12}
T_l \le \sum_{a_l\in\mathcal{P}_*(t_{l-1}',t_l)}\frac{L^{(k+1)}(\bv a)}{a_l}
	\prod_{i=l}^k\log^{-h_{l,i}}\left(P^+(a_1\cdots a_{l-1}) + \frac{t_i^{\epsilon/2+(l-1)\delta}}{a_1\cdots a_{l-1}}\right).
\ee
Next, we fix $m\in\{l+1,\dots,k+1\}$ and bound
$T_m$. For every $a_l\in I_m$ we have that $$P^+(a_1\cdots a_l)+\frac{t_i^{\epsilon/2+l\delta}}{a_1\cdots a_l}\ge\begin{cases}P^+(a_l)&\text{if}~l\le i<m,\cr \displaystyle
P^+(a_1\cdots a_{l-1})+\frac{t_i^{\epsilon/2+(l-1)\delta}}{a_1\cdots a_{l-1}}&\text{if}~m\le i\le k.\end{cases}$$ Moreover, the function
$a_l\to L^{(k+1)}(\bv a)$ satisfies the hypothesis of Lemma~\ref{prell4} with $C_f=k-l+2$, by Lemma~\ref{ub1l0}(b). Hence
\bes\begin{split}T_m&\le\left(\prod_{i=m}^k\log^{-h_{l,i}}\left(P^+(a_1\cdots a_{l-1})+\frac{t_i^{\epsilon/2+(l-1)\delta}}{a_1\cdots a_{l-1}}\right)\right)\sum_{\substack{a_l\in\mathcal{P}_*(t_{l-1}',t_l)\\a_l>t_{m-1}^\delta}}
\frac{L^{(k+1)}(\bv a)}{a_l(\log P^+(a_l))^{h_{l,l}+\cdots+h_{l,m-1}}}\\
&\ll_{k,\bv h,\epsilon}\left(\prod_{i=l}^k\log^{-h_{l,i}}\left(P^+(a_1\cdots
a_{l-1})+\frac{t_i^{\epsilon/2+(l-1)\delta}}{a_1\cdots a_{l-1}}\right)\right)\left(\prod_{i=l}^{m-1}\log^{h_{l,i}}t_i\right)\\
&\quad\quad\times\exp\left\{-\frac{\delta\log t_{m-1}}{2\log t_l}\right\}(\log t_l)^{-(h_{l,l}+\cdots+h_{l,m-1})}\sum_{a_l\in\mathcal{P}_*(t_{l-1}',t_l)}\frac{L^{(k+1)}(\bv a)}{a_l}\\
&\ll_{k,\bv h,\epsilon}\sum_{a_l\in\mathcal{P}_*(t_{l-1}',t_l)}\frac{L^{(k+1)}(\bv a)}{a_l}\prod_{i=l}^k\log^{-h_{l,i}}\left(P^+(a_1\cdots
a_{l-1})+\frac{t_i^{\epsilon/2+(l-1)\delta}}{a_1\cdots a_{l-1}}\right).\end{split}\ees Combining the above estimate with~\eqref{lge12} shows~\eqref{lge11}.
This completes the proof of~\eqref{lge10} and hence of the lemma.
\end{proof}

Before we prove the upper bound in Theorem~\ref{thm3}, we need one
last intermediate result.

\begin{lemma}\label{lgl4}  Let $1\le l\le k-1$ and $3\le t_1\le\cdots\le t_k$. Then
$$S^{(k-l+1)}(t_{l+1},\dots,t_k)\le(\log2)^{-l}S^{(k+1)}(t_1,\dots,t_k)$$ and $$S^{(k+1)}(t_1,\dots,t_k)\gg_k\log t_k.$$
\end{lemma}

\begin{proof} Note that \be\begin{split}\mathcal{L}^{(k+1)}(\bv a)&\supset\bigcup_{\substack{d_1\cdots d_i|a_1\cdots a_i~(1\le i\le
k)\\d_i=1~(1\le i\le l)}}\left[\log(d_1/2),\log d_1\right)\times\cdots\times\left[\log(d_k/2),\log d_k\right)\nonumber\\
&=[-\log2,0)^l\times\mathcal{L}^{(k-l+1)}(a_1\cdots a_{l+1},a_{l+2},\dots,a_k)\end{split}\ee and, consequently,
$$L^{(k+1)}(\bv a)\ge(\log2)^lL^{(k-l+1)}(a_1\cdots a_{l+1},a_{l+2},\dots,a_k).$$ Summing over $\bv a\in\mathcal{P}^k_*(\bv t)$ then proves the first part of the lemma.

For the second part, note that $$S^{(k+1)}(\bv a)\ge(\log2)^k\sum_{\bv a\in\mathcal{P}^k_*(\bv t)}\frac1{a_1\cdots a_k}\asymp_k\log t_k.$$
\end{proof}

We are now in position to show the upper bound in Theorem \ref{thm3}. In fact, we shall prove a slightly stronger estimate,
which will be useful in the proof of Theorem \ref{thm2}.

\begin{thm}\label{thm4} Fix $k\ge1$. Let $x\ge1$ and $C_k'\le y_1\le\cdots\le y_k$ with
$2^ky_1\cdots y_k\le x/y_k$. There exists a constant $c_k$ such that
$$\frac \hxy x\ll_k\left(\prod_{i=1}^k\log^{-e_{k,i}}y_i\right)\sum_{\substack{\bv a\in\mathcal{P}_*^k(\bv
y)\\a_i\le y_i^{c_k}\;(1\le i\le k)}}\frac{L^{(k+1)}(\bv a)}{a_1\cdots a_k}.$$
\end{thm}

\begin{proof} Observe that it suffices to show that \be\label{lge18}\hxy\ll_kx\left(\prod_{i=1}^k\log^{-e_{k,i}}y_i\right)T,\ee
where $$T:=\max\{S^{(k+1)}(\bv t):1\le t_1\le\cdots\le t_k,~\sqrt{y_i}\le t_i\le 2y_i~(1\le i\le k)\}.$$ Indeed, assume for
the moment that~\eqref{lge18} holds. Note that $$T\ll_k S^{(k+1)}(\bv y),$$ by Lemma \ref{ub1l0}(b) and inequality~\eqref{tauineq}. Also, for every $i\in\{1,\dots,k\}$, we have that $$\sum_{\substack{\bv a\in\mathcal{P}_*^k(\bv y)\\a_i>y_i^{c_k}}}\frac{L^{(k+1)}(\bv a)}{a_1\cdots a_k}\ll_ke^{-c_k/2}S^{(k+1)}(\bv y),$$ by Lemma~\ref{prell4} applied to the arithmetic function $a_i\to L^{(k+1)}(\bv a)$. Hence, if $c_k$ is large enough, we find that
$$T\ll_kS^{(k+1)}(\bv y)\le2\sum_{\substack{\bv a\in\mathcal{P}_*^k(\bv y)\\a_i\le y_i^{c_k}~(1\le i\le k)}}\frac{L^{(k+1)}(\bv a)}{a_1\cdots a_k}$$
which, together with~\eqref{lge18}, completes the proof of the theorem.

In order to prove~\eqref{lge18}, we first reduce the counting in $\hxy$ to square-free integers. Let $n\le x$ be an integer counted by $\hxy$ and write $n=ab$ with $a$ being square-full, $b$ square-free and $(a,b)=1$. The number of $n\le x$
with $a>(\log y_k)^{2k}$ is at most $$x\sum_{\substack{a>(\log y_k)^{2k}\\a~{\rm square-full}}}\frac1a\ll\frac x{(\log
y_k)^k}.$$ Assume now that $$a\in I_m:=\{a\in\SN\cap((\log y_{m-1})^{2k},(\log y_m)^{2k}]:a~{\rm square-full}\}$$ for some $m\in\{1,\dots,k\}$, where for the convenience of notation we have set $y_0=1$. Then we may uniquely write $d_i=f_ie_i$, $m\le i\le k$, with $f_m\cdots f_k|a$ and $e_m\cdots e_k|b$. Therefore 
\be\label{lge19}\bsp 
H^{(k+1)}(x,\bv y,2\bv y)\le\sum_{m=1}^k\sum_{a\in I_m}\sum_{f_m\cdots
f_k|a}H^{(k-m+2)}_*\left(\frac xa,\left(\frac{y_m}{f_m},\dots,\frac{y_k}{f_k}\right),2\left(\frac{y_m}{f_m},\dots,\frac{y_k}{f_k}\right)\right) \\
+O\left(\frac x{(\log y_k)^k}\right).\end{split}\ee Fix $m\in\{1,\dots,k\}$, $a\in I_m$ and $f_m,\dots ,f_k$ with $f_m\cdots f_k|a$. Let $z_m,\dots,z_k$ be the sequence $y_m/f_m,\dots,y_k/f_k$ in increasing order and set $\bv z^\prime=(z_m,\dots,z_k)$. Since $y_m\le\cdots\le y_k$ and
$$\frac{y_i}{f_i}\ge\frac{y_i}a\ge\frac{y_i}{(\log y_m)^k}\ge\sqrt{y_i}\quad(m\le i\le k),$$ we have that \be\label{lge2000}\sqrt{y_i}\le z_i\le y_i\quad(m\le i\le k).\ee Next, observe that \be\label{lge20}\bsp H^{(k-m+2)}_*\left(\frac xa,\bv z',2\bv z'\right)&\le\sum_{\substack{r\in\SN\\2^r\le(\log y_k)^k}}\left(H_*^{(k-m+2)}\left(\frac x{2^{r-1}a},\bv z^\prime,2\bv z^\prime\right)-H_*^{(k-m+2)}\left(\frac x{2^ra},\bv z^\prime,2\bv z^\prime\right)\right)\\
&\quad+\frac{2x}{a(\log y_k)^k}.\end{split}\ee For $r$ with $2^r\le(\log y_k)^k$ we have that
$$\frac{x/(2^{r-1}a)}{2^{k-m+2}z_m\cdots z_k}\ge\frac x{2^ky_1\cdots y_k}\frac1{2^ra}\ge\frac{y_k}{(\log y_k)^{3k}}\ge(2z_k)^{7/8}.$$ Thus Lemma~\ref{lgl2} (applied with $k-m+1$ in place of $k$, $x/(2^{r-1}a)$ in
place of $x$ and $z_m,\dots,z_k$ in place of $y_1,\dots,y_k$), Lemmas~\ref{lgl3} and~\ref{lgl4} and relation~\eqref{lge2000} yield
\be\label{lge1001}\bsp H_*^{(k-m+2)}\left(\frac x{2^{r-1}a},\bv z^\prime,2\bv z^\prime\right)-H_*^{(k-m+2)}\left(\frac x{2^ra},\bv z^\prime,2\bv z^\prime\right)
&\ll_k\frac x{2^ra}\left(\prod_{i=m}^k(\log z_i)^{-e_{k,i}}\right)S^{(k-m+2)}(2\bv z')\\
&\ll_k\frac x{2^ra}\left(\prod_{i=m}^k(\log y_i)^{-e_{k,i}}\right)T.\end{split}\ee Since $T\gg_k\log y_k$ by Lemma~\ref{lgl4}, inequalities~\eqref{lge20} and~\eqref{lge1001} yield \bes H_*^{(k-m+2)}\left(\frac xa,\bv z^\prime,2\bv z^\prime\right)\ll_k\frac xa\left(\prod_{i=m}^k\log^{-e_{k,i}}y_i\right)T+\frac x{a(\log y_k)^k}\ll_k\frac xa\left(\prod_{i=m}^k\log^{-e_{k,i}}y_i\right)T.\ees So
\bes\bsp\sum_{a\in I_m}\sum_{f_m\cdots f_k|a}H^{(k-m+2)}_*\left(\frac xa,\bv z',2\bv z'\right)
&\ll_k\frac{xT}{\prod_{i=m}^k(\log y_i)^{e_{k,i}}}\sum_{a\in I_m}\frac{\tau_{k-m+2}(a)}a\ll_k\frac{ x T}{ \prod_{i=1}^k(\log y_i)^{e_{k,i}}}.\end{split}\ees Inserting the above estimate into~\eqref{lge19} and using the inequality $T\gg_k\log y_k$ completes the proof of the theorem.
\end{proof}


\subsection{Proof of Theorem~\ref{thm1}}\label{thm1pf}

In this subsection we prove Theorem \ref{thm1}. Consider real numbers $3=N_0\le N_1\le\cdots\le N_{k+1}$. Using an inductive argument, similar to the one given in Remark~\ref{rk_induction}, we may assume without loss of generality that $N_1\ge4(C_k')^2$. Set $\bv N=(N_1,\dots,N_k)$ and note that
\be\label{lge25}A_{k+1}(N_1,\dots,N_{k+1})\ge H^{(k+1)}\left(\frac{N_1\cdots N_{k+1}}{2^{k^2}},\frac{\bv N}{2^k},\frac{\bv N}{2^{k-1}}\right)
\asymp_k H^{(k+1)}\left(N_1\cdots N_{k+1},\frac{\bv N}2,\bv N\right),\ee
by Corollary \ref{cor4}. Also, we have that
\be\label{thm1pfe1}A_{k+1}(N_1,\dots,N_{k+1}) \le \sum_{\substack{1\le 2^{m_i}\le N_i\\1\le i\le k}} H^{(k+1)} \left( \frac{N_1\cdots N_{k+1}}{2^{m_1+\cdots+m_k}},\left(\frac{N_1}{2^{m_1+1}},\dots,\frac{N_k}{2^{m_k+1}}\right),\left(\frac{N_1}{2^{m_1}},\dots,\frac{N_k}{2^{m_k}}\right)\right).\ee
For $i\in\{0,1,\dots,k\}$ let $\mathcal{M}_i$ be the set of
vectors $\bv m\in(\SN\cup\{0\})^k$ such that $2^{m_j}\le\sqrt{N_j}$
for $i<j\le k$ and $\sqrt{N_i}<2^{m_i}\le N_i$ and set
$$T_i=\sum_{\bv m\in\mathcal{M}_i}H^{(k+1)}\left(\frac{N_1\cdots
N_{k+1}}{2^{m_1+\cdots+m_k}},\left(\frac{N_1}{2^{m_1+1}},\dots,\frac{N_k}{2^{m_k+1}}\right),
\left(\frac{N_1}{2^{m_1}},\dots,\frac{N_k}{2^{m_k}}\right)\right).$$ We have that \be\label{lgae1}A_{k+1}(N_1,\dots,N_{k+1})\le\sum_{i=0}^kT_i,\ee by~\eqref{thm1pfe1}.
We fix $i\in\{0,1,\dots,k\}$ and proceed to the estimation of $T_i$. Consider $\bv m\in\mathcal{M}_i$ and let $\bv
N^\prime=(N_{i+1}^\prime,\dots,N_k^\prime)$ be the vector whose coordinates are the sequence $\{N_j/2^{m_j+1}\}_{j=i+1}^k$ in increasing order.
We have that $\sqrt{N_j}\le 2N_j'\le N_j$ for all $i+1\le j\le k$. Thus \be\begin{split}\label{lgae2}&H^{(k+1)}\left(\frac{N_1\cdots N_{k+1}}{2^{m_1+\cdots+m_k}},\left(\frac{N_1}{2^{m_1+1}},\dots,\frac{N_k}{2^{m_k+1}}\right),\left(\frac{N_1}{2^{m_1}},\dots,\frac{N_k}{2^{m_k}}\right)\right)\\
&\le H^{(k-i+1)}\left(\frac{N_1\cdots N_{k+1}}{2^{m_1+\cdots+m_k}},\bv N^\prime,2\bv N^\prime\right)\asymp_k\frac{N_1\cdots N_{k+1}}{2^{m_1+\cdots+m_k}}S^{(k-i+1)}(\bv N^\prime)\prod_{j=i+1}^k(\log N_j)^{-e_{k,j}},\end{split}\ee by Theorem~\ref{thm3}, with the notational convention that $S^{(1)}(\emptyset)=1$. Furthermore, we have that
\be\label{lgae3}\bsp S^{(k-i+1)}(\bv N')&\le(\log 2)^{-i}S^{(k+1)}(\sqrt{N_1},\dots,\sqrt{N_i},\bv N')\\
&\asymp_k S^{(k+1)}(N_1,\dots,N_k)\asymp_k\frac{H^{(k+1)}(N_1\cdots N_{k+1},\bv N/2,\bv N)}{N_1\cdots N_{k+1}}\prod_{j=1}^k(\log N_j)^{e_{k,j}},\end{split}\ee by Lemma~\ref{lgl4}, Corolllary~\ref{cor1} and Theorem~\ref{thm3}. Combining~\eqref{lgae2} and~\eqref{lgae3} we deduce that
\bes\begin{split}&H^{(k+1)}\left(\frac{N_1\cdots N_{k+1}}{2^{m_1+\cdots+m_k}},\bv N',2\bv N'\right)\ll_k\frac{H^{(k+1)}(N_1\cdots N_{k+1},\bv N/2,\bv N)}{2^{m_1+\cdots+m_k}}(\log N_i)^{k+1}.\end{split}\ees Summing the above inequality
over $\bv m\in\mathcal{M}_i$ gives us that $$T_i\ll_k H^{(k+1)}\left(N_1\cdots N_{k+1},\frac{\bv N}2,\bv N\right)\frac{(\log N_i)^{k+1}}{\sqrt{N_i}},$$ which together with~\eqref{lge25} and~\eqref{lgae1} completes the proof of Theorem~\ref{thm1}.


\section{Linear constraints on a Poisson distribution}\label{poisson}

A $k$-dimensional Poisson distribution with parameters $z_1,\dots,z_k$ is a probability distribution on the lattice
$(\SN\cup\{0\})^k$ that assigns to each lattice point $(r_1,\dots,r_k)$ the probability
$\prod_{i=1}^ke^{-z_i}z_i^{r_i}/r_i!$. Our goal in this section is to estimate the probability that lattice points obeying such a
distribution lie close to a hyperplane and other related quantities. Throughout this entire section we fix positive real numbers $\lambda_1,\dots,\lambda_k$ and we set $\Lambda=\max_{1\le i\le k}\lambda_i$. Given $R\ge\Lambda$, let $$\mathcal{H}^k(R)=\left\{(r_1,\dots,r_k)\in(\SN\cup\{0\})^k:R-\Lambda<\sum_{i=1}^k\lambda_ir_i\le R\right\}$$
$$\mathcal{H}_-^k(R)=\left\{(r_1,\dots,r_k)\in(\SN\cup\{0\})^k:\sum_{i=1}^k\lambda_ir_i\le R\right\}$$
and $$\mathcal{H}_+^k(R)=\left\{(r_1,\dots,r_k)\in(\SN\cup\{0\})^k:\sum_{i=1}^k\lambda_ir_i\ge
R\right\}.$$ Also, define the number $\alpha(R)=\alpha(R;k,\bv
z,\bv\lambda)$ implicitly via the equation
$$\sum_{i=1}^k\lambda_ie^{\alpha(R)\lambda_i}z_i=R$$ and set
$$\mathcal{H}^k(R,\delta)=\left\{\bv r\in\mathcal{H}^k(R):\left|r_i-e^{\alpha(R)\lambda_i}z_i\right|
\le\frac{\Lambda}{\lambda_i}\max\left\{k,\delta\sqrt{e^{\alpha(R)\lambda_i}z_i}\right\}\;(1\le i\le
k)\right\}.$$

\begin{rk}\label{sprk1} The motivation for the definition of $\alpha(R)$ may be briefly summarized as follows: By Stirling's
formula, we have that \be\label{spe1}\prod_{i=1}^k\frac{z_i^{r_i}}{r_i!}\sim_k\prod_{i=1}^k\frac1{\sqrt{2\pi r_i}}\left(\frac{z_ie}{r_i}\right)^{r_i}.\ee Using Lagrange multipliers, we see that when $\bv r$ ranges over
$\mathcal{H}^k(R)$, the maximum of the right hand side in~\eqref{spe1} occurs when
$r_i=e^{\alpha(R)\lambda_i}z_i+O_{k,\bv\lambda}(1)$ for all $i\in\{1,\dots,k\}$.
\end{rk}

\begin{lemma}\label{spl1} Let $k\in\SN$, $0<\delta\le1$, $z_1,\dots,z_k\ge1$ and $\lambda_1,\dots,\lambda_k>0$. There is a constant $c=c(k,\bv\lambda)$ such that the following hold:
\begin{enumerate}
\item If $R\ge\max\{\Lambda,\delta(z_1+\cdots+z_k)\}$, then $$\prob(\mathcal{H}^k(R,\delta))\gg_{k,\bv\lambda,\delta}\frac{e^{-c|\alpha(R)|}}{\sqrt{R}}\prod_{i=1}^k\exp\{-Q(e^{\alpha(R)\lambda_i})z_i\}.$$
\item If $R\ge\Lambda$, then $$\prob(\mathcal{H}^k(R))\ll_{k,\bv\lambda}\frac{e^{c|\alpha(R)|}}{\sqrt{R}}\prod_{i=1}^k\exp\{-Q(e^{\alpha(R)\lambda_i})z_i\}.$$
\end{enumerate}
\end{lemma}

\begin{proof} By Stirling's formula, we have that \be\label{spe0}\prod_{i=1}^ke^{-z_i}\frac{z_i^{r_i}}{r_i!}\asymp_k\left(\prod_{i=1}^k\frac{\sqrt{r_i+1}}{z_i}\right)e^{F(\bv r)},\ee where \be\label{spe00}F(\bv r)=-(z_1+\cdots+z_k)+\sum_{i=1}^k(r_i+1)\log\frac{z_ie}{r_i+1}.\ee Set $r_i^*=e^{\lambda_i\alpha(R)}z_i-1$ for $i\in\{1,\dots,k\}$. Without loss of generality, we assume that $r_k^*+1=\max_{1\le i\le k}(r_i^*+1)$, so that \be\label{spe2}r_k^*+1\asymp_{k,\bv\lambda}R.\ee In order to prove part (a) of the lemma, we shall employ quadratic approximation to $F(\bv r)$ around the point $\bv r^*$. However, for part (b) we need to be more careful: we shall reparametrize the set $\mathcal{H}^k(R)$ first and then use the saddle point method. We give the details of the proof below.

\medskip

(a) Since $R\ge\delta(z_1+\cdots+z_k)$, then~\eqref{spe2} yields $$e^{\lambda_k\alpha(R)}z_k\gg_{k,\lambda}R\ge\delta z_k$$ and thus $\alpha(R)\ge-C$ for some constant $C=C(k,\bv\lambda,\delta)$. In turn, this implies that $r_i^*+1\gg_{k,\bv\lambda,\delta}z_i\ge1$ for all $i\in\{1,\dots,k\}$. By Taylor's theorem, for every $\bv r\in\mathcal{H}^k(R,\delta)$ there is a vector $\bv\xi\in\SR^k$ that lies on the line segment connecting $\bv r$ and $\bv r^*$ and satisfies
\begin{align*} 
F(\bv r) 
	& = F(\bv r^*)+\sum_{i=1}^k\frac{\partial F(\bv r^*)}{\partial x_i}(r_i-r_i^*)
		+\frac12\sum_{1\le i,j\le k}\frac{\partial^2F(\bv\xi)}{\partial x_i\partial x_j}(r_i-r_i^*)(r_j-r_j^*) \\
	& = F(\bv r^*)-\sum_{i=1}^k\frac{(r_i-r_i^*)^2}{2\xi_i+2}
		+ O_{k,\bv\lambda}(|\alpha(R)|)=F(\bv r^*)
		+ O_{k,\bv\lambda}(1+|\alpha(R)|).
\end{align*}
Since we also have that 
\bes\bsp|\mathcal{H}^k(R,\delta)|&\ge\left\lvert\left\{\bv r\in\mathcal{H}^k(R):|r_i-r_i^*-1|\le\frac{\Lambda}{k\lambda_i}\max\left\{k,\delta\sqrt{r_i^*+1}-1\right\}~(1\le i\le k-1)\right\}\right\rvert\\
&\asymp_{k,\bv\lambda,\delta}\prod_{i=1}^{k-1}\sqrt{r_i^*+1}\asymp_{k,\bv\lambda}\sqrt{\frac{(r_1^*+1)\cdots(r_k^*+1)}R},\end{split}\ees 
by our assumption that $r_k^*+1=\max_{1\le i\le k}(r_i^*+1)$, and
\begin{equation}\label{Fr*}
F( \boldsymbol{r}^*) = - \sum_{i=1}^k Q( e^{\lambda_i \alpha(R)} ) z_i,
\end{equation}
the desired lower bound on $\prob(\mathcal{H}^k(R,\delta))$ follows.

\medskip

(b) Let $$\mathcal{R}=\{\tilde{\bv r}=(r_1,\dots,r_{k-1})\in(\SN\cup\{0\})^{k-1}:\lambda_1r_1+\cdots+\lambda_{k-1}r_{k-1}\le R\}$$ and, for $\tilde{\bv r}\in\mathcal{R}$, set $$f(\tilde{\bv r})=\frac1{\lambda_k}\left(R-\sum_{i=1}^{k-1}\lambda_ir_i\right)\quad{\rm and}\quad G(\tilde{\bv r})=F(\tilde{\bv r},f(\tilde{\bv r})),$$ where $F$ is defined by~\eqref{spe00}. Given $\tilde{\bv r}\in\mathcal{R}$, there is a positive but bounded number of integers $r_k$ such that $(r_1,\dots,r_k)\in\mathcal{H}^k(R)$: Indeed, we have that $(r_1,\dots,r_k)\in\mathcal{H}^k(R)$ if, and only if, \be\label{spe110}r_k\ge0\quad{\rm and}\quad f(\tilde{\bv r})-\Lambda/\lambda_k<r_k\le f(\tilde{\bv r}).\ee Also, relation~\eqref{spe110} and the Mean Value Theorem imply that there is some $\xi\in(r_k+1,f(\tilde{\bv r})+1)$ such that 
$$
(r_k+1)\log\frac{z_ke}{r_k+1}-(f(\tilde{\bv r})+1)\log\frac{z_ke}{f(\tilde{\bv r})+1}
	=  (f(\tilde{\bv r})-r_k)\log\frac{\xi}{z_k}.
$$ 
We have that 
$$
\log\frac\xi{z_k}\le\log\frac{R/\lambda_k+1}{z_k}=\log\frac{r_k^*+1}{z_k}
	+  O_{k,\bv\lambda}(1)\le\lambda_k|\alpha(R)|
	+ O_{k,\bv\lambda}(1),
$$ 
by~\eqref{spe2}. So~\eqref{spe0} yields that
$$
\prob(\mathcal{H}^k(R))\ll_{k,\bv\lambda}e^{\Lambda|\alpha(R)|}\sum_{\tilde{\bv r}\in\mathcal{R}}\left(\prod_{i=1}^{k-1}\frac{\sqrt{r_i+1}}{z_i}\right)\frac{\sqrt{f(\tilde{\bv r})+1}}{z_k}e^{G(\tilde{\bv r})}.
$$ 
Since we also have that $f(\tilde{\bv r})+1\le R/\lambda_k+1\asymp_{k,\bv\lambda}r_k^*+1,$ by~\eqref{spe2}, we deduce that
\be\label{spe3}
\prob(\mathcal{H}^k(R))
	\ll_{k,\bv\lambda}\frac{e^{O_{k,\bv\lambda}(|\alpha(R)|)}}{\sqrt{R}}
		\sum_{\tilde{\bv r}\in\mathcal{R}}\left(\prod_{i=1}^{k-1}\frac{\sqrt{r_i+1}}{z_i}\right)e^{G(\tilde{\bv r})}.
\ee 
In order to estimate the right hand side of~\eqref{spe3}, we shall use quadratic approximation to $G(\tilde{\bv r})$ around the point $\tilde{\bv r}^*=(r_1^*,\dots,r_{k-1}^*)$. We have that 
$$
\frac{\partial G(\tilde{\bv r}^*)}{\partial x_i}
	= \log\frac{z_i}{r_i^*+1}+\frac{\lambda_i}{\lambda_k}\log\frac{f(\tilde{\bv r}^*)+1}{z_k}	
	= \frac{\lambda_i}{\lambda_k}\log\frac{f(\tilde{\bv r}^*)+1}{r_k^*+1}
	\ll_{k,\bv\lambda}\frac1R
		\quad(1\le i\le k-1),
$$ 
by~\eqref{spe2} and~\eqref{spe110}. Also, 
$$
\frac{\partial^2G}{\partial x_i\partial x_j}(\tilde{\bv r})=-\frac{\delta_{i,j}}{r_i+1}-\frac{\lambda_i\lambda_j}{\lambda_k^2(f(\tilde{\bv r})+1)}\quad(1\le i,j\le k-1),
$$ 
where $\delta_{i,j}$ is the
standard Kronecker symbol. So for every $\tilde{\bv r}\in\mathcal{R}$ there is a vector $\bv\xi=(\xi_1,\dots,\xi_{k-1})$ that lies on the line
segment connecting $\tilde{\bv r}$ and $\tilde{\bv r}^*$ and satisfies
\be\label{spe4}\begin{split}
	G(\tilde{\bv r}) &=  G(\tilde{\bv r}^*)  
		+O_{k,\bv\lambda}\left(\frac1R\sum_{i=1}^{k-1}|r_i-r_i^*|\right)
			-   \sum_{i=1}^{k-1}\frac{(r_i-r_i^*)^2}{2(\xi_i+1)}
			-   \frac12\left(\sum_{i=1}^{k-1}\frac{\lambda_i(r_i-r_i^*)}{\lambda_k\sqrt{f(\bv\xi)+1}}\right)^2  \\
			& = G(\tilde{\bv r}^*)  
				+O_{k,\bv\lambda}(1)-\sum_{i=1}^{k-1}\frac{(r_i-r_i^*)^2}{2\xi_i+2}
				 -  \frac{(f(\tilde{\bv r})-f(\tilde{\bv r}^*))^2}{2f(\bv\xi)+2}.
\end{split}\ee
Next, we split the set $\mathcal{R}$ into certain subsets. Let 
$$
\mathcal{R}_1=\{\tilde{\bv r}\in\mathcal{R}:f(\tilde{\bv r})+1>(1+\eta)(r_k^*+1)-\Lambda/\lambda_k\},
$$ 
$$
\mathcal{R}_2=\{\tilde{\bv r}\in\mathcal{R}\setminus\mathcal{R}_1:r_i\le3r_i^*+4~(1\le i\le k-1)\},
$$ 
where $\eta=-1+2^{\lambda_k/\Lambda}>0$, and for $I\subset\{1,\dots,k-1\}$ set $$\mathcal{R}_3(I)=\{\tilde{\bv r}\in\mathcal{R}\setminus\mathcal{R}_1:r_i>3r_i^*+4~(i\in I),r_i\le3r_i^*+4~(i\notin I,1\le i\le k-1)\}.$$ If $\tilde{\bv r}\in\mathcal{R}_1$, then~\eqref{spe4} implies that 
$$
G(\tilde{\bv r})
	\le G(\tilde{\bv r}^*)
		+O_{k,\bv\lambda}(1)-\frac{(\eta r_k^*-O_{k,\bv\lambda}(1))^2}{2R/\lambda_k}
	\le G(\tilde{\bv r}_*)+O_{k,\bv\lambda}(1)-c_0R
$$ 
for some positive constant $c_0=c_0(k,\bv\lambda)$, by~\eqref{spe2}. Therefore 
\be\label{spe12}
	\sum_{\tilde{\bv r}\in\mathcal{R}_1} 
			\left(\prod_{i=1}^{k-1}\frac{\sqrt{r_i+1}}{z_i}\right)e^{G(\tilde{\bv r})}
		\ll_{k,\bv\lambda} R^{3(k-1)/2}e^{G(\tilde{\bv r}^*)-c_0R}
		\ll_{k,\bv\lambda}  e^{G(\tilde{\bv r}^*)} .
\ee
Next, if $\tilde{\bv r}\in\mathcal{R}_2$, then for any $\bv\xi$ that lies on the line
segment connecting $\tilde{\bv r}$ and $\tilde{\bv r}^*$ we have that $\xi_i\le3r_i^*+4$. Consequently, 
$$
G(\tilde{\bv r})
	\le G(\tilde{\bv r}^*)  +O_{k,\bv\lambda}(1)-\sum_{i=1}^{k-1}\frac{(r_i-r_i^*)^2}{6r_i^*+10},
$$ 
by~\eqref{spe4}. So we deduce that 
\be\label{spe10}\bsp
\sum_{\tilde{\bv r}\in\mathcal{R}_2}\left(\prod_{i=1}^{k-1}\frac{\sqrt{r_i+1}}{z_i}\right)e^{G(\tilde{\bv r})}
	&\ll_{k,\bv\lambda}\left(\prod_{i=1}^{k-1}\frac{\sqrt{r_i^*+2}}{z_i}\right)e^{G(\tilde{\bv r}^*)}
		\prod_{i=1}^{k-1}\sqrt{r_i^*+2} \\ 
	& = e^{O_{k,\bv\lambda}(1+|\alpha(R)|)+G(\tilde{\bv r}^*)}  .
\end{split}\ee 
Lastly, fix some non-empty set $I\subset\{1,\dots,k-1\}$ and $i\in I$ and consider $\tilde{\bv r}\in\mathcal{R}_3(I)$. Set 
$$
\tilde{\bv r_i}=(r_1,\dots,r_{i-1},r_i-1,r_{i+1},\dots,r_k).
$$ 
Then for every vector $\bv s$ that lies in the line segment connecting $\tilde{\bv r}$ and $\tilde{\bv r_i}$ we have that
\bes\begin{split}
	\frac{\partial G}{\partial x_i}(\bv s)
		\le  \log\frac{z_i}{r_i-1}+\frac{\lambda_i}{\lambda_k}\log\frac{f(\tilde{\bv r_i})+1}{z_k}
	& \le \log\frac{z_i}{3r_i^*+3}+\frac{\lambda_i}{\lambda_k}\log\frac{(1+\eta)(r_k^*+1)}{z_k}
		\le-\log\frac32.
\end{split}\ees 
So by the Mean Value Theorem we find that $e^{G(\tilde{\bv r})}\le\frac23e^{G(\tilde{\bv r_i})}$ and, consequently,
\bes
	\sum_{\tilde{\bv r}\in\mathcal{R}_3(I)}
		\left(\prod_{i=1}^{k-1}\frac{\sqrt{r_i+1}}{z_i}\right)e^{G(\tilde{\bv r})}
	\ll_{k,\bv\lambda}\sum_{\tilde{\bv r}\in\mathcal{R}_3(I\setminus\{i\})\cup\mathcal{R}_1}
		\left(\prod_{i=1}^{k-1}\frac{\sqrt{r_i+1}}{z_i}\right)e^{G(\tilde{\bv r})}.
\ees 
Iterating the above inequality yields that
\bes\sum_{\tilde{\bv r}\in\mathcal{R}_3(I)}
	\left(\prod_{i=1}^{k-1}\frac{\sqrt{r_i+1}}{z_i}\right)e^{G(\tilde{\bv r})}  \\
\ll_{k,\bv\lambda}\sum_{\tilde{\bv r}\in\mathcal{R}_1\cup\mathcal{R}_2}
	\left(\prod_{i=1}^{k-1}\frac{\sqrt{r_i+1}}{z_i}\right)e^{G(\tilde{\bv r})},
\ees 
since $\mathcal{R}_3(\emptyset)=\mathcal{R}_2$. Combining the above estimate with relations \eqref{spe3}, \eqref{spe12} and \eqref{spe10} shows that 
$$
\prob(\mathcal{H}^k(R))
	\le \frac{e^{O_{k,\bv\lambda}(1+|\alpha(R)|) + G(\tilde{\bv r}^*)}}{\sqrt{R}}.
$$ 
To complete the proof of the lemma, note that 
$$
|G(\tilde{\bv r}^*)-F(\bv r^*)| 
	=  |F(\tilde{\bv r}^*,f(\tilde{\bv r}^*))-F(\tilde{\bv r}^*,r_k^*)|
	\ll_{k,\bv\lambda}1+|\alpha(R)|,
$$ 
which together with\ \eqref{Fr*} implies that
$$
G(\boldsymbol{r}^*) = - \sum_{i=1}^k Q( e^{\lambda_i \alpha(R) } ) z_i
+ O_{k,\boldsymbol{\lambda} } (1+|\alpha(R)|).
$$
\end{proof}

Finally, as a consequence of Lemma~\ref{spl1}, we have the following estimates.

\begin{lemma}\label{spl2} Let $k\in\SN$, $C\ge0$, $z_1,\dots,z_k\ge1$, $\lambda_1,\dots,\lambda_k>0$ and $\mu_1,\dots,\mu_k>0$ such that $Z=\mu_1z_1+\cdots+\mu_kz_k\ge\Lambda$.
\renewcommand{\labelenumi}{(\alph{enumi})}
\begin{enumerate}\item If $\lambda_i<\mu_i$ for all $i\in\{1,\dots,k\}$, then
$$\sum_{\bv r\in\mathcal{H}^k_+(Z)}\left(1+\sum_{i=1}^k\lambda_ir_i-Z\right)^C\prod_{i=1}^k\frac{e^{-z_i}z_i^{r_i}}{r_i!}
\ll_{k,\bv\lambda,\bv\mu,C}\prob(\mathcal{H}^k(Z)).$$
\item If $\log(\mu_i/\lambda_i)<\lambda_i$ for all $i\in\{1,\dots,k\}$, then
$$\sum_{\bv r\in\mathcal{H}^k_-(Z)}\left(1+Z-\sum_{i=1}^k\lambda_ir_i\right)^C\prod_{i=1}^k\frac{e^{-z_i}(e^{\lambda_i}z_i)^{r_i}}{r_i!}
\ll_{k,\bv\lambda,\bv\mu,C}e^Z\prob(\mathcal{H}^k(Z)).$$
\end{enumerate}
\end{lemma}

\begin{proof}(a) Let $S_+$ be the sum in question. If we set $$G(R)=-\sum_{i=1}^kQ(e^{\alpha(R)\lambda_i})z_i=-R\,\alpha(R)+\sum_{i=1}^k\left(e^{\alpha(R)\lambda_i}-1\right)z_i,$$
then Lemma~\ref{spl1}(b) implies that \be\label{spe120}S_+\ll_{k,C,\bv\lambda,\bv\mu}\sum_{n=0}^\infty(1+n)^C\frac{\exp\{c|\alpha(Z+n\Lambda)|+G(Z+n\Lambda)\}}{\sqrt{Z+n\Lambda}}.\ee
Differentiating implicitly the defining equation of $\alpha(R)$, we find that there are positive constants $c_1=c_1(k,\bv\lambda)$ and $c_2=c_2(k,\bv\lambda)$ such that $$\frac{c_1}R\le\alpha'(R)=\left(\sum_{i=1}^k\lambda_i^2e^{\alpha(R)\lambda_i}z_i\right)^{-1}\le\frac{c_2}R\quad(R\ge\Lambda).$$ Also, we have that \be\label{spe341}\alpha(Z)\ge\min_{1\le i\le k}\frac1{\lambda_i}\log\left(\frac{\mu_i}{\lambda_i}\right)>0,\ee by the definition of $\alpha(Z)$ and our assumption that $\lambda_i<\mu_i$ for all $i$. So $$G'(R)=-\alpha(R)\le-\alpha(Z)<0\quad(R\ge Z).$$ Combining the above remarks, we see that the summands in the right hand side of~\eqref{spe120} decay exponentially. Hence $$S_+\ll_{k,C,\bv\lambda,\bv\mu}\frac{\exp\{c\alpha(Z)+G(Z)\}}{\sqrt{Z}},$$
which together with Lemma~\ref{spl1}(a) implies that $$S_+\ll_{k,C,\bv\lambda,\bv\mu}e^{2c\alpha(Z)}\prob(\mathcal{H}^k(Z)).$$
To complete the proof, note that \be\label{spe342}\alpha(Z)\le\max_{1\le i\le k}\frac1{\lambda_i}\log\left(\frac{\mu_i}{\lambda_i}\right)\ll_{k,\bv\lambda,\bv\mu}1,\ee by the definition of $\alpha(Z)$.

\medskip

(b) We argue as in part (a). Let $S_-$ be the sum we want to estimate. Then $$S_-\ll_{k,C,\bv\lambda,\bv\mu}\sum_{0\le n\le Z/\Lambda-1}(1+n)^C\frac{\exp\{c|\alpha(Z-n\Lambda)|+H(Z-n\Lambda)\}}{\sqrt{Z-n\Lambda}},$$
where $H(R)=R+G(R)$, by Lemma~\ref{spl1}(b). We have that $$H'(R)=1-\alpha(R)\ge1-\alpha(Z)>0\quad(R\ge Z),$$ by the first inequality in~\eqref{spe342} and our assumption that $\log(\mu_i/\lambda_i)<\lambda_i$ for all $i$. Thus $$S_-\ll_{k,C,\bv\lambda,\bv\mu}\frac{\exp\{c|\alpha(Z)|+H(Z)\}}{\sqrt{Z}}\ll_{k,\bv\lambda,\bv\mu}e^{Z+2c|\alpha(Z)|}\prob(\mathcal{H}^k(Z))
\ll_{k,\bv\lambda}e^Z\prob(\mathcal{H}^k(Z)),$$ by Lemma~\ref{spl1}(a) and relations~\eqref{spe341} and~\eqref{spe342}. This completes the proof of the lemma.
\end{proof}


\section{The upper bound in Theorem \ref{thm2}}\label{ub}

\subsection{Outline of the proof}\label{ub_outline}

In this subsection we give the key steps of the proof of the upper bound in Theorem \ref{thm2}\;with most of the technical details omitted. Observe that, in view of Corollary \ref{cor4}, we may assume that the numbers $\ell_1,\dots,\ell_k$ are sufficiently large. Our starting point is Theorem \ref{thm4}. We break the sum 
$$
\sum_{\substack{\bv a\in\mathcal{P}_*^k(\bv y)\\a_i\le y_i^{c_k}\;(1\le i\le k)}}\frac{L^{(k+1)}(\bv a)}{a_1\cdots a_k}
$$ 
into pieces according to the number of prime factors of the variables $a_1,\dots,a_k$. More precisely, set $\omega_k(a)=\lvert\{p|n:p>k\}\rvert$ and $$S^{(k+1)}_{\bv r}(\bv y)=\sum_{\substack{\bv a\in\mathcal{P}_*^k(\bv y)\\
\omega_k(a_i)=r_i,a_i\le y_i^{c_k}\\1\le i\le k}}\frac{L^{(k+1)}(\bv a)}{a_1\cdots a_k}\quad(\bv r\in(\SN\cup\{0\})^k).$$ Also, for each fixed $i\in\{1,\dots,k\}$
define a sequence of prime numbers $\lambda_{i,1},\lambda_{i,2},\dots$, as follows. Set
$\label{rho}\rho_m=(m+1)^{1/m}$ for $m\in\SN$, $\lambda_{i,0}=\max\{k,y_{i-1}\}$ and define inductively
$\lambda_{i,j}$ as the largest element of the set $\{p~{\rm prime}:\lambda_{i,0}<p\le y_i\}$ such that
\be\label{lambda1}\sum_{\lambda_{i,j-1}<p\le\lambda_{i,j}}\frac1p\le\log(\rho_{k-i+1}).\ee
Notice that the sequence $\{\lambda_{i,j}\}_{j\in\SN}$ eventually
becomes constant. Let $\label{vi}v_i$ be the smallest integer satisfying
$\lambda_{i,v_i}=\lambda_{i,v_i+1}$. Set
$$D_{i,j}=\{p\;\text{prime}:\lambda_{i,j-1}<p\le
\lambda_{i,j}\}\quad(1\le i\le k,1\le j\le v_i)$$ and observe that
\be\label{lambda2}\bigcup_{j=1}^{v_i}D_{i,j}=\{p\;{\rm
prime}:\max\{y_{i-1},k\}<p\le y_i\}\quad(1\le i\le k).\ee Also, we
have the following estimate.

\begin{lemma}\label{lambda} There exists some positive number $L_k$ such that 
$$
(\rho_{k-i+1})^{j-L_k}\le\frac{\log\lambda_{i,j}}{\log
y_{i-1}}\le(\rho_{k-i+1})^{j+L_k}\quad(1\le i\le k,1\le j\le v_i).
$$ 
Consequently, we have that
$$
v_i=\frac{\ell_i}{\log(\rho_{k-i+1})}+O_k(1)\quad(1\le i\le k).
$$
\end{lemma}

\begin{proof} The proof is similar to the proof of Lemma 4.6 in~\cite{kf2} and Lemma 3.4 in \cite{dk}.
\end{proof}

Set $\bv v=(v_1,\dots,v_k)$, $$\label{Dr}\Delta_r=\{(\xi_1,\dots,\xi_r)\in\SR^r:0\le\xi_1\le\cdots\le\xi_r\le1\}$$
and for $i\in\{1,\dots,k\}$ and
$\bv\xi_i=(\xi_{i,1},\dots,\xi_{i,r_i})\in \Delta_{r_i}$ define
$$F_i(\bv\xi_i)=\left(\min_{0\le j\le
r_i}\rho_{k-i+1}^{-j}(1+\rho_{k-i+1}^{v_i\xi_{i,1}}+\cdots+\rho_{k-i+1}^{v_i\xi_{i,j}})\right)^{k-i+1}.$$
We shall bound $S^{(k+1)}_{\bv r}(\bv y)$ in terms of
\be\begin{split}U^{(k+1)}_{\bv r}(\bv
v)=&\idotsint\limits_{\substack{\bv\xi_i\in
\Delta_{r_i}\\F_i(\bv\xi_i)\le C_k(k-i+2)^{v_i-r_i}\\1\le i\le
k}}\min_{1\le i\le
k}\left\{F_i(\bv\xi_i)\prod_{m=1}^{i-1}(k-m+2)^{v_m-r_m}\right\}d\bv\xi_1\cdots
d\bv\xi_k,\nonumber\end{split}\ee where $C_k$ is a sufficiently
large constant.

\begin{lemma}\label{ub1l1} If $y_1$ is large enough, then
$$
S^{(k+1)}_{\bv r}(\bv y)
\ll_k 
U^{(k+1)}_{\bv r}(\bv v) \prod_{i=1}^k\left(v_i(k-i+2)\log(\rho_{k-i+1})\right)^{r_i}.
$$
\end{lemma}

Lemma \ref{ub1l1}\;will be proven in Subsection \ref{ub_proof}. Next, we give an upper bound on $U_{\bv r}^{(k+1)}(\bv v)$, but
first we need to introduce some notation. For $\bv r\in(\SN\cup\{0\})^k$, $1\le i\le k+1$ and $1\le j\le k+1$ set $$B_{i,j}=\begin{cases}-\displaystyle\sum_{m=j}^{i-1}(r_m-v_m)\log(k-m+2)&\text{if}\;1\le
j<i,\cr \displaystyle0&\text{if}\;j=i,\cr
\displaystyle\sum_{m=i}^{j-1}(r_m-v_m)\log(k-m+2)&\text{if}\;i<j.\end{cases}$$
Observe that \be\label{ub1e7}B_{i,m}+B_{m,j}=B_{i,j}\quad(1\le i,m,j\le k+1).\ee For $j\in\{1,\dots,k+1\}$ set
$$
\mathcal{R}_j=\{\bv r\in(\SN\cup\{0\})^k:B_{i,j}\ge0\;(1\le i\le k+1)\}.
$$ 
Then 
$$
\bigcup_{j=1}^{k+1}\mathcal{R}_j=(\SN\cup\{0\})^k.
$$
Indeed, for every $\bv r\in(\SN\cup\{0\})^k$ there is some $j\in\{1,\dots,k+1\}$ such that $B_{1,j}\ge B_{1,i}$ for all
$i\in\{1,\dots,k+1\}$. So $\bv r\in\mathcal{R}_j$, by \eqref{ub1e7}.

\medskip

The following estimate will be shown in Subsection \ref{ub_proof}.

\begin{lemma}\label{ub1l2}Let $j\in\{1,\dots,k+1\}$ and $\bv
r\in\mathcal{R}_j$. Then $$U_{\bv r}^{(k+1)}(\bv
v)\ll_k\min\left\{1,\frac{(1+B_{i_0,j})(1+B_{i_0+1,j})}{r_{i_0}+1}\right\}\frac{\prod_{m=1}^{j-1}(k-m+2)^{v_m-r_m}}{r_1!\cdots
r_k!}.$$
\end{lemma}

By Lemma \ref{ub1l2}\;and the results of Section~\ref{poisson}, we obtain the following estimate, which will be proven in Subsection~\ref{ub_proof}.

\begin{lemma}\label{ub1l3} We have that 
$$
\sum_{\bv r\in(\SN\cup\{0\})^k}S^{(k+1)}_{\bv r}(\bv y)
\ll_k \frac {\beta}{ \sqrt{\log\log y_k} }
	\prod_{i=1}^k\left(\frac{\log y_i}{\log y_{i-1}}\right)^{k-i+2-Q((k-i+2)^\alpha)}.
$$
\end{lemma}

The upper bound in Theorem \ref{thm2}\;now follows immediately by Theorem \ref{thm4}\;and Lemma \ref{ub1l3}.


\subsection{Completion of the proof}\label{ub_proof}

In this subsection we give the proofs of Lemmas \ref{ub1l1},
\ref{ub1l2}\;and \ref{ub1l3}.

\begin{proof}[Proof of Lemma \ref{ub1l1}] Let $a_1=a_1^\prime p_{1,1}\cdots p_{1,r_1}\le y_1^{c_k}$ with $a_1^\prime\in\mathcal{P}_*(1,k)$ and $k<p_{1,1}<\cdots<p_{1,r_1}\le y_1$. Also, for $m\in\{2,\dots,k\}$ let $a_m=p_{m,1}\cdots p_{m,r_m}\le y_m^{c_k}$ with $y_{m-1}<p_{m,1}<\cdots<p_{m,r_m}$. For each $m\in\{1,\dots,k\}$ let $b_m=p_{m,1}\cdots p_{m,r_m}$. Also, for $1\le m\le k$ and $1\le i\le r_m$ define $n_{m,i}\in\{1,\dots,v_m\}$ by $p_{m,i}\in D_{m,n_{m,i}}$ and put $\bv n_m=(n_{m,1},\dots,n_{m,r_m})$. For every $i\in\{1,\dots,k\}$ Lemma \ref{ub1l0}(b)\;implies that
\bes\begin{split}
L^{(k+1)}(\bv a)
	&\le\tau_{k+1}(a_1^\prime,\underbrace{1,\dots,1}_{i-1\ \text{times}},b_{i+1},\dots,b_k)
		L^{(k+1)}(b_1,\dots,b_i,\underbrace{1,\dots,1}_{k-i\ \text{times}})\\
	&=\tau_{k+1}(a_1^\prime)\left(\prod_{m=i+1}^k(k-m+2)^{r_m}\right)
		L^{(k+1)}(b_1,\dots,b_i,\underbrace{1,\dots,1}_{k-i\ \text{times}}).
\end{split}\ees
Moreover, Lemmas \ref{ub1l0}~and \ref{lambda}~together with our assumption that $a_i\le y_i^{c_k}$ for $1\le i\le k$ imply that for every $j\in\{0,1,\dots,r_i\}$ we have
\bes\begin{split}
	&L^{(k+1)}(b_1,\dots,b_i,\underbrace{1,\dots,1}_{k-i\ \text{times}})\\
	&\quad\le(k-i+2)^{r_i-j}L^{(k+1)}(b_1,\dots,b_{i-1},p_{i,1}\cdots p_{i,j},\underbrace{1,\dots,1}_{k-i\ \text{times}})\\
	&\quad\le(k-i+2)^{r_i-j}\left(\prod_{m=1}^{i-1}\log(2b_1\cdots b_m)\right)
		\left(\log(2b_1\cdots b_{i-1})+\log(p_{i,1}\cdots  p_{i,j})\right)^{k-i+1}\\
	&\quad\ll_k(k-i+2)^{r_i-j}\left(\prod_{m=1}^{i-1}\log y_m\right)
		\left(\log y_{i-1}\left(1+\rho_{k-i+1}^{n_{i,1}}+\cdots+	\rho_{k-i+1}^{n_{i,j}}\right)\right)^{k-i+1}\\
	&\quad\asymp_k(k-i+2)^{r_i}\left(\prod_{m=1}^{i-1}(k-m+2)^{v_m}\right)
		\left(\rho_{k-i+1}^{-j}\left(1+\rho_{k-i+1}^{n_{i,1}}+\cdots+\rho_{k-i+1}^{n_{i,j}}\right)\right)^{k-i+1} .
\end{split}\ees
So if we set 
$$
G_i(\bv n_i)=\left(\min_{0\le j\le r_i}\rho_{k-i+1}^{-j}
	\left(1+\rho_{k-i+1}^{n_{i,1}}+\cdots+\rho_{k-i+1}^{n_{i,j}}\right)\right)^{k-i+1}
$$
and 
$$
G(\bv n_1,\dots,\bv n_k)=\min_{1\le i\le k}\left\{G_i(\bv n_i)\prod_{m=1}^{i-1}(k-m+2)^{v_m-r_m}\right\},
$$
then we find that 
$$
L^{(k+1)}(\bv a)
	\ll_k\tau_{k+1}(a_1^\prime)G(\bv n_1,\dots,\bv n_k)\prod_{i=1}^k(k-i+2)^{r_i}.
$$ 
Next, note that
\be\label{ub2e4}\bsp G_i(\bv n_i)&\le(k-i+2)^{-r_i}\left(1+\rho_{k-i+1}^{n_{i,1}}+\cdots+\rho_{k-i+1}^{n_{i,r_i}}\right)^{k-i+1}\\
&\asymp_k(k-i+2)^{-r_i}\left(\frac{\log2a_i}{\log y_{i-1}}\right)^{k-i+1}\ll_k(k-i+2)^{v_i-r_i},\end{split}\ee
by Lemma \ref{lambda}\;and our assumption that $a_i\le y_i^{c_k}$.
Also, 
$$
\sum_{a_1^\prime\in\mathcal{P}_*(1,k)}\frac{\tau_{k+1}(a_1')}{a_1'}\ll_k1.
$$ 
So if
$\mathcal{N}$ denotes the set of $k$-tuples $\bv n=(\bv n_1,\dots,\bv n_k)$ satisfying $1\le n_{m,1}\le\cdots\le
n_{m,r_m}\le v_m$ for $1\le m\le k$ and inequality \eqref{ub2e4}, then
\be\label{ub2e3}
	S^{(k+1)}_{\bv r}(\bv y)
		\ll_k\sum_{\bv n\in\mathcal{N}}G(\bv n)  
			\prod_{i=1}^k \left((k-i+2)^{r_i}
			\sum_{\substack{p_{i,1}<\cdots<p_{i,r_i}\\p_{i,j}\in D_{i,n_{i,j}}\\1\le j\le r_i}}
				\frac1{p_{i,1}\cdots p_{i,r_i}}\right).
\ee
Fix $i\in\{1,\dots,k\}$. Let $g_{i,s}=|\{1\le j\le r_i:n_{i,j}=s\}|$ for $s\in\{1,\dots,v_i\}$. By
\eqref{lambda1}, the sum over $p_{i,1},\dots,p_{i,r_i}$ in \eqref{ub2e3}~is at most
\be\label{ub2e5}
	\prod_{s=1}^{v_i}\frac1{g_{i,s}!}
		\left(\sum_{p\in D_{i,s}}\frac1p\right)^{g_{i,s}}
	\le  \frac{(\log(\rho_{k-i+1}))^{r_i}}{g_{i,1}!\cdots g_{i,v_i}!}
	= (v_i\log(\rho_{k-i+1}))^{r_i}\vol(I(\bv n_i)),
\ee 
where
$$
I(\bv n_i)
	:=\{\bv\xi_i\in \Delta_{r_i}:n_{i,j}-1\le v_i\xi_{i,j}<n_{i,j}~(1\le j\le r_i)\}.
$$
By \eqref{ub2e3}\;and \eqref{ub2e5}\;we deduce that
\be\label{ub2e6}
	S^{(k+1)}_{\bv r}(\bv y)
		\ll_k\left(\prod_{i=1}^k\left(v_i(k-i+2)\log(\rho_{k-i+1})\right)^{r_i}\right)
			\sum_{\bv n\in\mathcal{N}}G(\bv n)
				\vol\left(I(\bv n_1)\times\cdots\times I(\bv n_k)\right).
\ee 
Finally, note that the definition of $I(\bv n_i)$ and~\eqref{ub2e4} imply that
$$
G_i(\bv n_i)\le(k-i+2)F_i(\bv\xi_i)\le(k-i+2)G_i(\bv n_i)\le C_k(k-i+2)^{v_i-r_i+1}\quad(\bv\xi_i\in I(\bv n_i))
$$ 
for some
sufficiently large constant $C_k$ and, consequently,
$$\sum_{\bv n\in\mathcal{N}}G(\bv n)\vol\left(I(\bv
n_1)\times\cdots\times I(\bv n_k)\right)\ll_kU^{(k+1)}_{\bv r}(\bv
v).$$ Inserting the above estimate into \eqref{ub2e6}\;completes the
proof of the lemma.
\end{proof}

Our next goal is to show Lemma \ref{ub1l2}. First, we state an auxiliary result.

\begin{lemma}\label{ub2l1} Let $\mu>1$, $A\ge0$, $r,v\in\SN$ and $\gamma\ge0$. Consider the set $\mathcal{T}_\mu(r,v,\gamma)$ of all vectors $(\xi_1,\dots,\xi_r)\in \Delta_r$ such that
$\mu^{v\xi_1}+\cdots+\mu^{v\xi_j}\ge\mu^{j-\gamma}$ for $1\le j\le r$. If $\gamma\ge r-v-A$, then
$$\vol\left(\mathcal{T}_\mu(r,v,\gamma)\right)\ll_{\mu,A}\frac1{r!}\min\left\{1,\frac{(\gamma-r+v+A+1)(\gamma+1)}r\right\}.$$
\end{lemma}

\begin{proof} If $1\le r\le2v$, then the result follows by Lemma 5.3 in \cite{dk} (see also Lemma 4.4 in \cite{kf1}) and the trivial bound $\vol(\mathcal{T}_\mu(r,v,\gamma) )\le\vol(\Delta_r)=1/r!$. If $r>2v$, then we have that $\gamma\ge r-v-A\ge r/2-A$ and, consequently, $$\frac{(\gamma-r+v+A+1)(\gamma+1)}r\gg_A1.$$ So the lemma holds in this case too by the trivial estimate $\vol\left(\mathcal{T}_\mu(r,v,\gamma)\right)\le1/r!$.
\end{proof}

\begin{proof}[Proof of Lemma \ref{ub1l2}] Let $j\in\{1,\dots,k+1\}$ and $\bv r\in\mathcal{R}_j$.
For each $i\in\{1,\dots,k\}$, let $\mathcal{T}_i$ be the set of $\bv\xi=(\bv\xi_1,\dots,\bv\xi_k)\in
\Delta_{r_1}\times\cdots\times\Delta_{r_k}$ such that
\be\label{defti}
	\min_{1\le s\le k}\left\{F_s(\bv\xi_s)\prod_{m=1}^{s-1}(k-m+2)^{v_m-r_m}\right\}
		= \min_{1\le s\le k}\{F_s(\bv\xi_s)e^{-B_{1,s}}\}
		= F_i(\bv\xi_i)e^{-B_{1,i}}
\ee
and 
\[
F_s(\bv\xi_s)\le C_k(k-s+2)^{v_s-r_s}\quad(1\le s\le k).
\]
Then for every $\bv\xi\in\mathcal{T}_i$ we have that 
\[
F_i(\bv\xi_i)e^{-B_{1,i}}\le\min_{1\le s\le k}
	\left\{\min\{C_k(k-s+2)^{v_s-r_s},1\}e^{-B_{1,s}}\right\},
\]
which, together with~\eqref{ub1e7}, implies that 
\[
F_i(\bv\xi_i)\le C_ke^{B_{1,i}} \min_{1\le s\le k}e^{-\max\{B_{1,s},B_{1,s+1}\}}
	= C_k e^{-\max\{B_{i,1},\dots,B_{i,k+1}\}}.
\]
Relation \eqref{ub1e7}\;and our assumption that $\bv r\in\mathcal{R}_j$ imply that $B_{i,j}=B_{i,s}+B_{s,j}\ge B_{i,s}$
for all $s\in\{1,\dots,k+1\}$, that is to say, $\max\{B_{i,1},\dots,B_{i,k+1}\}=B_{i,j}$ and, consequently,
\[
F_i(\bv\xi_i)\le C_ke^{-B_{i,j}}.
\] 
For $i\in\{1,\dots,k\}$ and $n\ge B_{i,j}\ge\max\{B_{i,i_0},B_{i,i_0+1},0\}$, define
$\mathcal{T}_i(n)$ to be the set of $(\bv\xi_1,\dots,\bv\xi_k)\in\mathcal{T}_i$ such that
\[
C_ke^{-n}<F_i(\bv\xi_i)\le C_ke^{-n+1}.
\] 
Then for $\bv(\xi_1,\dots,\xi_k)\in\mathcal{T}_i(n)$ relations~\eqref{ub1e7} and~\eqref{defti} imply that
\[
F_{i_0}(\bv\xi_{i_0})\ge e^{B_{i,i_0}}F_i(\bv\xi_i)>C_ke^{B_{i,i_0}-n}.
\]
Hence, for every $j\in\{1,\dots,r_{i_0}\}$, we have that
\be\begin{split}
\rho_{k-i_0+1}^{-j}\left(\rho_{k-i_0+1}^{v_{i_0}\xi_{i_0,1}}
		+\cdots+\rho_{k-i_0+1}^{v_{i_0}\xi_{i_0,j}}\right)
	&\ge\max\left\{\left(F_{i_0}(\bv\xi_{i_0})\right)^{1/(k-i_0+1)}
		- \rho_{k-i_0+1}^{-j},\rho_{k-i_0+1}^{-j}\right\}  \nonumber\\
	&\ge\frac12\left(F_{i_0}(\bv\xi_{i_0})\right)^{1/(k-i_0+1)}
		\ge(\rho_{k-i_0+1})^{-\frac{n-B_{i,i_0}}{\log(k-i_0+2)}},
\end{split}\ee
provided that $C_k$ is large enough. So Lemma~\ref{ub2l1} gives us that 
\be\begin{split}
	U^{(k+1)}_{\bv r}(\bv v)
		& \le \sum_{i=1}^k\int\limits_{\mathcal{T}_i}e^{B_{i,1}}F_i(\bv\xi_i)d\bv\xi\nonumber\\
		&\le C_k\sum_{i=1}^k\sum_{n\ge B_{i,j}}e^{B_{i,1}-n+1}
			\left(\prod_{\substack{1\le j\le k\\j\neq i_0}}\frac1{r_j!}\right)
			\vol\left(\mathcal{T}_{\rho_{k-i_0+1}}\left(r_{i_0},v_{i_0},\frac{n-B_{i,i_0}}{\log(k-i_0+2)}\right)\right)\\
		&\ll_k\sum_{i=1}^k\frac{e^{B_{i,1}}}{r_1!\cdots r_k!}	
			\sum_{n\ge B_{i,j}}\frac1{e^n}\min\left\{1,\frac{(n-B_{i,i_0}+1)(n-B_{i,i_0+1}+1)}{r_{i_0}+1}\right\}\\
		&\ll_k\sum_{i=1}^k\frac{e^{B_{i,1}}}{r_1!\cdots r_k!}\frac1{e^{B_{i,j}}}	
			\min\left\{1,\frac{(B_{i,j}-B_{i,i_0}+1)(B_{i,j}-B_{i,i_0+1}+1)}{r_{i_0}+1}\right\}\\
		&=\frac{ke^{B_{j,1}}}{r_1!\cdots r_k!}\min\left\{1,\frac{(B_{i_0,j}+1)(B_{i_0+1,j}+1)}{r_{i_0}+1}\right\},
\end{split}\ee 
which completes the proof of the lemma.
\end{proof}

We conclude this section with the proof of Lemma \ref{ub1l3}.

\begin{proof}[Proof of Lemma \ref{ub1l3}]Lemmas \ref{ub1l1}\;and \ref{ub1l2}\;imply that
\be\label{ub2e9}\begin{split}\sum_{\bv
r\in(\SN\cup\{0\})^k}S^{(k+1)}_{\bv r}(\bv
y)&\ll_k\sum_{j=1}^{k+1}\sum_{\bv
r\in\mathcal{R}_j}\left(\prod_{m=1}^{j-1}(k-m+2)^{v_m}\right)\left(\prod_{m=j}^k(k-m+2)^{r_m}\right)\\
&\qquad\times\min\left\{1,\frac{(1+B_{i_0,j})(1+B_{i_0+1,j})}{r_{i_0}+1}\right\}\prod_{m=1}^k\frac{(v_m\log\rho_{k-m+1})^{r_m}}{r_m!}\\
&=:\sum_{j=1}^{k+1}T_j.\end{split}\ee We fix $j\in\{1,\dots,k+1\}$
and bound $T_j$. We have that $\bv r\in\mathcal{R}_j$ if, and only if,
\be\label{ub2e10}\sum_{m=i}^{j-1}\log(k-m+2)(r_m-v_m)\ge0\quad(1\le
i\le j-1)\ee and \be\label{ub2e11}\sum_{m=j}^i\log(k-m+2)(r_m-v_m)\le0\quad(j\le i\le
k).\ee Let $\mathcal{R}_{1,j}$ be the set of vectors $\bv
r_1=(r_1,\dots,r_{j-1})\in(\SN\cup\{0\})^{j-1}$ such that
\eqref{ub2e10}\;holds and let $\mathcal{R}_{2,j}$ be the set of
vectors $\bv r_2=(r_j,\dots,r_k)\in(\SN\cup\{0\})^{k-j+1}$ such that \eqref{ub2e11}\;holds.
Note that if $\bv r_1\in\mathcal{R}_{1,j}$, then
$$1+B_{i_0,j}=1+B_{i_0,1}+B_{1,j}\le(1+\max\{0,B_{i_0,1}\})(1+B_{1,j})\ll_k(1+\ell_1+\cdots+\ell_{i_0-1})(1+B_{1,j}),$$
since \eqref{ub2e10}\;implies that $B_{1,j}\ge0$. Similarly, if $\bv
r_2\in\mathcal{R}_{2,j}$, then
$$1+B_{i_0+1,j}=1+B_{i_0+1,k+1}+B_{k+1,j}\ll_k(1+r_{i_0+1}+\cdots+r_k)(1+B_{k+1,j}),$$
since \eqref{ub2e11}\;implies that $B_{k+1,j}\ge0$. So, if we set
$$\beta(\bv r)=\min\left\{1,\frac{(1+\ell_1+\cdots+\ell_{i_0-1})(1+r_{i_0+1}+\cdots+r_k)}{r_{i_0}+1}\right\},$$
then we have that \bes\begin{split}T_j&\ll_k\sum_{\substack{\bv
r_i\in\mathcal{R}_{i,j}\\i\in\{1,2\}}}\left(\prod_{m=1}^{j-1}(k-m+2)^{v_m}\right)\left(\prod_{m=j}^k(k-m+2)^{r_m}\right)\\
&\qquad\times\beta(\bv
r)(1+B_{1,j})(1+B_{k+1,j})\prod_{m=1}^k\frac{(v_m\log\rho_{k-m+1})^{r_m}}{r_m!}.\end{split}\ees
For $s\in\{0,1,\dots,k\}$ set
\be\begin{split}T_{j,s}&=\sum_{\substack{\bv
r_i\in\mathcal{R}_{i,j}\\i\in\{1,2\}}}\left(\prod_{m=1}^{j-1}(k-m+2)^{v_m}\right)\left(\prod_{m=j}^k(k-m+2)^{r_m}\right)\nonumber\\
&\qquad\times(1+B_{1,j})(1+B_{k+1,j})\frac{r_s+1}{r_{i_0}+1}\prod_{m=1}^k\frac{(v_m\log\rho_{k-m+1})^{r_m}}{r_m!},\end{split}\ee
where $r_0=0$. Then
\be\label{ub2e12}T_j\ll_k\min\left\{T_{j,i_0},(1+\ell_1+\cdots+\ell_{i_0-1})(T_{j,0}+T_{j,i_0+1}+T_{j,i_0+2}+\cdots+T_{j,k})\right\}.\ee
Observe that $T_{j,s}$ may be written as a product of two sums, with
the first one ranging over $\bv r_1\in\mathcal{R}_{1,j}$ and the
second one over $\bv r_2\in\mathcal{R}_{2,j}$. Lemma \ref{spl2}(a)
can be applied to the first of these sums (with $j-1$ in place of
$k$, $\{v_i\log(\rho_{k-i+1})\}_{i=1}^{j-1}$ in place of
$\{z_i\}_{i=1}^k$, $\{\log(k-i+2)\}_{i=1}^{j-1}$ in place of
$\{\lambda_i\}_{i=1}^k$ and $\{k-i+1\}_{i=1}^{j-1}$ in place of
$\{\mu_i\}_{i=1}^k$). Similarly, Lemma \ref{spl2}(b) can be applied
to the second sum. As a result, we deduce that
\be\label{ub2e13}T_{j,s}\ll_k\frac{1+\ell_s}{1+\ell_{i_0}}\left(\prod_{m=1}^k(k-m+2)^{v_m}\right)
\sum_{\substack{\bv
r_i\in\mathcal{R}_{i,j}^\prime\\i\in\{1,2\}}}\prod_{m=1}^k\frac{(v_m\log\rho_{k-m+1})^{r_m}}{r_m!},\ee
where $\ell_0=0$, $$\mathcal{R}^\prime_{1,j}=\left\{\bv
r_1\in(\SN\cup\{0\})^{j-1}:-\log(k+1)\le\sum_{m=1}^{j-1}\log(k-m+2)(r_m-v_m)\le0\right\}$$
and
$$\mathcal{R}^\prime_{2,j}=\left\{\bv r_2\in(\SN\cup\{0\})^{k-j+1}:-\log(k+1)\le\sum_{m=j}^k\log(k-m+2)(r_m-v_m)\le0\right\}.$$
Clearly, we have that
$$\mathcal{R}_{1,j}^\prime\times\mathcal{R}_{2,j}^\prime\subset\left\{\bv
r\in(\SN\cup\{0\})^k:-2\log(k+1)\le\sum_{m=1}^k\log(k-m+2)(r_m-v_m)\le0\right\},$$
which, in combination with relation~\eqref{ub2e13} and Lemmas~\ref{lambda} and~\ref{spl1}(b), implies that
$$T_{j,s}\ll_k\frac{\ell_s+1}{\ell_{i_0}+1}\frac1{\sqrt{\log\log y_k}}\prod_{i=1}^k\left(\frac{\log
y_i}{\log y_{i-1}}\right)^{k-i+2-Q((k-i+2)^\alpha)}.$$ By the above
estimate and \eqref{ub2e12}\;we deduce that
\bes T_j\ll_k\frac{\displaystyle\min\left\{1,\frac{(1+\ell_1+\cdots+\ell_{i_0-1})(1+\ell_{i_0+1}+\cdots+\ell_k)}{\ell_{i_0}}\right\}}
{\displaystyle\sqrt{\log\log y_k}}\prod_{i=1}^k\left(\frac{\log y_i}{\log y_{i-1}}\right)^{k-i+2-Q((k-i+2)^\alpha)}.\ees Finally,
inserting this inequality and~\eqref{altbeta} into~\eqref{ub2e9} proves the lemma.
\end{proof}


\section{The lower bound in Theorem \ref{thm2}: outline of the proof}\label{lb_outline}

As in the proof of the upper bound in Theorem~\ref{thm2}, our starting point in order to prove the corresponding lower
bound is Theorem \ref{thm3}. Also, we may assume that the numbers $\ell_1,\dots,\ell_k$ are large enough, by
Corollary \ref{cor4}. However, the arguments deviate significantly from those in Section \ref{ub}. As in~\cite{kf1,kf2,dk}, our strategy is to construct a subset of
$\mathcal{P}_*^k(\bv y)$ which contributes a positive proportion to $S^{(k+1)}(\bv y)$ and on which we have good control of the size of
$L^{(k+1)}(\bv a)$ via H\"older's inequality. First, for $P\in(1,+\infty)$ and $\bv a=(a_1,\dots,a_k)\in\SN^k$ set $$W_{k+1}^P(\bv a)=\sum_{\substack{d_1\cdots d_i|a_1\cdots a_i\\1\le i\le k}}\left(\sum_{\substack{d_1'\cdots d_i'|a_1\cdots a_i\\|\log(d_i^\prime/d_i)|<\log2\\1\le i\le k}}1\right)^{P-1}.$$ We have the following inequality.

\begin{lemma}\label{lb1l0} Let $P\in(1,+\infty)$ and consider a finite set $\mathcal{A}\subset\SN^k$. Then $$\left(\sum_{\bv
a\in\mathcal{A}}\frac{W_{k+1}^P(\bv a)}{a_1\cdots a_k}\right)^{1/P}\left(\frac1{(\log2)^k}\sum_{\bv
a\in\mathcal{A}}\frac{L^{(k+1)}(\bv a)}{a_1\cdots a_k}\right)^{1-1/P}\ge\sum_{\bv a\in\mathcal{A}}\frac{\tau_{k+1}(\bv a)}{a_1\cdots a_k}.$$
\end{lemma}

\begin{proof}The proof is similar to the proof of Lemma 3.3 in \cite{dk}
\end{proof}

Our next goal is to bound $$\sum_{\bv a\in\mathcal{A}}\frac{W_{k+1}^P(\bv a)}{a_1\cdots a_k}$$ from above for
suitably chosen sets $\mathcal{A}\subset\SN^k$. In order to construct these sets, recall the definition of the
numbers $\lambda_{i,j}$ and $v_i$ and of the sets $D_{i,j}$ from the beginning of Subsection \ref{ub_outline}. Then for $\bv g=(\bv g_1,\dots,\bv g_k)\in(\SN\cup\{0\})^{v_1}\times\cdots\times(\SN\cup\{0\})^{v_k}$
with $\bv g_i=(g_{i,1},\dots,g_{i,v_i})$ let $$\mathcal{A}(\bv g)=\mathcal{A}_1(\bv g_1)\times\cdots\times\mathcal{A}_k(\bv g_k),$$
where for each $i\in\{1,\dots,k\}$ $\mathcal{A}_i(\bv g_i)$ is defined to be the set of square-free integers composed of exactly
$g_{i,j}$ prime factors from $D_{i,j}$ for each $j\in\{1,\dots,v_i\}$. Set $G_{i,0}=0$ and $G_{i,j}=g_{i,1}+\cdots+g_{i,j}$, $j=1,\dots,v_i$.
We shall estimate 
$$
\sum_{\bv a\in\mathcal{A}(\bv g)}\frac{W_{k+1}^P(\bv a)}{a_1\cdots a_k},
$$ 
but first we need to introduce some additional notation. Fix $P\in(1,2]$ and set $$t_{i,j}=\frac{j+(k-i+2-j)^P}{k-i+2}\quad(1\le i\le k,~0\le j\le k-i+1).$$ Also, for integers $1\le i\le k$, $\nu\ge0$ and $n\ge0$ with $\nu+n\le k-i+1$ and for $\bv g_i\in(\SN\cup\{0\})^{v_i}$, set
$$T_i(\bv g_i;\nu,n)=\sum_{0=s_0\le s_1\le \cdots\le s_n\le s_{n+1}=v_i}(\rho_{k-i+1}^{P-1})^{-(s_1+\cdots+s_n)}
\prod_{j=0}^n(t_{i,\nu+j})^{G_{i,s_{j+1}}-G_{i,s_j}}.$$ Lastly, we define $$T(\bv g)=\sum_{\substack{0=J_0\le J_1\le\cdots\le J_k\le
k\\J_i\ge i~(1\le i\le k)}}\prod_{i=1}^k(\rho_{k-i+1}^{P-1})^{-(k-J_i)v_i}T_i(\bv g_i;J_{i-1}-i+1,J_i-J_{i-1}).$$

\begin{lemma}\label{lb1l2} Let $\bv r\in\SN^k$ and $\bv g=(\bv g_1,\dots,\bv g_k)\in(\SN\cup\{0\})^{v_1}\times\cdots\times(\SN\cup\{0\})^{v_k}$
such that $G_{i,v_i}=r_i$ for all $i\in\{1,\dots,k\}$. Then $$\sum_{\bv a\in\mathcal{A}(\bv g)}\frac{W_{k+1}^P(\bv a)}{a_1\cdots
a_k}\ll_kT(\bv g)\prod_{i=1}^k\frac{((k-i+2)\log\rho_{i-1+1})^{r_i}}{g_{i,1}!\cdots g_{i,v_i}!}.$$
\end{lemma}

The proof of Lemma \ref{lb1l2}\;will be given in Section \ref{lb_proof1}. Next, we use the above result to show that
$W_{k+1}^P(\bv a)$ is bounded on average over a union of suitably chosen sets $\mathcal{A}(\bv g)$, which we construct below. Define
\begin{align*}\mathcal{R}^*=\left\{(r_1,\dots,r_k)\in(\SN\cup\{0\})^k:-\log(k+1)\le\sum_{i=1}^k\log(k-i+2)(r_i-v_i)\le0,\right.\\
\left.|r_i-(k-i+2)^\alpha\ell_i|\le\sqrt{\ell_i}\quad(1\le i\le k)\right\}.\end{align*}
Fix $\bv r\in\mathcal{R}^*$ and $i\in\{1,\dots,k\}$
and set $$u_i'=1+\frac1{\log(k-i+2)}\sum_{j=1}^{i-1}\log(k-j+2)(v_j-r_j)$$
and $$w_i'=u_i'+v_i-r_i=1+\frac1{\log(k-i+2)}\sum_{j=1}^i\log(k-j+2)(v_j-r_j).$$
By Lemma~\ref{hl1} and the definition of $i_0$ (see also the derivation of~\eqref{he101}), we have that
\be\label{inequ}u_i'\asymp_k\begin{cases}1+\ell_1+\cdots+\ell_{i-1}&\text{if}~1\le i\le i_0,\cr
1+\ell_i+\cdots+\ell_k&\text{if}~i_0+1\le i\le k,\end{cases}\ee and
\be\label{ineqw}w_i'\asymp_k\begin{cases}1+\ell_1+\cdots+\ell_i&\text{if}~1\le i\le i_0-1,\cr
1+\ell_{i+1}+\cdots+\ell_k&\text{if}~i_0\le i\le k.\end{cases}\ee Define
$$u_i=\min\left\{u_i',\frac{r_i-v_i+\sqrt{(r_i-v_i)^2+4r_i}}2\right\}$$
and
$$w_i=u_i+v_i-r_i=\min\left\{w_i',\frac{v_i-r_i+\sqrt{(r_i-v_i)^2+4r_i}}2\right\}.$$
Note that $u_i\gg_k1$ and $w_i\gg_k1$, since $r_i\asymp_k v_i$ for
$\bv r\in\mathcal{R}^*$. Also, since $$u_i'w_i'=(u_i')^2+(v_i-r_i)u_i',$$ we have that $u_i'w_i'\le r_i$ exactly when $$u_i'\le\frac{r_i-v_i+\sqrt{(r_i-v_i)^2+4r_i}}2,$$ in which case $u_i=u_i'$ and $w_i=w_i'$. On the other hand, if $u_i'w_i'>r_i$, then we find similarly that
$$
u_i=\frac{r_i-v_i+\sqrt{(r_i-v_i)^2+4r_i}}2\quad{\rm and}\quad w_i=\frac{v_i-r_i+\sqrt{(r_i-v_i)^2+4r_i}}2.
$$ 
In any case, we have that
\be\label{betai}
	\beta_i:=\frac{u_iw_i}{r_i}=\min\left\{1,\frac{u_i'w_i'}{r_i}\right\}.
\ee 
Lastly, observe that
\be\label{beta2}
	\beta_i\asymp_k
		\begin{cases}	
			\beta &\text{if}\;i=i_0,\cr 
			1	 &\text{otherwise},
		\end{cases}
\ee 
by relations \eqref{inequ},~\eqref{ineqw} and~\eqref{altbeta}. For every $i\in\{2,\dots,k\}$ let $\mathcal{G}_i(r_i)$ be the set of vectors $\bv g_i\in(\SN\cup\{0\})^{v_i}$ such that
\be\label{lb1e2}G_{i,v_i}=r_i\quad\text{and}\quad G_{i,j}\le j+u_i\quad(1\le j\le v_i).\ee Also, let $\mathcal{G}_1(r_1)$ be the set
of vectors $\bv g_1=(g_{1,1},\dots,g_{1,v_1})\in(\SN\cup\{0\})^{v_1}$ that satisfy \eqref{lb1e2}\;with $i=1$ and have the additional
property that $g_{1,j}=0$ for $1\le j\le N-1$, where $N=N(k)$ is a sufficiently large constant to be chosen later. Finally, let $\mathcal{G}(\bv r)=\mathcal{G}_1(r_1)\times\cdots\times\mathcal{G}_k(r_k)$. Then the following estimates hold.

\begin{lemma}\label{lb1l3} For every $\bv r\in\mathcal{R}^*$ we have that $$\sum_{\bv g\in\mathcal{G}(\bv r)}\sum_{\bv a\in\mathcal{A}(\bv
g)}\frac1{a_1\cdots a_k}\gg_k\beta\prod_{i=1}^k\frac{\ell_i^{r_i}}{r_i!},$$ provided that $N$ is large enough.
\end{lemma}

\begin{lemma}\label{lb1l4} Assume that $\alpha$ satisfies~\eqref{e0} for some fixed $\epsilon>0$. If $P=P(k,\epsilon)$ is
close enough to 1, then for $\bv r\in\mathcal{R}^*$ we have that $$\sum_{\bv g\in\mathcal{G}(\bv r)}\sum_{\bv a\in\mathcal{A}(\bv
g)}\frac{W_{k+1}^P(\bv a)}{a_1\cdots a_k}\ll_{k,\epsilon}\beta\prod_{i=1}^k\frac{\left((k-i+2)\ell_i\right)^{r_i}}{r_i!}.$$
\end{lemma}

Lemmas~\ref{lb1l3} and~\ref{lb1l4} will be proven in Section \ref{lb_proof2}. Using these results, we complete the proof of Theorem \ref{thm2}.

\begin{proof}[Proof of Theorem \ref{thm2} (lower bound)] Assume that $\alpha$ satisfies~\eqref{e0} for
some fixed $\epsilon>0$. Fix $\bv r\in\mathcal{R}^*$. For every $\bv a\in\bigcup_{\bv g\in\mathcal{G}(\bv r)}\mathcal{A}(\bv g)$ we have
that $$\tau_{k+1}(\bv a)=\prod_{i=1}^k(k-i+2)^{r_i}\asymp_k\prod_{i=1}^k(k-i+2)^{v_i}\asymp_k\prod_{i=1}^k\left(\frac{\log
y_i}{\log y_{i-1}}\right)^{k-i+1},$$ by Lemma~\ref{lambda} and the definition of $\mathcal{R}^*$. Therefore \be\label{lb1e4}\sum_{\bv
g\in\mathcal{G}(\bv r)}\sum_{\bv a\in\mathcal{A}(\bv g)}\frac{L^{(k+1)}(\bv a)}{a_1\cdots a_k}\gg_{k,\epsilon}\beta\prod_{i=1}^k\left(\frac{\log y_i}{\log
y_{i-1}}\right)^{k-i+1}\prod_{i=1}^k\frac{\ell_i^{r_i}}{r_i!},\ee by
Lemmas \ref{lb1l0}, \ref{lb1l3}\;and \ref{lb1l4}. Also, relation~\eqref{lambda2} implies that $$\bigcup_{\bv r\in\mathcal{R}^*}\bigcup_{\bv g\in\mathcal{G}(\bv r)}\mathcal{A}(\bv g)\subset\mathcal{P}_*^k(\bv y).$$ Hence, combining~\eqref{lb1e4} with
Theorem~\ref{thm3}, we deduce that 
$$
\frac\hxy x\gg_{k,\epsilon}\beta e^{-(\ell_1+\cdots+\ell_k)}	
	\sum_{\bv r\in\mathcal{R}^*}\prod_{i=1}^k\frac{\ell_i^{r_i}}{r_i!}.
$$
Finally, we have that 
$$
e^{-(\ell_1+\cdots+\ell_k)}\sum_{\bv r\in\mathcal{R}^*}\prod_{i=1}^k\frac{\ell_i^{r_i}}{r_i!}
	\gg_k\frac1{\sqrt{\log\log y_k}}\prod_{i=1}^k\left(\frac{\log y_i}{\log y_{i-1}}\right)^{-Q((k-i+2)^\alpha)},
$$ 
by Lemma~\ref{spl1}(a), which completes the proof.
\end{proof}


\section{The method of low moments}\label{lb_proof1} 
This section is devoted to establishing Lemma \ref{lb1l2}. This will be done in three steps. Throughout this entire section we fix a vector $\bv r\in\SN^k$ and a vector $\bv g=(\bv g_1,\dots,\bv
g_k)\in(\SN\cup\{0\})^{v_1}\times\cdots\times(\SN\cup\{0\})^{v_k}$ with $G_{i,v_i}=r_i$ for all $i\in\{1,\dots,k\}$.
We set $R_i=\sum_{j=1}^i r_j$ and define 
\[
\mathcal{P}_{\bv r}=\left\{(Y_1,\dots,Y_k) : 
	Y_i\subset\{1,\dots,R_i\},\;Y_i\cap Y_j=\emptyset\;\text{if}\;i\neq j\right\}.
\] 
Also, we set 
\[
\mathcal{R}_i
	=\begin{cases}\{0,1,\dots,R_1\}
		&\text{if}\;i=1\cr \{R_{i-1}+1,\dots,R_i\}
		&\text{if}\;2\le i\le k.
\end{cases}
\]
For $I\in\{0,1,\dots,R_k\}$, we define $E_{\bv g}(I)\in\bigcup_{i=1}^k\{0,1,\dots,v_i\}$ as follows: if $I=0$, we set
$E_{\bv g}(I)=0$; else, we let $i$ be the unique number in $\{1,\dots,k\}$ such that $R_{i-1}<I\le R_i$ and we define $E_{\bv g}(I)$ by
$$
G_{i,E_{\bv g}(I)-1}<I-R_{i-1}\le G_{i,E_{\bv g}(I)}.
$$ 
For $\bv Y=(Y_1,\dots,Y_k)\in\mathcal{P}_{\bv r}$, $\bv m=\{m_1,\dots,m_k\}$ a permutation of $\{1,\dots,k\}$ and $I_1,\dots,I_k\in\{0,1,\dots,R_k\}$ we put 
\begin{align*}
&M_{\bv r}(\bv Y;\bv I;\bv m) \\
	&\qquad = \left| \left\{(Z_1,\dots,Z_k)\in\mathcal{P}_{\bv r}:\bigcup_{i=j}^k\left(Z_{m_i}\cap(I_j,R_k]\right)
	= \bigcup_{i=j}^k\left(Y_{m_i}\cap(I_j,R_k]\right)\;(1\le j\le k)\right\} \right| .
\end{align*}
In addition, we let 
\[
\mathcal{J}=\left\{(\mathcal{J}_1,\dots,\mathcal{J}_k):\mathcal{J}_i\subset\{1,\dots,k\},\ 
	\sum_{m=1}^i|\mathcal{J}_m|\ge i~(1\le i\le k),\ 
	\mathcal{J}_i\cap\mathcal{J}_j=\emptyset\;\text{if}\;i\neq j\right\}
\]
and, for $(\mathcal{J}_1,\dots,\mathcal{J}_k)\in\mathcal{J}$,
we set $J_i=|\mathcal{J}_1|+\cdots+|\mathcal{J}_i|\ge i$ for all $i\in\{0,\dots,k\}$. Lastly, for a family of sets $\{X_i\}_{i\in I}$ we define
\[
\mathcal{U}(\{X_i:i\in I\}):=\left\{x\in\bigcup_{i\in I}X_i:|\{j\in I:x\in X_j\}|=1\right\}.
\] 
In particular, $\mathcal{U}(Y,Z)=Y\triangle Z$, the symmetric difference of $Y$ and $Z$.

\begin{rk}\label{lb2rk1} Assume that $Y_1,\dots,Y_n$ and $Z_1,\dots,Z_n$ satisfy $Y_i\cap Y_j=Z_i\cap Z_j=\emptyset$ for $i\neq j$. Then
\[
\mathcal{U}(\{Y_j\triangle Z_j:1\le j\le n\})
	=\left(\bigcup_{j=1}^nY_j\right)\triangle\left(\bigcup_{j=1}^nZ_j\right) .
\]
\end{rk}

\subsection{Interpolating between $L^1$ and $L^2$ estimates}\label{interpolation} The main difficulty in bounding $W_{k+1}^P(\bv a)$ when $P\in(1,2)$ is that it is hard to use combinatorial arguments directly due to the presence of the fractional exponent $P-1$ in the definition of $W_{k+1}^P(\bv a)$. To overcome this difficulty, we perform a special type of interpolation between $L^1$ and $L^2$ estimates. This is accomplished in Lemma~\ref{lb21l1} below, which is a generalization of Lemma 3.5 in~\cite{dk}.

\begin{lemma}\label{lb21l1}Let $P\in(1,2]$, $\bv r\in(\SN\cup\{0\})^k$ and $\bv g=(\bv g_1,\dots,\bv g_k)\in(\SN\cup\{0\})^{v_1}\times\cdots\times(\SN\cup\{0\})^{v_k}$ such that $G_{i,v_i}=r_i$ for $i=1,\dots,k$. Then
\be\begin{split}\sum_{\bv a\in\mathcal{A}(\bv g)}\frac{W_{k+1}^P(\bv
a)}{a_1\cdots
a_k}&\ll_k\sum_{(\mathcal{J}_1,\dots,\mathcal{J}_k)\in\mathcal{J}}\sum_{\bv
m}\sum_{\substack{I_j\in\mathcal{R}_i\\1\le i\le
k,j\in\mathcal{J}_i}}\sum_{\bv Y\in\mathcal{P}_{\bv r}}\left(M_{\bv r}(\bv
Y;\bv I;\bv
m)\right)^{P-1}\\
&\qquad\times\prod_{i=1}^k\frac{(\log(\rho_{k-i+1}))^{r_i}(\rho_{k-i+1}^{P-1})^{-(k-J_i)v_i}}{g_{i,1}!\cdots
g_{i,v_i}!}\prod_{j\in\mathcal{J}_i}(\rho_{k-i+1}^{P-1})^{-E_{\bv
g}(I_j)}.\nonumber\end{split}\ee
\end{lemma}

\begin{proof} Consider \be\label{lb2e0}\bv a=(a_1,\dots,a_k)=(p_1\cdots p_{R_1},p_{R_1+1}\cdots p_{R_2},\dots,p_{R_{k-1}+1}\cdots
p_{R_k})\in\mathcal{A}(\bv g)\ee such that \be\label{lb2e1}p_{R_{i-1}+G_{i,j-1}+1},\dots,p_{R_{i-1}+G_{i,j}}\in
D_{i,j}\quad(1\le i\le k,\;1\le j\le v_i)\ee and the primes in each interval $D_{i,j}$ for $i=1,\dots,k$ and $j=1,\dots,v_i$ are
unordered. Since the number $\prod_{i=1}^ka_i$ is square-free and $\omega(a_i)=r_i$ for all $i\in\{1,\dots,k\}$, the
$k$-tuples $(d_1,\dots,d_k)$ with $d_1\cdots d_i|a_1\cdots a_i$ for $1\le i\le k$ are in one
to one correspondence with the $k$-tuples $(Y_1,\dots,Y_k)\in\mathcal{P}_{\bv r}$ via the relation
$$d_j=\prod_{i\in Y_j}p_i\quad(1\le j\le k).$$ Using this observation twice, we find that
$$W_{k+1}^P(\bv a)=\sum_{(Y_1,\dots,Y_k)\in\mathcal{P}_{\bv r}}\left(\sum_{\substack{(Z_1,\dots,Z_k)\in\mathcal{P}_{\bv r}\\\eqref{lb2e2}}}1\right)^{P-1},$$
where for two $k$-tuples $(Y_1,\dots,Y_k)\in\mathcal{P}_{\bv r}$ and $(Z_1,\dots,Z_k)\in\mathcal{P}_{\bv r}$ condition~\eqref{lb2e2} is
defined by \be\label{lb2e2}-\log2<\sum_{i\in Y_j}\log p_i-\sum_{i\in Z_j}\log p_i<\log2\quad(1\le j\le k).\ee Moreover, each $k$-tuple $(a_1,\dots,a_k)\in\mathcal{A}(\bv g)$ has exactly $\prod_{i,j}g_{i,j}!$ representations of the form given in~\eqref{lb2e0}, corresponding to all the possible permutations of the prime numbers $p_1,\dots,p_{R_k}$ under condition~\eqref{lb2e1}. Hence \bes\bsp\sum_{\bv a\in\mathcal{A}(\bv g)}\frac{W^P_{k+1}(\bv a)}{a_1\cdots a_k}&=\left(\prod_{\substack{1\le i\le k\\1\le j\le v_i}}\frac1{g_{i,j}!}\right)\sum_{\substack{p_1,\dots,p_{R_k}\\\eqref{lb2e1}}}\frac1{p_1\cdots p_{R_k}}\sum_{(Y_1,\dots,Y_k)\in\mathcal{P}_{\bv r}}\left(\sum_{\substack{(Z_1,\dots,Z_k)\in\mathcal{P}_{\bv r}\\\eqref{lb2e2}}}1\right)^{P-1}\\
&=\left(\prod_{\substack{1\le i\le k\\1\le j\le v_i}}\frac1{g_{i,j}!}\right)\sum_{(Y_1,\dots,Y_k)\in\mathcal{P}_{\bv r}}\sum_{\substack{p_1,\dots,p_{R_k}\\\eqref{lb2e1}}}\frac1{p_1\cdots p_{R_k}}\left(\sum_{\substack{(Z_1,\dots,Z_k)\in\mathcal{P}_{\bv r}\\\eqref{lb2e2}}}1\right)^{P-1}.\end{split}\ees
So H\" older's inequality yields that
\bes\bsp\sum_{\bv a\in\mathcal{A}(\bv g)}\frac{W^P_{k+1}(\bv a)}{a_1\cdots a_k}\le\left(\prod_{\substack{1\le i\le k\\1\le j\le v_i}}\frac1{g_{i,j}!}\right)\sum_{(Y_1,\dots,Y_k)\in\mathcal{P}_{\bv r}}&\left(\sum_{\substack{p_1,\dots,p_{R_k}\\\eqref{lb2e1}}}\frac1{p_1\cdots p_{R_k}}\sum_{\substack{(Z_1,\dots,Z_k)\in\mathcal{P}_{\bv r}\\\eqref{lb2e2}}}1\right)^{P-1}\\
&\times\left(\sum_{\substack{p_1,\dots,p_{R_k}\\\eqref{lb2e1}}}\frac1{p_1\cdots p_{R_k}}\right)^{2-P}.\end{split}\ees
Note that 
\begin{align*}
	\sum_{\substack{p_1,\dots,p_{R_k}\\\eqref{lb2e1}}}
		\frac1{p_1\cdots p_{R_k}}
	\le\prod_{i=1}^k\prod_{j=1}^{v_i}\left(\sum_{p\in D_{i,j}}\frac1p\right)^{g_{i,j}}
	\le \prod_{i=1}^k(\log(\rho_{k-i+1}))^{r_i}
\end{align*}
by~\eqref{lambda1} and, consequently,
\be\label{lb2e3}\bsp 
	\sum_{\bv a\in\mathcal{A}(\bv g)}\frac{W_{k+1}^P(\bv a)}{a_1\cdots a_k}
		&\le\left(\prod_{i=1}^k\frac{(\log(\rho_{k-i+1}))^{(2-P)r_i}}{g_{i,1}!\cdots g_{i,v_i}!}\right) \\
		&\qquad\times \sum_{(Y_1,\dots,Y_k)\in\mathcal{P}_{\bv r}}\left(\sum_{(Z_1,\dots,Z_k)\in\mathcal{P}_{\bv r}}
\sum_{\substack{p_1,\dots,p_{R_k}\\\eqref{lb2e1},\eqref{lb2e2}}}\frac1{p_1\cdots p_{R_k}}\right)^{P-1} .
\end{split}\ee

Next, we estimate the sum over the primes above. In order to do so, we need to understand condition~\eqref{lb2e2}. Note that~\eqref{lb2e2} is equivalent to
\be\label{lb2e2d}-\log2<\sum_{i\in Y_j\setminus Z_j}\log p_i-\sum_{i\in Z_j\setminus Y_j}\log p_i<\log2\quad(1\le j\le k).\ee
Fix two $k$-tuples $(Y_1,\dots,Y_k)\in\mathcal{P}_{\bv r}$ and $(Z_1,\dots,Z_k)\in\mathcal{P}_{\bv r}$ and define the numbers $I_1,\dots,I_k$ and $m_1,\dots,m_k$ with $I_i\in(Y_{m_i}\triangle Z_{m_i})\cup\{0\}$ for all $i\in\{1,\dots,k\}$ inductively, as follows (see the proof of \cite[Lemma 3.5]{dk} for the motivation behind these definitions). Let $$I_1=\max\left\{\mathcal{U}(Y_1\triangle Z_1,\dots,Y_k\triangle Z_k)\cup\{0\}\right\}.$$
If $I_1=0$, set $m_1=k$. Else, define $m_1$ to be the unique element of $\{1,\dots,k\}$ such that $I_1\in Y_{m_1}\triangle Z_{m_1}$. Assume
we have defined $I_1,\dots,I_i$ for some $i\in\{1,\dots,k-1\}$ with $I_r\in (Y_{m_r}\triangle Z_{m_r})\cup\{0\}$ for $r=1,\dots,i$. Then
set $$I_{i+1}=\max\left\{\mathcal{U}\left(\{Y_j\triangle Z_j:j\in\{1,\dots,k\}\setminus\{m_1,\dots,m_i\}\}\right)\cup\{0\}\right\}.$$
If $I_{i+1}=0$, set $m_{i+1}=\max\{\{1,\dots,k\}\setminus\{m_1,\dots,m_i\}\}.$
Otherwise, define $m_{i+1}$ to be the unique element of $\{1,\dots,k\}\setminus\{m_1,\dots,m_i\}$ such that $I_{i+1}\in
Y_{m_{i+1}}\triangle Z_{m_{i+1}}$. This completes the inductive step.

Note that we must have $\{m_1,\dots,m_k\}=\{1,\dots,k\}$. Also, if we set $$\mathcal{J}_i=\{1\le j\le k:I_j\in\mathcal{R}_i\}\quad(1\le i\le k),$$ then observe that $(\mathcal{J}_1,\dots,\mathcal{J}_k)\in\mathcal{J}$, since $$J_i=\sum_{m=1}^i|\mathcal{J}_m|=|\{1\le j\le k:I_j\le R_i\}|\ge|\{1\le j\le k:m_j\le i\}|=i$$ for all $i\in\{1,\dots,k\}$. Set $\mathcal{I}=\{I_j:1\le j\le k,I_j>0\}$ and fix for the moment the primes $p_i$ for $i\in\{1,\dots,R_k\}\setminus\mathcal{I}$. Then~\eqref{lb2e2d} becomes a system of linear inequalities with respect to the set of variables $\{\log p_I:I\in\mathcal{I}\}$ that corresponds to a triangular matrix, up to a permutation of its rows. So a straightforward manipulation of the inequalities which constitute~\eqref{lb2e2d} implies that $p_I\in[X_I,4^kX_I]$ for $I\in\mathcal{I}$, where the numbers $X_I$ depend only on the primes $p_i$ for $i\in\{1,\dots,R_k\}\setminus\mathcal{I}$ and the $k$-tuples $(Y_1,\dots,Y_k)$ and $(Z_1,\dots,Z_k)$, which we have fixed. Consequently, \bes\sum_{\substack{p_I,~I\in\mathcal{I}\\\eqref{lb2e1},\eqref{lb2e2d}}}\prod_{I\in\mathcal{I}}\frac1{p_I}
\ll_k\prod_{i=1}^k\prod_{\substack{j\in\mathcal{J}_i\\I_j>0}}\frac1{\log(\max\{\lambda_{i,E_{\bv g}(I)-1},X_{I_j}\})}\ll_k\prod_{i=1}^k\prod_{j\in\mathcal{J}_i}\frac{(\rho_{k-i+1})^{-E_{\bv g}(I_j)}}{\log y_{i-1}},\ees by Lemma~\ref{lambda}. So we find that 
$$
\sum_{\substack{p_1,\dots,p_{R_k}\\\eqref{lb2e1},\eqref{lb2e2d}}}\frac1{p_1\cdots p_{R_k}}\ll_k\prod_{i=1}^k(\log(\rho_{k-i+1}))^{r_i}\prod_{j\in\mathcal{J}_i}\frac{(\rho_{k-i+1})^{-E_{\bv g}(I_j)}}{\log y_{i-1}}
$$ 
which, together with~\eqref{lb2e3}, implies that
\be\label{lb2e4}\bsp
	\sum_{\bv a\in\mathcal{A}(\bv g)}\frac{W_{k+1}^P(\bv a)}{a_1\cdots a_k}
		&\ll_k\left(\prod_{i=1}^k\frac{(\log(\rho_{k-i+1}))^{r_i}}{g_{i,1}!\cdots g_{i,v_i}!}\right) \\
		&\qquad \sum_{(Y_1,\dots,Y_k)\in\mathcal{P}_{\bv r}}\left(\sum_{(Z_1,\dots,Z_k)\in\mathcal{P}_{\bv r}}
			\prod_{i=1}^k\prod_{j\in\mathcal{J}_i}\frac{(\rho_{k-i+1})^{-E_{\bv g}(I_j)}}{\log y_{i-1}}\right)^{P-1}.
\end{split}\ee
Note that 
\be\label{lb2e5}
	\prod_{i=1}^k(\log y_{i-1})^{|\mathcal{J}_i|}
	\asymp_k\prod_{i=1}^ke^{(k-J_i)\ell_i}
	\asymp_k\prod_{i=1}^k(\rho_{k-i+1})^{(k-J_i)v_i},
\ee 
by Lemma~\ref{lambda}. Moreover, the definition of the numbers $I_1,\dots,I_k$ and $m_1,\dots,m_k$ implies that
$$
(I_j,R_k]\cap\mathcal{U}(\{Y_{m_r}\triangle Z_{m_r}:j\le r\le k\})=\emptyset\quad(1\le j\le k),
$$ 
which is equivalent to
$$
\bigcup_{r=j}^k\left(Z_{m_r}\cap(I_j,R_k]\right)
	=\bigcup_{r=j}^k\left(Y_{m_r}\cap(I_j,R_k]\right)\quad(1\le
j\le k),
$$ by Remark \ref{lb2rk1}. Hence for fixed $(Y_1,\dots,Y_k)\in\mathcal{P}_{\bv r}$, $0\le I_1,\dots,I_k\le R_k$
and $\bv m=\{m_1,\dots,m_k\}$, a permutation of $\{1,\dots,k\}$, the number of admissible $k$-tuples
$(Z_1,\dots,Z_k)\in\mathcal{P}_{\bv r}$ is at most $M_{\bv r}(\bv Y;\bv I;\bv m)$. Combining this observation with~\eqref{lb2e4} and~\eqref{lb2e5} we deduce that
\bes\begin{split}\sum_{\bv a\in\mathcal{A}(\bv g)}&\frac{W_{k+1}^P(\bv a)}{a_1\cdots
a_k}\ll_k\left(\prod_{i=1}^k\frac{(\log(\rho_{k-i+1}))^{r_i}}{g_{i,1}!\cdots g_{i,v_i}!}\right)\\
&\times\sum_{(Y_1,\dots,Y_k)\in\mathcal{P}_{\bv r}}\left(\sum_{I_1,\dots,I_k}\sum_{\bv m}M_{\bv r}(\bv Y;\bv I;\bv
m)  \prod_{i=1}^k (\rho_{k-i+1})^{ -(k-J_i) v_i } \prod_{j\in\mathcal{J}_i } (\rho_{k-i+1})^{-E_{\bv g}(I_j) }  \right)^{P-1}.
\end{split}\ees
Finally, the inequality $(a+b)^{P-1}\le a^{P-1}+b^{P-1}$ for $a\ge0$ and $b\ge0$, which
holds precisely when $1<P\le 2$, completes the proof of the lemma.
\end{proof}


\subsection{Combinatorial arguments}\label{combinatorics} In this subsection we use combinatorial arguments to calculate $M_{\bv r}(\bv Y;\bv I;\bv m)$ and, as a result, simplify the estimate given by Lemma~\ref{lb21l1}. Note that the following lemma is similar to Lemma 3.6 in~\cite{dk}.

\begin{lemma}\label{lb21l2} Let $P\in(1,+\infty)$, $\bv r\in(\SN\cup\{0\})^k$, $\bv m=\{m_1,\dots,m_k\}$ a permutation of $\{1,\dots,k\}$,
$(\mathcal{J}_1,\dots,\mathcal{J}_k)\in\mathcal{J}$ and $0\le I_1,\dots,I_k\le R_k$ such that $I_s\in\mathcal{R}_i$ for
$s\in\mathcal{J}_i$ and $1\le i\le k$. Assume that $\sigma\in S_k$ is a permutation such that $I_{\sigma(1)}\le\cdots\le I_{\sigma(k)}$. Then
$$\sum_{\bv Y\in\mathcal{P}_{\bv r}}\left(M_{\bv r}(\bv Y;\bv I;\bv
m)\right)^{P-1}\le\prod_{i=1}^k\left((k-i+2)^{r_i}\frac{(t_{i,J_i-i+1})^{R_i}}{(t_{i,J_{i-1}-i+1})^{R_{i-1}}}
\prod_{J_{i-1}<j\le
J_i}\left(\frac{t_{i,j-i}}{t_{i,j-i+1}}\right)^{I_{\sigma(j)}}\right).$$
\end{lemma}

\begin{proof} Set $\sigma(0)=0$, $\sigma(k+1)=k+1$, $I_0=0$ and $I_{k+1}=R_k$.
First, we calculate $M_{\bv r}(\bv Y;\bv I;\bv m)$ for fixed $\bv Y\in\mathcal{P}_{\bv r}$.
Let $$\mathcal{N}_{i,j}=\mathcal{R}_i\cap(I_{\sigma(j)},I_{\sigma(j+1)}]\quad(1\le i\le k,~J_{i-1}\le
j\le J_i)$$ and
$$Y_{s,i,j}=Y_s\cap\mathcal{N}_{i,j},\quad y_{s,i,j}=|Y_{s,i,j}|\quad(0\le s\le k,~1\le i\le k,~J_{i-1}\le j\le J_i),$$ where
$$Y_0=\{1,\dots,R_k\}\setminus\bigcup_{i=1}^kY_i.$$ The $k$-tuple $(Z_1,\dots,Z_k)\in\mathcal{P}_{\bv r}$ is counted by $M_{\bv r}(\bv
Y;\bv I;\bv m)$ when
\be\label{lb2e7a}\bigcup_{s=j}^k\left(Z_{m_s}\cap(I_j,R_k]\right)=\bigcup_{s=j}^k\left(Y_{m_s}\cap(I_j,R_k]\right)\quad(1\le
j\le k).\ee So if we set $$Z_{s,i,j}=Z_s\cap\mathcal{N}_{i,j}\quad(0\le s\le k,~1\le i\le k,~J_{i-1}\le
j\le J_i),$$ where $$Z_0=\{1,\dots,R_k\}\setminus\bigcup_{i=1}^kZ_i,$$ then~\eqref{lb2e7a} can be written as
\be\label{lb2e7}\bigcup_{s=\sigma(t)}^kZ_{m_s,i,j}=\bigcup_{s=\sigma(t)}^kY_{m_s,i,j}\quad(1\le i\le k,~J_{i-1}\le
j\le J_i,~0\le t\le j).\ee For $j\ge0$ let $$\chi_j:\{0,1,\dots,j,j+1\}\to\{\sigma(0),\sigma(1),\dots,\sigma(j),\sigma(k+1)\}$$
be the bijection uniquely determined by the property that $\chi_j(0)<\cdots<\chi_j(j+1)$. So the sequence
$\chi_j(0),\dots,\chi_j(j+1)$ is the sequence $\sigma(0),\dots,\sigma(j),\sigma(k+1)$ ordered increasingly. In
particular, $\chi_j(0)=\sigma(0)=0$ and $\chi_j(j+1)=\sigma(k+1)=k+1$. Note that $Z_{m_s,i,j}=Y_{m_s,i,j}=\emptyset$ if
$1\le m_s<i$, by the definition of $\mathcal{P}_{\bv r}$. So if we set $m_0=0$ and
$$A_{t,i,j}=\{\chi_j(t)\le s<\chi_j(t+1):m_s\ge i~{\rm or}~s=0\}\quad(1\le i\le k,~J_{i-1}\le j\le J_i,~0\le t\le j),$$ then~\eqref{lb2e7} is equivalent to \be\label{lb2e8}\bigcup_{s\in A_{t,i,j}}Z_{m_s,i,j}=\bigcup_{s\in A_{t,i,j}}Y_{m_s,i,j}\quad(0\le
t\le j),\ee for all $1\le i\le k$ and $J_{i-1}\le j\le J_i$. For such a pair $(i,j)$, let $M_{i,j}$ be the set of mutually disjoint
$(k-i+2)$-tuples $(Z_{0,i,j},Z_{i,i,j},Z_{i+1,i,j},\dots,Z_{k,i,j})$ that satisfy \eqref{lb2e8}. Then \be\label{lb2e9}M_{\bv r}(\bv Y;\bv
I;\bv m)=\prod_{\substack{1\le i\le k\\J_{i-1}\le j\le J_i}}M_{i,j}.\ee Moreover, it is immediate from the definition of
$M_{i,j}$ that $$M_{i,j}=\prod_{t=0}^j|A_{t,i,j}|^{\sum_{s\in
A_{t,i,j}}y_{m_s,i,j}}$$ (with the standard notational convention that $0^0=1$). Let
\be\label{lb2e10}W_{t,i,j}=\bigcup_{s\in A_{t,i,j}}Y_{m_s,i,j}\quad{\rm and}\quad w_{t,i,j}=|W_{t,i,j}|\quad(1\le i\le k,~J_{i-1}\le j\le J_i,~0\le t\le j).\ee
With this notation, we have that
$$M_{i,j}=\prod_{t=0}^j|A_{t,i,j}|^{w_{t,i,j}}.$$ Inserting the
above relation into~\eqref{lb2e9}, we deduce that $$M_{\bv r}(\bv Y;\bv I;\bv m)=\prod_{i=1}^k\prod_{j=J_{i-1}}^{J_i}\prod_{t=0}^j|A_{t,i,j}|^{w_{t,i,j}}.$$
Therefore $$S:=\sum_{\bv Y\in\mathcal{P}_{\bv r}}\left(M_{\bv r}(\bv Y;\bv I;\bv m)\right)^{P-1}=\prod_{i=1}^k\prod_{j=J_{i-1}}^{J_i}\sum_{Y_{0,i,j},Y_{i,i,j},\dots,Y_{k,i,j}}\prod_{t=0}^j|A_{t,i,j}|^{(P-1)w_{t,i,j}}.$$
Next, for fixed $i\in\{1,\dots,k\}$, $j\in\{J_{i-1},\dots,J_i\}$ and $W_{0,i,j},\dots,W_{j,i,j}$, a partition of $\mathcal{N}_{i,j}$, the
number of $Y_{0,i,j},Y_{i,i,j},\dots,Y_{k,i,j}$ that satisfy \eqref{lb2e10}~is equal to $$\prod_{t=0}^j|A_{t,i,j}|^{w_{t,i,j}}.$$ Consequently,
\be\label{lb2e11}S=\prod_{i=1}^k\prod_{j=J_{i-1}}^{J_i}\sum_{W_{0,i,j},\dots,W_{j,i,j}}\prod_{t=0}^j|A_{t,i,j}|^{Pw_{t,i,j}}
=\prod_{i=1}^k\prod_{j=J_{i-1}}^{J_i}\left(|A_{0,i,j}|^P+\cdots+|A_{j,i,j}|^P\right)^{|\mathcal{N}_{i,j}|},\ee by the multinomial theorem.
Fix $1\le i\le k$ and $J_{i-1}\le j\le J_I$ and set $$K_{i,j}=\{0\le t\le j:|A_{t,i,j}|\ge1\}.$$
We claim that \be\label{lb2e12}j-i+2\le|K_{i,j}|\le k-i+2.\ee Indeed, we have that
$$\{0\}\cup\{1\le s\le k:m_s\ge i\}=\bigcup_{t\in K_{i,j}}A_{t,i,j}\subset\bigcup_{t\in K_{i,j}}\{s\in\SZ:\chi_j(t)\le s<\chi_j(t+1)\}.$$ The above relation implies that
$$k-i+2=|\{0\}\cup\{1\le s\le k:m_s\ge i\}|=\sum_{t\in K_{i,j}}|A_{t,i,j}|\ge|K_{i,j}|$$ and
\bes\bsp 
	k-i+2&\le\left\lvert\bigcup_{t\in K_{i,j}}\{s\in\SZ:\chi_j(t)\le s<\chi_j(t+1)\}\right\rvert\\
		&=k+1-\left\lvert\bigcup_{t\in\{0,1,\dots,j\}\setminus K_{i,j}} 
			\{s\in\SZ:\chi_j(t)\le s<\chi_j(t+1)\}\right\rvert\le k-j+|K_{i,j}|,
\end{split}\ees 
which together prove \eqref{lb2e12}. Lastly, note that for $n\le X$ we have that
\[
\max\left\{\sum_{j=1}^nx_j^P:\sum_{j=1}^nx_j=X,~x_j\ge 1~(1\le j\le n)\right\}=n-1+(X-n+1)^P,
\]
since the maximum of a convex function in a simplex occurs at its vertices. Therefore
\bes\bsp|A_{0,i,j}|^P+\cdots+|A_{j,i,j}|^P\le|K_{i,j}|-1+(k-i+3-|K_{i,j}|)^P&\le j-i+1+(k-j+1)^P\\
&=(k-i+2)t_{i,j-i+1},\end{split}\ees by \eqref{lb2e12}. Finally, inserting the above inequality into \eqref{lb2e11} yields
\bes\bsp 
	S&\le\prod_{i=1}^k(k-i+2)^{r_i}\prod_{j=J_{i-1}}^{J_i}(t_{i,j-i+1})^{|\mathcal{N}_{i,j}|}\\
	&=\prod_{i=1}^k(k-i+2)^{r_i}(t_{i,J_{i-1}-i+1})^{I_{\sigma(J_{i-1}+1)}-R_{i-1}} \\
	&\qquad \times \left(\prod_{j=J_{i-1}+1}^{J_i-1}(t_{i,j-i+1})^{I_{\sigma(j+1)}-I_{\sigma(j)}}\right)
(t_{i,J_i-i+1})^{R_i-I_{\sigma(J_i)}},
\end{split}\ees 
which completes the proof of the lemma.
\end{proof}

\subsection{Proof of Lemma~\ref{lb1l2}} In this last subsection we combine the results of Subsections~\ref{interpolation} and~\ref{combinatorics}
to show Lemma \ref{lb1l2}.

\begin{proof}[Proof of Lemma \ref{lb1l2}] By Lemmas~\ref{lb21l1} and~\ref{lb21l2} we have that
\be\label{lme1}\begin{split}\sum_{\bv a\in\mathcal{A}(\bv
g)}\frac{W_{k+1}^P(\bv a)}{a_1\cdots a_k}& \ll_k \sum_{\substack{0=J_0\le J_1\le\cdots\le J_k\le k\\J_i\ge i~(1\le i\le k)}}
\sum_{\substack{0\le I_1\le\cdots\le I_k\le R_k\\I_j\in\mathcal{R}_i,~J_{i-1}<j\le J_i\\1\le i\le k}}
\prod_{i=1}^k\left(\frac{((k-i+2)\log(\rho_{k-i+1}))^{r_i}}{g_{i,1}!\cdots g_{i,v_i}!}\right.\\
&\qquad\times\left.\frac{(t_{i,J_i-i+1})^{R_i}}{(\rho_{k-i+1}^{P-1})^{(k-J_i)v_i}(t_{i,J_{i-1}-i+1})^{R_{i-1}}}\prod_{J_{i-1}<j\le J_i}(\rho_{k-i+1}^{P-1})^{-E_{\bv g}(I_j)}\left(\frac{t_{i,j-i}}{t_{i,j-i+1}}\right)^{I_j}\right).\end{split}\ee
Write $e_j=E_{\bv g}(I_j)$ for $i\in\{1,\dots,k\}$ and note that $$0\le e_{J_{i-1}+1}\le\cdots\le e_{J_i}\le v_i\quad(1\le i\le k).$$
Moreover, for $1\le i\le k$ and $J_{i-1}<j\le J_i$ we have that
\[
\sum_{I_j\in\mathcal{R}_i,~E_{\bv g}(I_j)
	=e_j}\left(\frac{t_{i,j-i}}{t_{i,j-i+1}}\right)^{I_j}
	\le\sum_{G_{i,e_j-1}+R_{i-1} \le I_j\le G_{i,e_j}+R_{i-1}}
		\left(\frac{t_{i,j-i}}{t_{i,j-i+1}}\right)^{I_j}
	\ll_{k,P}\left(\frac{t_{i,j-i}}{t_{i,j-i+1}}\right)^{G_{i,e_j}+R_{i-1}},
\] 
since $t_{i,j-i}>t_{i,j-i+1}$. Inserting the above inequality into~\eqref{lme1} completes the proof.
\end{proof}

\section{The lower bound in Theorem~\ref{thm2}: completion of the proof}\label{lb_proof2} In
this section we complete the proof of Theorem~\ref{thm2} by showing Lemmas \ref{lb1l3}\;and \ref{lb1l4}.

\subsection{Preliminaries}\label{prelim} \renewcommand{\labelenumi}{(\alph{enumi})}
We state here some inequalities we will need later. For $0<h\le x$ set $$F(x,h)=\frac{(x+1)\log(x+1)-(x-h+1)\log(x-h+1)}h.$$ 

\begin{lemma}\label{lb22l0b} The function $F$ has the following properties:
\begin{enumerate}\item For $0<h\le x$ we have
$$\frac{\partial F(x,h)}{\partial
h}<0,\quad\frac{\partial F(x,h)}{\partial x}>0\quad\text{and}\quad\frac{\partial(F(x,x))}{\partial x}>0.$$
\item For $0<h\le x-1$ we have $$F(x,h)>F(x-h,1).$$
\end{enumerate}
\end{lemma}

\begin{proof} (a) We have that $$\frac{\partial F(x,h)}{\partial h}=\frac1{h^2}\left[h+(x+1)\log\left(1-\frac
h{x+1}\right)\right]<0\quad(0<h\le x).$$ Also, $$\frac{\partial F(x,h)}{\partial x}=\frac1h\log\left(\frac{x+1}{x-h+1}\right)>0\quad(0<h\le x).$$ Finally, $$\frac{\partial(F(x,x))}{\partial x}=\frac{x-\log(x+1)}{x^2}>0\quad(x>0).$$

\medskip

(b) Fix $x>1$ and note that it suffices to show that
$$g(h)=(x+1)\log(x+1)-(h+1)(x-h+1)\log(x-h+1)+h(x-h)\log(x-h)>0$$ for $0<h\le x-1$.
Since $g(0)=0$, it is enough to show that $g'(h)>0$. We have that
$$g'(h)=1+(2h-x)\log\left(\frac{x-h+1}{x-h}\right).$$
If $h\ge x/2$, then $g'(h)\ge1$. If $0<h<x/2$, then
$$g'(h)>1-\frac{x-2h}{x-h}=\frac h{x-h}>0.$$ In any case, we have that $g'(h)>0$,
which completes the proof of the lemma.
\end{proof}

Finally, we have the following lemma.

\begin{lemma}\label{lb22l0a} The sequence
$$\left\{1-\frac1{\log(n+2)}\log\left(\frac{(n+2)\log(n+2)-\log4}n\right)\right\}_{n\in\SN}$$
is strictly increasing.
\end{lemma}

\begin{proof} For $x>0$ set $$g(x)=\frac{(x+2)\log(x+2)-\log4}x\quad{\rm and}\quad G(x)=1-\frac{\log(g(x))}{\log(x+2)}.$$
First, we check numerically that $G(1)<G(2)<\cdots<G(14)$. Next, we
handle the larger terms of the sequence. We have $$G'(x)=\frac{h(x)}{x(x+2)[(x+2)\log(x+2)-\log4]\log^2(x+2)},$$
where $$h(x)=x[(x+2)\log(x+2)-\log4]\log(g(x))-(x+2)\log(x+2)(x-2\log(x+2)+\log 4).$$ Observe that for
$x\ge14$ we have $\log(g(x))\ge\log\log(x+2)\ge1$. Consequently,
\bes\bsp h(x)&\ge-x\log 4+2(x+2)\log^2(x+2)-(\log4)(x+2)\log(x+2)\\
&\ge(-x+3(x+2)\log(x+2))\log 4>0\end{split}\ees for all $x\ge14$, that
is $G'(x)>0$ for $x\ge14$ and the desired result follows.
\end{proof}

\subsection{Estimates from order statistics}\label{order_stats}

Throughout this subsection we fix a vector $\bv r\in\mathcal{R}^*$. Our goal is to bound on average the quantities $T_i(\bv g_i;\nu,n)$, which were defined in Section~\ref{lb_outline}, on average. To achieve this, we appeal to certain estimates from probability theory proven by Ford in~\cite{kf3}. Recall that
$$\Delta_r=\{\bv\xi=(\xi_1,\dots,\xi_r)\in\SR^r:0\le\xi_1\le\cdots\le\xi_r\le1\}.$$ For $r\in\SN$, $u>0$ and $v\ge1$ set
\bes\bsp Q_r(u,v)&=r!\vol\left(\left\{\bv\xi\in\Delta_r:\xi_i\ge\frac{i-u}v\quad(1\le i\le r)\right\}\right)\\
&=\mathbf{Prob}\left(\xi_i\ge\frac{i-u}v~(1\le i\le r)\,\bigg|\,\bv\xi\in\Delta_r\right).\end{split}\ees
Then we have the following estimate, which essentially follows from Theorem 1 in~\cite{kf3}. This estimate is stated in~\cite{kf4} too without proof. For the sake of completeness, we supply the details of its proof.

\begin{lemma}\label{lb22l1} Let $r\in\SN$, $u\ge1$ and $v\ge1$. If $w=u+v-r\ge1$, then $$Q_r(u,v)\asymp\min\left\{1,\frac{uw}r\right\}.$$
\end{lemma}

\begin{proof} The desired upper bound follows immediately by Theorem 1 in \cite{kf3} and the trivial bound $Q_r(u,v)\le1$. For the lower bound we distinguish several cases. First, assume that $v\ge2r$. Then $$Q_r(u,v)\ge Q_r(1,v)=\frac{1+v-r}v\left(1+\frac1v\right)^{r-1}\asymp1=\min\left\{1,\frac{uw}r\right\},$$ by~\cite[Lemma 2.1(i)]{kf3}. Next, consider the case $v\le2r$. Set $$u'=\min\left\{u,\frac{r-v+\sqrt{(r-v)^2+4r}}2\right\}\ge\frac12$$ and $$w'=u'+v-r=\min\left\{w,\frac{v-r+\sqrt{(r-v)^2+4r}}2\right\}\ge\frac12$$ By a similar argument with the one leading to~\eqref{betai}, we have that $$\min\left\{1,\frac{uw}r\right\}=\frac{u'w'}r.$$ Fix some constant $C$. If $u'\ge C$ and $w'\ge C$, then the lower bound follows by Theorem 1 in~\cite{kf3} applied to $Q_r(u',v)\le Q_r(u,v)$, provided that $C$ is large enough. If $1/2\le u'\le w'$ and $u'\le C$, then $r\le v\le2r$ and thus $$Q_r(u,v)\ge Q_r(1,v)=\frac{1+v-r}v\left(1+\frac1v\right)^{r-1}\asymp_C\frac{u'+v-r}r\asymp_C\frac{u'w'}r,$$ by~\cite[Lemma 2.1(i)]{kf3}. Finally, if $1/2\le w'\le u'$ and $w'\le C$, then $v\le r$ and thus $$Q_r(u,v)\ge Q_r(1+r-v,v)\gg\frac{1+r-v}r\asymp_C\frac{w'+r-v}r\asymp_C\frac{u'w'}r,$$ by~\cite[Lemma 11.1]{kf2}. In any case, we obtain the desired result.
\end{proof}

For $r,v\in\SN$ and $u\ge0$ set $$\mathcal{G}_r(u,v)=\{(g_1,\dots,g_v)\in(\SN\cup\{0\})^v:g_1+\cdots+g_v=r,\;g_1+\cdots+g_i\le i+u\;(1\le i\le v)\}.$$
Then an equivalent formulation of Lemma \ref{lb22l1}~is the following result.

\begin{lemma}\label{lb22l2} Let $r\in\SN$, $v\in\SN$ and $u\ge0$. If $w=u+v-r\ge0$, then
$$\sum_{\bv g\in\mathcal{G}_r(u,v)}\frac1{g_1!\cdots
g_v!}\asymp\frac{v^r}{r!}\min\left\{1,\frac{(u+1)(w+1)}r\right\}.$$
\end{lemma}

\begin{proof} For every $\bv g\in\mathcal{G}_r(u,v)$, let $R(\bv g)$ be the set of
$\bv\xi\in\Delta_r$ such that, for any $i\in\{1,\dots,v\}$, exactly $g_i$ of the numbers $\xi_j$
lie in $[(i-1)/v,i/v)$. Then
\be\label{lb22e1}\vol(R(\bv g))=\frac1{v^r}\frac1{g_1!\cdots
g_v!}.\ee Also, we have that $g_1+\cdots+g_i\le i+u$ if, and only
if, $\xi_{i+\lfloor u+1\rfloor}\ge\frac iv$. Hence summing
\eqref{lb22e1}\;over $\bv g\in\mathcal{G}_r(u,v)$ and applying Lemma
\ref{lb22l1}\;completes the proof.
\end{proof}


\begin{lemma}\label{lb22l3} Let $\bv r\in\mathcal{R}^*$. Consider integers $1\le i\le k$, $\nu\ge0$ and $n\ge1$ with
$\nu+n\le k-i+1$ and $\eta\in(0,1]$. There exists a constant $c_k'>0$ such that the following hold:
\begin{enumerate}\item If $$\left\lvert\frac{F(k-i+1-\nu,n)}{(k-i+2)^{1-\alpha}}-1\right\rvert\ge\eta,$$ $P\le1+\eta/c_k'$ and $v_i\ge(c_k'/\eta)^2$, then
$$\sum_{\bv g_i\in\mathcal{G}_{r_i}(u_i,v_i)}\frac{T_i(\bv
g_i;\nu,n)}{g_{i,1}!\cdots g_{i,v_i}!}\ll_{k,P,\eta}\beta_i\frac{v_i^{r_i}}{r_i!}\max_{j\in\{0,n\}}(\rho_{k-i+1}^{P-1})^{-(n-j)v_i}(t_{i,\nu+j})^{r_i}.$$
\item If $P\le1+1/c_k'$ and $v_i\ge c_k'$, then
$$\sum_{\bv g_i\in\mathcal{G}_{r_i}(u_i,v_i)}\frac{T_i(\bv
g_i;\nu,n)}{g_{i,1}!\cdots
g_{i,v_i}!}\ll_{k,P}\beta_i\frac{v_i^{r_i}e^{O_k((P-1)^2v_i)}}{r_i!}\max_{j\in\{0,n\}}(\rho_{k-i+1}^{P-1})^{-(n-j)v_i}(t_{i,\nu+j})^{r_i}.$$
\end{enumerate}
\end{lemma}

\begin{proof} For now, we treat parts (a) and (b) together. Their proofs will deviate only towards the end. Set 
\bes\bsp 
S  &=\sum_{\bv g_i\in\mathcal{G}_{r_i}(u_i,v_i)}\frac{T_i(\bv g_i;\nu,n)}{g_{i,1}!\cdots g_{i,v_i}!}\\
	&=\sum_{0=s_0\le s_1\le\cdots\le s_n\le s_{n+1}=v_i}
		\sum_{\bv g_i\in\mathcal{G}_{r_i}(u_i,v_i)}\frac1{g_{i,1}!\cdots g_{i,v_i}!}
		\prod_{j=0}^n(t_{i,\nu+j})^{G_{i,s_{j+1}}-G_{i,s_j}}.
\end{split}\ees
Fix $0=s_0\le s_1\le\cdots\le s_n\le s_{n+1}=v_i$ and let $m_1,\dots,m_{n+1}$ be non-negative integers with
$$M_j=m_1+\cdots+m_j\le s_j+u_i\quad(1\le j\le n),\quad M_{n+1}=m_1+\cdots+m_{n+1}=r_i.$$ Also, put $M_0=0$.
Then we have that
\bes\begin{split}&\sum_{\substack{\bv g_i\in\mathcal{G}_{r_i}(u_i,v_i)\\G_{i,s_j}=M_j\\1\le j\le
n}}\frac1{g_{i,1}!\cdots g_{i,v_i}!}=\prod_{j=0}^n\sum_{\substack{(g_{i,s_j+1},\dots,g_{i,s_{j+1}})\\\in\mathcal{G}_{m_{j+1}}(u_i+s_j-M_j,s_{j+1}-s_j)}}
\frac1{g_{i,s_j+1}!\cdots g_{i,s_{j+1}}!}\\
&\quad\ll\min\left\{\frac{u_i(u_i+s_1-m_1+1)}{m_1+1},\frac{w_i(u_i+s_n-M_n+1)}{m_{n+1}+1}\right\}
\prod_{j=0}^n\frac{(s_{j+1}-s_j)^{m_{j+1}}}{m_{j+1}!},\end{split}\ees
by Lemma~\ref{lb22l2} applied for $j=0$ and $j=n$. Also, note that
\bes\bsp\frac{u_i(u_i+s_1-m_1+1)}{m_1+1}\le\frac{u_i(w_i+r_i-m_1+1)}{m_1+1}&\le\frac{u_i(w_i+1)(r_i-m_1+1)}{m_1+1}\\
&\ll\beta_i\frac{r_i(r_i-m_1+1)}{m_1+1}\end{split}\ees and, similarly,
$$\frac{w_i(u_i+s_n-M_n+1)}{m_{n+1}+1}\le\frac{w_i(u_i+1)(s_n+1)}{m_{n+1}+1}\ll\beta_i\frac{r_i(s_n+1)}{m_{n+1}+1}.$$ So
\be\label{lb22e5a}\begin{split}S&\ll\beta_ir_i\sum_{0=s_0\le
s_1\le\cdots\le s_n\le
s_{n+1}=v_i}(\rho_{k-i+1}^{P-1})^{-(s_1+\cdots+s_n)}\\
&\quad\quad\times\sum_{m_1+\cdots+m_{n+1}=r_i}\min\left\{\frac{r_i-m_1+1}{m_1+1},\frac{s_n+1}{m_{n+1}+1}\right\}
\prod_{j=0}^n\frac{(t_{i,\nu+j}(s_{j+1}-s_j))^{m_{j+1}}}{m_{j+1}!}\end{split}\ee
The inner sum in the right hand side of~\eqref{lb22e5a} satisfies the following two upper bounds: it is at most
\bes\begin{split}\sum_{m_1=0}^{r_i}&\frac{r_i-m_1+1}{m_1+1}\frac{(t_{i,\nu}s_1)^{m_1}}{m_1!}
\frac{\left(\sum_{j=1}^nt_{i,\nu+j}(s_{j+1}-s_j)\right)^{r_i-m_1}}{(r_i-m_1)!}\\
&\ll_k\frac1{r_i!}\frac{v_i-s_1+1}{s_1+1}\left(\sum_{j=0}^nt_{i,\nu+j}(s_{j+1}-s_j)\right)^{r_i}\end{split}\ees
and, also, it is at most
\bes\begin{split}(s_n+1)\sum_{m_{n+1}=0}^{r_i}&\frac{(t_{i,\nu+n}(s_{n+1}-s_n))^{m_{n+1}}}{(m_{n+1}+1)!}
\frac{\left(\sum_{j=0}^{n-1}t_{i,\nu+j}(s_{j+1}-s_j)\right)^{r_i-m_{n+1}}}{(r_i-m_{n+1})!}\\
&\ll_k\frac1{r_i!}\frac{s_n+1}{v_i-s_n+1}\left(\sum_{j=0}^nt_{i,\nu+j}(s_{j+1}-s_j)\right)^{r_i}.\end{split}\ees
Consequently, 
\be\label{lb22e6}\begin{split}
	S&\ll_k\beta_i\frac{v_i^{r_i+1}}{r_i!}
		\sum_{0\le s_1\le\cdots\le s_n\le v_i}(\rho_{k-i+1}^{P-1})^{-(s_1+\cdots+s_n)}
			\min\left\{\frac{v_i-s_1+1}{s_1+1},\frac{s_n+1}{v_i-s_n+1}\right\} \\
	&\qquad\qquad\qquad\qquad\qquad
		\times\left(\sum_{j=1}^n(t_{i,\nu+j-1}-t_{i,\nu+j})\frac{s_j}{v_i}+t_{i,\nu+n}\right)^{r_i}\\
	&=\beta_i\frac{v_i^{r_i+1}}{r_i!}\sum_{0\le s_1\le\cdots\le s_n\le v_i}g(s_1,s_n)
		\exp\{G(s_1,\dots,s_n)\},
\end{split}\ee 
where for $\bv x=(x_1,\dots,x_n)\in[0,+\infty)^n$ we have set
$$
G(\bv x)=\log\left((\rho_{k-i+1}^{P-1})^{-(x_1+\cdots+x_n)}
	\left(\sum_{j=1}^n(t_{i,\nu+j-1}-t_{i,\nu+j})\frac{x_j}{v_i}+t_{i,\nu+n}\right)^{r_i}\right)
$$ 
and for $(x,y)\in[0,+\infty)^2$ we have set 
$$
g(x,y)=\min\left\{\frac{v_i-x+1}{x+1},\frac{y+1}{v_i-y+1}\right\}.
$$
We claim that \be\label{lb22e7}S\ll_{k,P}\beta_i\frac{v_i^{r_i+1}}{r_i!}\sum_{0\le
s\le v_i}g(s,s)\exp\{G(s,\dots,s)\}.\ee To show~\eqref{lb22e7} we will make extensive use of the following simple fact:
if $b:[m,m+1]\to\SR$ is a differentiable function satisfying $b'(x)\ge\delta>0$ for all $x\in(m,m+1)$, where $\delta$ is a fixed positive number, then
\be\label{lb22obs1}\frac{e^{b(m+1)}}{e^{b(m)}}\ge e^\delta,\ee by the Mean Value Theorem.
Fix a small positive constant $\eta_0=\eta_0(k)$ to be chosen later and define $J\in\{0,1,\dots,n-1\}$ as
follows. If 
$$
\frac{F(k-i+1-\nu,1)}{(k-i+2)^{1-\alpha}}<1+\eta_0,
$$ 
then set $J=0$; else, put
$$
J=\max\left\{1\le j\le n-1:\frac{F(k-i+1-\nu,j)}{(k-i+2)^{1-\alpha}}\ge1+\eta_0\right\}.
$$
Observe that 
$$
t_{i,j}=1+\frac{(k-i+2-j)\log(k-i+2-j)}{k-i+2}(P-1)+O_k((P-1)^2)~\quad(0\le j\le k-i+1).
$$
Therefore if $1\le j\le J$, then Lemma~\ref{lb22l0b}(a) yields that 
\be\begin{split}\label{lb22e110}
	\frac{\partial}{\partial x_j}
		&(G(\underbrace{x_j,\dots,x_j}_{j\ \text{times}},x_{j+1},x_{j+2},\dots,x_n))\\
		&=-j(P-1)\log(\rho_{k-i+1})+\frac{r_i(t_{i,\nu}-t_{i,\nu+j})}{(t_{i,\nu}-t_{i,\nu+j})x_j
			+\sum_{m=j+1}^n(t_{i,\nu+m-1}-t_{i,\nu+m})x_m+t_{i,\nu+n}v_i}  \\
		&=j(P-1)\log(\rho_{k-i+1})\left(-1+\frac{F(k-i+1-\nu,j)}{(k-i+2)^{1-\alpha}}
			+O_k\left(P-1+v_i^{-1/2}\right)\right)  \\
		&\ge  \frac{\eta_0(P-1)j\log(\rho_{k-i+1}) }2>0
\end{split}\ee 
uniformly in $0\le x_j\le\cdots\le x_n\le v_i$, provided that $c_k'$ is large enough. So if $J\ge1$ and we fix $0\le s_2\le\cdots\le s_n$, then we have that
$$\sum_{0\le s_1\le s_2}g(s_1,s_n)\exp\{G(s_1,\dots,s_n)\}
\ll_{k,P}g(s_2,s_n)\exp\{G(s_2,s_2,s_3,\dots,s_n)\},$$ by~\eqref{lb22e110} with $j=1$, and \eqref{lb22obs1}. Similarly, if $J\ge2$ and we fix $0\le s_3\le\cdots\le s_n\le v_i$, then
$$\sum_{0\le s_2\le s_3}g(s_2,s_n)\exp\{G(s_2,s_2,s_3,\dots,s_n)\}\ll_{k,P}g(s_3,s_n)\exp\{G(s_3,s_3,s_3,s_4,\dots,s_n)\}.$$
Continuing in the above fashion, we deduce that
\be\label{lb22e9}\begin{split}&\sum_{0\le s_1\le\cdots\le s_n\le
v_i}g(s_1,s_n)\exp\{G(s_1,\dots,s_n)\}\\
&\quad\ll_{k,P}\sum_{0\le s_{J+1}\le\cdots\le s_n\le
v_i}g(s_{J+1},s_n)\exp\{G(\underbrace{s_{J+1},\dots,s_{J+1}}_{J+1~{\rm times}},s_{J+2},\dots,s_n)\},\end{split}\ee which also holds trivially if $J=0$. If, now, $J=n-1$, then \eqref{lb22e7}\;follows immediately by~\eqref{lb22e9}. So assume that $J<n-1$. Then Lemma
\ref{lb22l0b}(b) implies that
$$F(k-i-\nu-J,1)<F(k-i+1-\nu,J+1)\le1+\eta_0$$
and hence $$F(k-i+2-\nu-j,1)\le F(k-i-\nu-J,1)\le1-\eta_0\quad(J+2\le
j\le n),$$ provided that $2\eta_0\le
F(k-i+1-\nu,J+1)-F(k-i-\nu-J,1)$. Consequently,
\be\begin{split}\label{lb22e220}
	\frac{\partial G}{\partial x_j}(\bv x)
		&=(P-1)\log(\rho_{k-i+1}) \left(-1+\frac{F(k-i+2-\nu-j,1)}{(k-i+2)^{1-\alpha}}
			+O_k\left(P-1+v_i^{-1/2}\right)\right)\\
		&\le-\frac{\eta_0(P-1)\log(\rho_{k-i+1})}2<0\quad(J+2\le j\le n)
\end{split}\ee 
uniformly in $0\le x_1\le\cdots\le x_n\le v_i$, provided that $c_k'$ is large enough. Thus, if we fix $s_{J+1}\ge0$ and $v_i\ge s_n\ge s_{n-1}\ge\cdots\ge s_{J+3}\ge s_{J+1}$, then we find that 
$$
\sum_{s_{J+1}\le s_{J+2}\le
s_{J+3}}\exp\{G(\underbrace{s_{J+1},\dots,s_{J+1}}_{J+1~{\rm times}},s_{J+2},\dots,s_n)\}
	\ll_{k,P}\exp\{G(\underbrace{s_{J+1},\dots,s_{J+1}}_{J+2~{\rm times}},s_{J+3},\dots,s_n)\},
$$ 
by~\eqref{lb22e220} with $j=J+2$, and \eqref{lb22obs1}. Similarly, if we fix $s_{J+1}\ge0$ and $v_i\ge s_n\ge s_{n-1}\ge\cdots\ge s_{J+4}\ge s_{J+1}$, then we have $$\sum_{s_{J+1}\le s_{J+3}\le s_{J+4}}\exp\{G(\underbrace{s_{J+1},\dots,s_{J+1}}_{J+2~{\rm times}},s_{J+3},\dots,s_n)\}\ll_{k,P}\exp\{G(\underbrace{s_{J+1},\dots,s_{J+1}}_{J+3~{\rm times}},s_{J+4},\dots,s_n)\}.$$ Continuing in this fashion, we deduce that
\bes\begin{split}\sum_{0\le s_{J+1}\le\cdots\le s_n\le v_i}&g(s_{J+1},s_n)\exp\{G(\underbrace{s_{J+1},\dots,s_{J+1}}_{J+1~{\rm times}},s_{J+2},\dots,s_n)\}\\
\ll_{k,P}&\sum_{0\le s_{J+1}\le
v_i}g(s_{J+1},s_{J+1})\exp\{G(s_{J+1},\dots,s_{J+1})\}\end{split}\ees which, together with~\eqref{lb22e9} and~\eqref{lb22e6}, proves~\eqref{lb22e7} in this case too. Finally, we use~\eqref{lb22e7} to prove parts (a) and (b).

\medskip

(a) First, assume that $$\frac{F(k-i+1-\nu,n)}{(k-i+2)^{1-\alpha}}\ge1+\eta.$$ Note that 
\bes\bsp
	\frac{\partial}{\partial x}(G(x,\dots,x))
		&=n(P-1)\log(\rho_{k-i+1})\left(-1+\frac{F(k-i+1-\nu,n)}{(k-i+2)^{1-\alpha}}
			+O_k(P-1+v_i^{-1/2})\right)\\
		&\ge\frac{\eta(P-1)n\log(\rho_{k-i+1})}2>0
\end{split}\ees 
uniformly in $0\le x\le v_i$, provided that $c_k'$ is large enough. Hence
$$
\sum_{0\le s\le v_i}g(s,s)\exp\{G(s,\dots,s)\}
	\ll_{k,\eta,P}g(v_i,v_i)\exp\{G(v_i,\dots,v_i)\}=\frac{\exp\{G(v_i,\dots,v_i)\}}{v_i+1},
$$ 
by \eqref{lb22obs1}, which together with~\eqref{lb22e7} yields the desired result. Similarly, if $$\frac{F(k-i+1-\nu,n)}{(k-i+2)^{1-\alpha}}\le1-\eta,$$ then we find that 
$$
\frac{\partial}{\partial x}(G(x,\dots,x))\le-\frac{\eta(P-1)n\log(\rho_{k-i+1})}2<0,
$$ 
uniformly in $0\le x\le v_i$, and therefore
$$
\sum_{0\le s\le v_i}g(s,s)\exp\{G(s,\dots,s)\}\ll_{k,\eta,P}\frac{\exp\{G(0,\dots,0)\}}{v_i}.
$$ 
Inserting this estimate into~\eqref{lb22e7} gives us the desired result in this case as well.

\medskip

(b) By~\eqref{lb22e7}, we have that \be\label{lb22e100}S\ll_{k,P}\beta_i\frac{v_i^{r_i+2}}{r_i!}\max_{0\le s\le v_i}e^{G(s,\dots,s)}\ll_{k,P}\beta_i\frac{e^{O_k((P-1)^2v_i)}v_i^{r_i}}{r_i!}\max_{0\le s\le v_i}e^{G(s,\dots,s)}.\ee
Since $t_{i,j}=1+O_k(P-1)$ for all $j$, we find that, for any $0\le s\le v_i$, we have that \bes\bsp\log\left((t_{i,\nu}-t_{i,\nu+n})\frac{s}{v_i}+t_{i,\nu+n}\right)&=(t_{i,\nu}-t_{i,\nu+n})\frac{s}{v_i}+t_{i,\nu+n}-1+O_k((P-1)^2)\\
&=\frac{s}{v_i}\log\left(\frac{t_{i,\nu}}{t_{i,\nu+n}}\right)+\log(t_{i,\nu+n})+O_k((P-1)^2)\end{split}\ees and, consequently,
\bes\bsp
	\max_{0\le s\le v_i}G(s,\dots,s)
		=\max\{r_i\log(t_{i,\nu+n}),-(P-1)nv_i\log(\rho_{k-i+1})+r_i\log(t_{i,\nu})\} \\
			+O_k((P-1)^2v_i).
\end{split}\ees 
Inserting the above estimate into~\eqref{lb22e100} completes the proof of the lemma.
\end{proof}

The proof of the next lemma uses some ideas from the proof of~\cite[Lemmas 4.8 and 11.1]{kf2} and~\cite[Lemma 3.8]{dk}.

\begin{lemma}\label{lb22l4} Let $\bv r\in\mathcal{R}^*$ and $i\in\{1,\dots,k\}$. There is a constant $c_k''>0$ such that
$$\sum_{\bv g_i\in\mathcal{G}_{r_i}(u_i,v_i)}\frac{T_i(\bv g_i;0,k-i+1)}{g_{i,1}!\cdots
g_{i,v_i}!}\ll_{k,P}\beta_i\frac{v_i^{r_i}}{r_i!}\prod_{j=1}^{i-1}(k-j+2)^{(P-1)(v_j-r_j)},$$ provided that $P\le1+1/c_k''$.
\end{lemma}

\begin{proof} Since $t_{i,0}>\cdots>t_{i,k-i}>t_{i,k-i+1}=1$, we have that
$$T_i(\bv g_i;0,k-i+1)=\sum_{0\le
s_1\le\cdots\le s_{k-i+1}\le v_i}\prod_{j=1}^{k-i+1}(\rho_{k-i+1}^{P-1})^{-s_j}\left(\frac{t_{i,j-1}}{t_{i,j}}\right)^{G_{i,s_j}}.$$
Also, \bes\prod_{j=1}^{k-i+1}\left(\frac{t_{i,j-1}}{t_{i,j}}\right)^{G_{i,s_j}}
\le\left(\frac{t_{i,0}}{t_{i,1}}\right)^{G_{i,s_1}}\prod_{j=2}^{k-i+1}\left(\frac{t_{i,j-1}}{t_{i,j}}\right)^{s_j+u_i}
=\left(\frac{t_{i,0}}{t_{i,1}}\right)^{G_{i,s_1}}(t_{i,1})^{s_1+u_i}\prod_{j=2}^{k-i+1}(t_{i,j-1})^{s_j-s_{j-1}}.\ees
Thus, by setting
$$
\lambda=\frac{t_{i,0}}{t_{i,1}}=\frac{(\rho_{k-i+1}^{P-1})^{k-i+1}}{t_{i,1}},
$$
$m_1=s_1$ and $m_j=s_j-s_{j-1}$ for $j=2,\dots,k-i+1$, we deduce
that 
\be\label{lb22e13a}
T_i(\bv g_i;0,k-i+1)\le(t_{i,1})^{u_i}
	\sum_{\substack{m_1+\cdots+m_{k-i+1}\le v_i\\m_j\ge 0\;(1\le j\le k-i+1)}}
		\lambda^{G_{i,m_1}-m_1}
		\prod_{j=2}^{k-i+1}\left(\frac{t_{i,j-1}}{(\rho_{k-i+1}^{P-1})^{k-i+2-j}}\right)^{m_j}.
\ee 
Note that
\bes\bsp
\log(t_{i,j-1})&=(P-1)\frac{(k-i-j+3)\log(k-i-j+3)}{k-i+2}+O_k((P-1)^2)\\
&<(P-1)(k-i-j+2)\log(\rho_{k-i+1})\quad(2\le j\le k-i+1),
\end{split}\ees 
provided that $P-1$ is small enough, by Lemma~\ref{lb22l0b}(a). Combining the above relation with~\eqref{lb22e13a} and summing the resulting inequality over $\bv g_i\in\mathcal{G}_{r_i}(u_i,v_i)$, we find that
\be\label{lb22e13}
\sum_{\bv g_i\in\mathcal{G}_{r_i}(u,v_i)}\frac{T_i(\bv g_i;0,k-i+1)}{g_{i,1}!\cdots g_{i,v_i}!}	
	\ll_{k,P}(t_{i,1})^{u_i}\sum_{\bv g_i\in\mathcal{G}_{r_i}(u,v_i)}\frac1{g_{i,1}!\cdots g_{i,v_i}!}
		\sum_{m=0}^{v_i}\lambda^{G_{i,m}-m}=:(t_{i,1})^{u_i}T.
\ee Next, we claim that 
\be\label{lb22claim}
T\le\frac{v_i^{r_i}}{1-1/\lambda}\int_{\mathcal{D}}\left(1+\sum_{j=1}^{r_i}\lambda^{j-v_i\xi_j}\right)d\bv\xi,
\ee 
where
$$
\mathcal{D}=\left\{\bv\xi\in\Delta_{r_i}:\xi_j\ge\frac{j-\lfloor u_i+1\rfloor}{v_i}\quad(1\le
j\le r_i)\right\}.
$$ 
To see this, fix $\bv g_i\in\mathcal{G}_{r_i}(u_i,v_i)$ and consider the set $I(\bv g_i)$ of vectors $\bv\xi\in\Delta_{r_i}$ such that $$|\{1\le j\le r_i:s-1\le v_i\xi_j<s\}|=g_{i,s}\quad(1\le s\le v_i).$$ Notice that if $\bv\xi\in I(\bv g_i)$, then $v_i\xi_{j+\lfloor u_i+1\rfloor}\ge j$ for all $j$ because $\bv g_i \in \mathcal{G}_{r_i}(u_i,v_i)$, that is to say, $\bv\xi\in\mathcal{D}$. Moreover, \bes\bsp
\frac1{1-1/\lambda}\sum_{j=1}^{r_i}\lambda^{j-v_i\xi_j}
		&\ge\frac1{1-1/\lambda}\sum_{s=1}^{v_i}\lambda^{-s}\sum_{j:~v_i\xi_j\in[s-1,s)}\lambda^j \\
		&\ge\sum_{s=1}^{v_i}\sum_{m=s}^{v_i}\lambda^{-m}\sum_{j:~v_i\xi_j\in[s-1,s)}\lambda^j\\
		&=\sum_{m=1}^{v_i}\lambda^{-m}\sum_{j:~v_i\xi_j<m}\lambda^j \\
		&\ge\sum_{\substack{1\le m\le v_i\\G_{i,m}>0}}\lambda^{-m+G_{i,m}} \\
		&\ge-\frac1{1-1/\lambda}+\sum_{m=0}^{v_i}\lambda^{-m+G_{i,m}}.
\end{split}\ees
Lastly, we have that $$\vol(I(\bv g_i))=\frac1{v_i^{r_i}}\frac1{g_{i,1}\cdots g_{i,v_i}}.$$ Combining the above remarks,~\eqref{lb22claim} follows.
To bound the integral in the right hand side of~\eqref{lb22claim}, we proceed as in the proof of Lemma 4.9 in~\cite{kf2}. The only difference is that we use Lemma~\ref{lb22l1} from this paper in place of \cite[Lemma 11.1]{kf2}. This method gives us
$$T\ll_{k,P}\beta_i\frac{v_i^{r_i}}{r_i!}\lambda^{u_i}.$$ By the above
estimate, \eqref{lb22e13}~and~\eqref{lb22claim}, we deduce that
$$\sum_{\bv g_i\in\mathcal{G}_{r_i}(u_i,v_i)}\frac{T_i(\bv g_i;0,k-i+1)}{g_{i,1}!\cdots
g_{i,v_i}!}\ll_k\beta_i\frac{v_i^{r_i}}{r_i!}(k-i+2)^{(P-1)u_i}.$$
To complete the proof of the lemma, recall that
$$u_i\le1+\frac1{\log(k-i+2)}\sum_{j=1}^{i-1}\log(k-j+2)(v_j-r_j).$$
\end{proof}


\subsection{Proof of Lemmas~\ref{lb1l3} and~\ref{lb1l4}} In this subsection we establish Lemmas~\ref{lb1l3} and~\ref{lb1l4}
thus completing all the steps in the proof of the lower bound implicit in Theorem~\ref{thm2}.

\begin{proof}[Proof of Lemma \ref{lb1l3}] Since $g_{1,j}\le G_{1,j}\le j+u_1\le j+1$
for all $j\in\{1,\dots,v_1\}$, we have that
\be\label{lb22e2}\begin{split}
\sum_{a_1\in\mathcal{A}_1(\bv g_1)}\frac1{a_1}
	&=\prod_{j=N}^{v_1}\frac1{g_{1,j}!}\left(\sum_{p_1\in D_{1,j}} \frac1{p_1}
		\sum_{\substack{p_2\in D_{1,j}\\p_2\neq p_1}} \frac1{p_2}
			\cdots\sum_{\substack{p_{g_{1,j}}\in D_{1,j}\\p_{g_{1,j}}\notin\{p_1,...,p_{g_{1,j}-1}\}}}
				\frac1{p_{g_{1,j}}}\right)\\
&\ge\frac1{g_{1,1}!\cdots g_{1,v_1}!}\prod_{j=N}^{v_1}\left(\log\rho_k-\frac{g_{1,j}}{\lambda_{1,j-1}}\right)^{g_{1,j}}\\
&\ge\prod_{i=1}^k\frac{(\log\rho_k)^{r_1}}{g_{1,1}!\cdots g_{1,v_1}!}
	\prod_{j=N}^{v_1}\left(1-\frac{j+1}{(\log\rho_k)\exp\left\{\rho_k^{j-L_k-1}\right\}}\right)^{j+1}\\
	&\ge\frac12\frac{(\log\rho_k)^{r_1}}{g_{1,N}!\cdots g_{1,v_1}!},
\end{split}\ee 
by Lemma~\ref{lambda}, provided that $N$ is large. Similarly, if $i\in\{2,\dots,k\}$, then we have that 
$g_{i,j}\le G_{i,j}\le j+u_i\le j+c\log\log y_{i-1}$ for some $c=c(k)$. Therefore
\begin{align}
\sum_{a_i\in\mathcal{A}_i(\bv g_i)}\frac1{a_i}
	&\ge\frac{(\log(\rho_{k-i+1}))^{r_i}}{g_{i,1}!\cdots g_{i,v_i}!}
		\prod_{j=1}^{v_i}\left(1-\frac{j+c\log\log y_{i-1}}{(\log(\rho_{k-i+1}))(\log y_{i-1})
		\exp\left\{\rho_{k-i+1}^{j-L_k-1}\right\}}\right)^{j+c\log\log y_{i-1}}  \nonumber \\
	&\ge\frac12\prod_{i=1}^k\frac{(\log(\rho_{k-i+1}))^{r_i}}{g_{i,N}!\cdots g_{i,v_i}!}, \label{lb22e3}
\end{align}
provided that $y_{i-1}\ge y_1\ge C_k'$ is large enough. Combine~\eqref{lb22e2} and~\eqref{lb22e3} with Lemmas~\ref{lambda} and~\ref{lb22l2} and relation~\eqref{beta2} to complete the proof.
\end{proof}


\begin{proof}[Proof of Lemma \ref{lb1l4}]Fix $\bv
r\in\mathcal{R}^*$. In view of Lemmas~\ref{lambda} and~\ref{lb1l2}, it suffices to
show that
\be\label{lb22e16}\prod_{i=1}^k(\rho_{k-i+1}^{P-1})^{-(k-J_i)v_i}\sum_{\bv
g_i\in\mathcal{G}_{r_i}(u_i,v_i)}\frac{T_i(\bv
g_i;J_{i-1}-i+1,J_i-J_{i-1})}{g_{i,1}!\cdots
g_{i,v_i}!}\ll_k\beta\prod_{i=1}^k\frac{v_i^{r_i}}{r_i!}\ee for
every choice of integers $0=J_0\le J_1\le\cdots\le J_k\le k$ with
$J_i\ge i$ for all $i\in\{1,\dots,k\}$. So fix such a $(k+1)$-tuple $(J_0,J_1,\dots,J_k)$ and set $$T_i=(\rho_{k-i+1}^{P-1})^{-(k-J_i)v_i}\sum_{\bv
g_i\in\mathcal{G}_{r_i}(u_i,v_i)}\frac{T_i(\bv
g_i;J_{i-1}-i+1,J_i-J_{i-1})}{g_{i,1}!\cdots g_{i,v_i}!}\quad(1\le
i\le k).$$ Also, let $$I=\min\{1\le i\le k:J_i=k\}$$ (note that $J_k=k$, so $I$ is well-defined). We claim that
\be\label{lb22e24}T_i\ll_{k,\epsilon}\beta_i\frac{v_i^{r_i}}{r_i!}\times
\begin{cases}(k-i+2)^{(P-1)(r_i-v_i)}&{\rm if}~1\le i<I,\cr
\displaystyle\max\left\{1,(k-I+2)^{(P-1)(r_I-v_I)},\prod_{j=1}^{I-1}(k-j+2)^{(P-1)(v_j-r_j)}\right\}&{\rm if}~i=I,\cr
1&{\rm if}~I<i\le k.
\end{cases}\ee
Note that if inequality~\eqref{lb22e24} is indeed true, then
\bes\prod_{i=1}^kT_i\ll_{k,\epsilon}\left(\prod_{i=1}^k\beta_i\frac{v_i^{r_i}}{r_i!}\right)\max_{m\in\{1,I,I+1\}}\prod_{j=1}^{m-1}(k-j+2)^{(P-1)(r_j-v_j)}
\ll_k\beta\prod_{i=1}^k\frac{v_i^{r_i}}{r_i!},\ees by relations~\eqref{inequ},~\eqref{ineqw} and~\eqref{beta2},
that is~\eqref{lb22e16} holds. So establishing~\eqref{lb22e24} will complete the proof of the lemma. 

Before embarking on the proof of~\eqref{lb22e24}, we introduce some notation and prove an intermediate result. For $i\in\{1,\dots,k\}$ define $J_i'\in\{J_{i-1},J_i\}$ by 
$$
(\rho_{k-i+1}^{P-1})^{-(J_i-J_i')v_i}(t_{i,J_i'-i+1})^{r_i}
	=\max_{j\in\{J_{i-1},J_i\}}(\rho_{k-i+1}^{P-1})^{-(J_i-j)v_i}(t_{i,j-i+1})^{r_i}.
$$ 
We claim that if $i\le J_i'\le k-1$, then 
\be\label{lb22e19}
	T_i\ll_{k,\epsilon}\beta_i\frac{v_i^{r_i}}{r_i!}(k-i+2)^{(P-1)(r_i-v_i)},
\ee
provided that $P-1$ is small enough. Indeed, Lemmas~\ref{lambda} and \ref{lb22l3}(b) give us that
\begin{align*}
T_i&\cdot(k-i+2)^{(P-1)(v_i-r_i)}\\
	&\ll_k\frac{e^{O_k((P-1)^2v_i)}v_i^{r_i}}{r_i!}\frac{(\rho_{k-i+1}^{P-1})^{-(k-J_i)v_i}
		(\rho_{k-i+1}^{P-1})^{-(J_i-J_i')v_i}(t_{i,J_i'-i+1})^{r_i}}{(k-i+2)^{(P-1)(r_i-v_i)}}\\
	&=\frac{e^{O_k((P-1)^2v_i )}v_i^{r_i}}{r_i!}(\rho_{k-i+1}^{P-1})^{(J_i'-i+1)v_i}	
		\left(\frac{t_{i,J_i'-i+1}}{(k-i+2)^{P-1}}\right)^{r_i}
\end{align*}
So
\be\label{lb22e20}\begin{split}
&\log\left(\frac{ T_i \cdot(k-i+2)^{(P-1)(v_i-r_i)}} {v_i^{r_i}/r_i!} \right) \\
&\qquad= (P-1)(J_i'-i+1) \left( \ell_i-\frac{F(k-i+1,J_i'-i+1)}{k-i+2}r_i+O_k((P-1)\ell_i+1)\right) \\
&\qquad = (P-1)(J_i'-i+1)\ell_i \left( 1-\frac{F(k-i+1,J_i'-i+1)}{(k-i+2)^{1-\alpha}}
			+  O_k\left(P-1+\ell_i^{-1/2}\right) \right) .
\end{split}\ee
For every $i\in\{1,\dots,k-1\}$, condition \eqref{e0}\;and Lemma
\ref{lb22l0a}\;imply that 
$$
\alpha\ge1+\epsilon-\frac1{\log(k-i+2)}\log\left(\frac{(k-i+2)\log(k-i+2)-2\log2}{k-i}\right)
$$
or, equivalently, that
$$
(k-i+2)^{\alpha-1}F(k-i+1,k-i)\ge(k-i+2)^\epsilon.
$$
So if $i\le J_i'\le k-1$, then
\be\label{lb22e22}(k-i+2)^{\alpha-1}F(k-i+1,J_i'-i+1)\ge(k-i+2)^{\alpha-1}F(k-i+1,k-i)\ge(k-i+2)^\epsilon,\ee
by Lemma \ref{lb22l0b}(a). Inserting the above inequality into~\eqref{lb22e20} proves \eqref{lb22e19}. 

We are now in position to show~\eqref{lb22e24}. First, if $I<i\le k$, then $J_i=J_{i-1}=k$. So 
$$
T_i(\bv g_i;J_{i-1}-i+1,J_i-J_{i-1})=T_i(\bv g_i;k-i+1,0)=1
$$ 
for
every $\bv g_i\in\mathcal{G}_i(r_i)$ and~\eqref{lb22e24} follows immediately by Lemma \ref{lb22l2}. Next, let $1\le i<I$. If $J_i'\ge i$, then
~\eqref{lb22e24} follows by~\eqref{lb22e19}, since we also have that $J_i'\le J_i\le
J_{I-1}\le k-1$. Assume now that $J_i'=i-1$, in which case
$J_{i-1}=i-1$. Then
\be\begin{split}\frac{F(k-i+1-(J_{i-1}-i+1),J_i-J_{i-1})}{(k-i+2)^{1-\alpha}}=\frac{F(k-i+1,J_i-i+1)}{(k-i+2)^{1-\alpha}}&\ge\frac{F(k-i+1,k-i)}{(k-i+2)^{1-\alpha}}\\
&\ge(k-i+2)^\epsilon,\nonumber\end{split}\ee by Lemma~\ref{lb22l0b}(a) and relation~\eqref{lb22e22}. The above inequality allows us to apply Lemma \ref{lb22l3}(a) with $\eta = (k-i+2)^\epsilon-1>0$. Therefore we deduce that
$$
T_i\ll_{k,\epsilon}\beta_i\frac{v_i^{r_i}}{r_i!}
(\rho_{k-i+1}^{P-1})^{-(k-J_i)v_i}(\rho_{k-i+1}^{P-1})^{-(J_i-J_i')v_i}(t_{i,J_i'-i+1})^{r_i}=\beta_i\frac{v_i^{r_i}}{r_i!}(k-i+2)^{(P-1)(r_i-v_i)},
$$ 
that is~\eqref{lb22e24} holds in this case too. Finally, we bound from above $T_I$. If $I\le J_I'\le k-1$ or $J_{I-1}=I-1$, then~\eqref{lb22e24} follows immediately by relation \eqref{lb22e19} or Lemma~\ref{lb22l4}, respectively. So suppose that $J_I'\in\{I-1,k\}$ and $J_{I-1}\ge I$, in which case we must have $J_I'=J_I=k$. We separate two cases. Set 
$$
\eta_1=\frac{F(k-I+1,k-I+1)-F(k-I,k-I)}{2(k-I+2)^{1-\alpha}}>0
$$ 
and assume first that 
$$
\frac{F(k-I+1,J_I'-I+1)}{(k-I+2)^{1-\alpha}}=\frac{F(k-I+1,k-I+1)}{(k-I+2)^{1-\alpha}}\ge1+\eta_1.
$$ 
Inserting the above inequality into \eqref{lb22e20} implies that 
$$
T_I\ll_k\beta_I\frac{v_I^{r_I}}{r_I!}(k-I+2)^{(P-1)(r_I-v_I)},
$$ 
provided that $P-1$ is small enough, thus proving~\eqref{lb22e24} in this case. Finally, assume that 
$$
\frac{F(k-I+1,k-I+1)}{(k-I+2)^{1-\alpha}}\le1+\eta_1.
$$ 
Then \bes\bsp\frac{F(k-I+1-(J_{I-1}-I+1),J_I-J_{I-1})}{(k-I+2)^{1-\alpha}}=\frac{F(k-J_{I-1},k-J_{I-1})}{(k-I+2)^{1-\alpha}}
&\le\frac{F(k-I,k-I)}{(k-I+2)^{1-\alpha}}\\
&\le1-\eta_1\end{split}\ees which, together with Lemma~\ref{lb22l3}(a), shows that 
$$
T_I\ll_{k,\epsilon}\beta_I\frac{v_I^{r_I}}{r_I!} (\rho_{k-I+1}^{P-1})^{-(k-J_I)v_I}
	(\rho_{k-I+1}^{P-1})^{-(J_I-J_I')v_I}(t_{I,J_I'-I+1})^{r_I}
	=\beta_I\frac{v_I^{r_I}}{r_I!},
$$ 
thus proving~\eqref{lb22e24} in this last case too. This completes the proof of the lemma.
\end{proof}

\bigskip\bigskip

\bibliographystyle{alpha}

\end{document}